\documentclass[11pt,reqno]{amsart}
\usepackage{amsfonts,amssymb,amsmath, amsthm}
\usepackage{graphicx}
\usepackage{caption, subcaption} 
\usepackage{systeme}
\usepackage{pgf,tikz,pgfplots}
\usepackage{kotex}
\usepackage[margin=1in]{geometry}

\usepackage{comment}
\usepackage{hyperref}
\usepackage{array}

\pgfplotsset{compat=1.15}
\usepgfplotslibrary{fillbetween}
\usepackage{mathrsfs}
\usetikzlibrary{arrows}
\usetikzlibrary{calc}
\newtheorem{theorem}{Theorem}[section]
\newtheorem{lemma}{Lemma}[section]

\newtheorem{proposition}{Proposition}[section]
\newtheorem{remark}{Remark}[section]

\theoremstyle{definition}

\theoremstyle{remark}

\numberwithin{equation}{section}

\newcommand{\rd}{\partial}

\newcommand{\R}{\mathbb{R}}

\newcommand{\sX}{\dot{\mathbf{X}}}

\newcommand{\bX}{\mathbf{X}}

\def \rd {\partial}
\def \d {\delta}
\def \ubar {\overline{u}}

\def \r {\rho}
\def \th {\theta}
\def \b {\beta}

\def \Rp {\mathbb{R}_+}
\def \intRp {\int_{\mathbb{R}_+}}

\def \rhotilX {\bar{\rho}^{\mathbf{X}, \beta}}

\def \Util {\bar{U}}

\def \util {\bar{u}}
\def \rhotil {\bar{\rho}}
\def \thtil {\bar{\theta}}

\def \RR {\mathbb{R}}

\newcommand{\bbr}{\mathbb R}

\newcommand{\e}{\varepsilon}

\newcommand{\norm}[1]{\left\lVert#1\right\rVert}

\usepackage{float}
\restylefloat{table}
\restylefloat{figure}
\usepackage{dsfont}

\title[Shock stability for the Navier-Stokes-Fourier system in the half space]{Stability of viscous shock for the Navier-Stokes-Fourier system: outflow and impermeable wall problems}
\author[Huang]{Xushan Huang}
\address[Xushan Huang]{\newline Department of Mathematical Sciences \newline Korea Advanced Institute of Science and Technology, Daejeon  34141, Republic of Korea}
\email{xushanhuang@kaist.ac.kr}

\author[Lee]{Hobin Lee}
\address[Hobin Lee]{\newline Department of Mathematical Sciences \newline Korea Advanced Institute of Science and Technology, Daejeon  34141, Republic of Korea}
\email{lcuh11@kaist.ac.kr}

\author[Oh]{HyeonSeop Oh}
\address[HyeonSeop Oh]{\newline Department of Mathematical Sciences \newline Korea Advanced Institute of Science and Technology, Daejeon  34141, Republic of Korea}
\email{ohs2509@kaist.ac.kr}



\thanks{\textbf{Acknowledgment.} X. Huang was supported by the National Research Foundation of Korea (NRF) grant funded by the Korea government (MSIT) (No. RS-2019-NR040050). H. Lee and H. Oh were partially supported by the National Research Foundation of Korea (RS-2024-00361663 and NRF-2019R1A5A1028324). The authors thank Professor Moon-Jin Kang for valuable comments.
}

\begin{document}

\begin{abstract} 
We investigate the time-asymptotic stability of solutions to the one-dimensional Navier\allowbreak-Stokes\allowbreak-Fourier system in the half-space, focusing on the outflow and impermeable wall problems. When the prescribed boundary and far-field conditions form an outgoing viscous shock, we prove that the solution converges to the viscous shock profile, up to a dynamical shift, provided that the initial perturbation and the shock amplitude are sufficiently small. In order to obtain our results, we employ the method of $a$-contraction with shifts. Although the impermeable wall problem is technically simpler to analyze in Lagrangian mass coordinates, the outflow problem leads to a free boundary in that framework. Therefore, we use Eulerian coordinates to provide a unified approach to both problems. This is the first result on the time-asymptotic stability of viscous shocks for initial-boundary value problems of the Navier-Stokes-Fourier system for the outflow and impermeable wall cases.
\end{abstract}

\maketitle
	\tableofcontents
\section{Introduction}
\setcounter{equation}{0}
We consider the initial-boundary value problem (IBVP) for the one-dimensional Navier-Stokes-Fourier (NSF) system in Eulerian coordinates on the half-space $\mathbb{R}_+ := (0,\infty)$: 
\begin{equation} \label{eq:NSF}
	\begin{cases}
		& \rho_t + (\rho u)_x = 0, \quad x\geq 0,\quad t\geq 0,\\
		&(\rho u)_t + (\rho u^2 + p)_x = \mu u_{xx}, \\
		&E_t + \left(uE+ pu\right)_x = \kappa \theta_{xx} + \mu(uu_x)_x,
	\end{cases}
\end{equation}
where $\rho =\rho(t,x), u= u(t,x)$, and $ \theta = \theta(t,x) $ represent the fluid density, velocity, and absolute temperature, respectively. In addition, $E=\rho (e+ \frac{u^2}{2})$ is the total energy function. For an ideal polytropic gas, the pressure function $p$ and the internal energy function $e$ are given by 
$$p(\rho, \theta) = R\rho \theta,\qquad e(\rho, \theta) = \frac{R}{\gamma -1} \theta + const,$$
with $R>0$, $\gamma>1$ being both constants related to the fluid, while $\mu>0$ and $\kappa>0$ denote the viscosity and the heat-conductivity.\\

We impose the initial data $(\rho_0(x), u_0(x), \theta_0(x))$ for \eqref{eq:NSF} satisfying
\begin{equation}\label{eq:initial}
		\inf_{x\in \mathbb{R}_+} \rho_0(x) > 0, \quad \inf_{x\in \mathbb{R}_+} \theta_0(x) > 0, \quad (\rho_0(x), u_0(x), \theta_0(x) ) \to(\rho_+,u_+, \theta_+) \, \,  \text{as }  x \to \infty,
\end{equation}
where $\rho_+>0, u_+$, and $\theta_+>0$ are prescribed constants.\\

According to the sign of the velocity $u_-$ on the boundary, there are three types of boundary conditions for the IBVP of the NSF system, as follows. \\

\noindent Case 1. Outflow problem (negative velocity on the boundary)
\begin{equation}\label{outflow}
	u(t,0) = u_- <0,  \quad \theta(t,0) = \theta_- >0, \quad t>0;
\end{equation}
Case 2. Impermeable wall problem (zero velocity on the boundary)
\begin{equation}\label{imperable}
	u(t,0) = u_- =0, \quad \theta(t,0) = \theta_- >0, \quad t>0;
\end{equation}
Case 3.  Inflow problem (positive velocity on the boundary)
\begin{equation*} 
	\rho (t,0) = \rho_- >0,  \quad u(t,0) = u_- >0,   \quad \theta(t,0) = \theta_- >0, \quad t>0.
\end{equation*}

In this paper, we focus on Cases 1 and 2: the outflow problem \eqref{outflow} and the impermeable wall problem \eqref{imperable}. Accordingly, we assume that the initial data \eqref{eq:initial} satisfy the compatibility condition corresponding to either \eqref{outflow} or \eqref{imperable}. 

It is worth noting that in Case 3, the boundary density $\rho_-$ must be prescribed to ensure the well-posedness of the system \eqref{eq:NSF}.\\

The goal of this paper is to study the large-time behavior of solutions in Cases 1 and 2. It is well known that the large-time behavior of solutions to \eqref{eq:NSF} is closely related to the Riemann problem for the (inviscid) full compressible Euler system (\eqref{eq:NSF} with $\mu = \kappa = 0$). To capture the characteristic structure of the Euler system, we split the state space $\Omega:=\{(\rho,u,\theta):\rho>0,\theta>0\}$ into 6 regions based on the sign of the velocity $u$ and the sign of eigenvalues $\lambda_i=\lambda_i(\rho,u,\theta)$, $i=1,2,3$ of the Euler system (see the figure below):
\begin{figure}[h]
	\centering
    \begin{minipage}{0.50\textwidth}
        \begin{align*}
				&\Omega_{super}^+ := \{(\rho,u,\th)\in \Omega: 0<\lambda_1 < \lambda_2 <\lambda_3\}, \\[1em]
            &\Gamma_{trans}^+ := \{(\rho,u,\th)\in \Omega: 0=\lambda_1 < \lambda_2 <\lambda_3\}, \\[1em]
            &\Omega_{sub}^+ := \{(\rho,u,\th)\in \Omega: u>0, \, \,  \lambda_1 < 0 < \lambda_3\},  \\[1em]
            &\Omega_{sub}^- := \{(\rho,u,\th)\in \Omega: u<0, \, \,  \lambda_1 < 0 < \lambda_3\}, \\[1em]
            &\Gamma_{trans}^- := \{(\rho,u,\th)\in \Omega: \lambda_1 < \lambda_2 < 0 = \lambda_3\}, \\[1em]
            &\Omega_{super}^- := \{(\rho,u,\th)\in \Omega: \lambda_1 < \lambda_2 < \lambda_3 < 0\}.
        \end{align*}
    \end{minipage}%
    \hspace{3em} 
    \begin{minipage}{0.30\textwidth} 
        \centering
        \includegraphics[width=\textwidth]{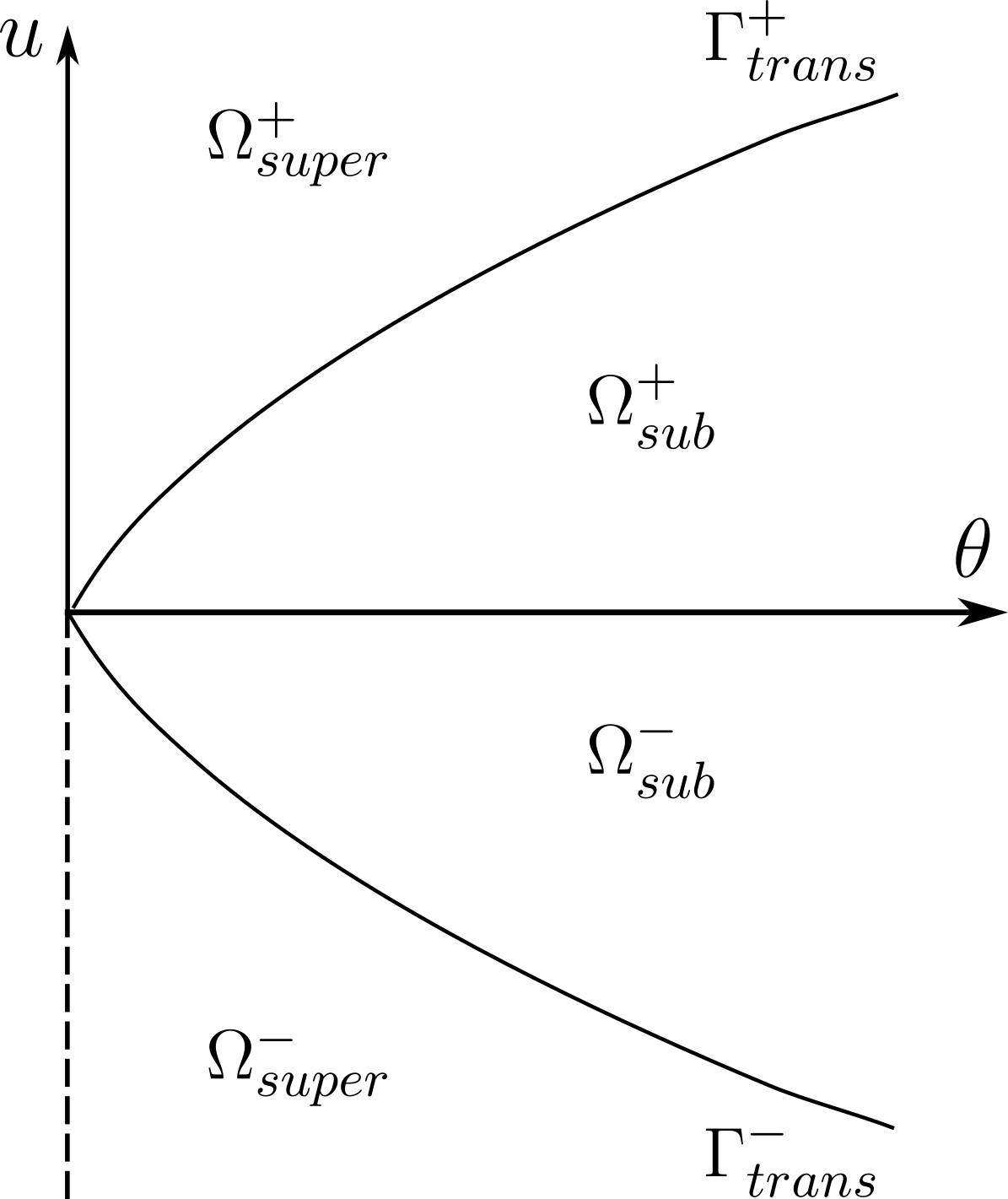}
    \end{minipage}
\end{figure}

Since the eigenvalues of the Euler system are given by
\[
\lambda_1(\rho,u,\theta) = u - c, \quad \lambda_2(\rho,u,\theta) = u, \quad \lambda_3(\rho,u,\theta) = u + c,
\]
where $c=\sqrt{\gamma R \theta}$ is the sound speed, the above regions can be equivalently written as follows:
\begin{align*}
	\begin{split}
	\begin{aligned}
	&(\rho,u,\th) \in \Omega_{super}^+ && \iff \quad u>c, &&(\rho,u,\th) \in \Gamma_{trans}^+ && \iff \quad u=c,\\
	&(\rho,u,\th) \in \Omega_{sub}^+ && \iff  \quad 0<u<c, &&(\rho,u,\th) \in \Omega_{sub}^- && \iff  \quad -c<u<0,\\
	&(\rho,u,\th) \in \Gamma_{trans}^- &&\iff  \quad u=-c, &&(\rho,u,\th) \in \Omega_{super}^- && \iff  \quad u<-c.	
	\end{aligned}
	\end{split} 
\end{align*}
$\bullet$ {\bf Viscous shock waves as asymptotic profiles.} We are interested in cases where viscous shock waves appear as asymptotic profiles. The viscous shock wave connecting the two end states $(\rho_-, u_-, E_-)$ and $(\rho_+, u_+, E_+)$ is a traveling wave solution on $\mathbb{R}$ that satisfies the Rankine-Hugoniot condition and the Lax entropy conditions:
\begin{equation}\label{eq:RH}
	\begin{aligned}
		&\exists \, \sigma \, \, \text{s.t. } 	
		\begin{cases}
			& -\sigma (\rho_+- \rho_- ) + (\rho_+u_+ -\rho_- u_-) =0,\\
			&-\sigma(\rho_+u_+ - \rho_- u_-) + (\rho_+ u^2_+ + p_+ - \rho_- u_-^2 - p_-)=0, \\
			&-\sigma\left( E_+- E_-\right) + \left( u_+E_+ + p_+u_+ -u_-E_- - p_-u_-  \right)= 0,
		\end{cases}\\
		&\text{and either } \rho_- < \rho_+, u_- > u_+ \, \, \text{and } \theta_-<\theta_+ \, \, \text{or } \rho_- > \rho_+, u_- > u_+ \, \, \text{and } \theta_- > \theta_+ \, \, \text{holds}.
	\end{aligned}
\end{equation}
In other words, for any two constant states $(\rho_\pm, u_\pm, E_\pm)$ satisfying \eqref{eq:RH}, there exists a viscous shock wave $(\rhotil, \util, \overline{E})(\xi):=(\rhotil, \util, \overline{E})(x-\sigma t)$ given by the solution to the following ODEs:
\begin{equation}\label{eq:Vshock}
    \begin{cases}
    		& -\sigma\bar{\rho}'+(\bar{\rho}\bar{u})'=0, \qquad ' = \frac{d}{d\xi}, \\
    		& -\sigma (\bar{\rho} \bar{u})' + (\bar{\rho} \bar{u}^2 + \bar{p})' = \mu \bar{u}''  ,\\
    		 & -\sigma \bar{E}' + (\bar{E}\bar{u} + \bar{p}\bar{u})' = \kappa \bar{\theta}'' + \mu (\bar{u}\bar{u}')', \\
    		&  (\bar{\rho}, \bar{u}, \overline{E})(-\infty) =  (\rho_-,u_-,E_-), \quad (\bar{\rho}, \bar{u}, \overline{E})(+\infty) = (\rho_+,u_+,E_+),
    \end{cases}
\end{equation}
where $\bar{E} := \bar{\rho} (\bar{e}+ \frac{\bar{u}^2}{2})$, and $\bar{p}:=p(\bar{\rho}, \bar{\theta})$. Here, if $\rho_- < \rho_+$, the solution of \eqref{eq:Vshock} is a 1-shock wave with $\sigma = u_+ - \sqrt{\frac{\rho_-}{\rho_+}\left(\frac{p_+ - p_-}{\rho_- - \rho_+}\right)}$. If $\rho_- > \rho_+$, the solution of \eqref{eq:Vshock} is a 3-shock wave with $\sigma = u_+ + \sqrt{\frac{\rho_-}{\rho_+}\left(\frac{p_+ - p_-}{\rho_- - \rho_+}\right)}$. It is well-known that the system \eqref{eq:NSF} admits viscous $i$-shock wave $(i=1, 3)$, which are smooth and unique up to translation.\\

In what follows, we investigate the necessary conditions for the asymptotic profile to be an outgoing viscous shock for each of the IBVPs \eqref{eq:NSF}-\eqref{outflow} and \eqref{eq:NSF}-\eqref{imperable}. By the Lax entropy condition in \eqref{eq:RH}, we restrict our analysis to this case $u_->u_+$.

First, for the outflow problem, we consider the case where the far-field state $(\rho_+,u_+,\theta_+)$ lies in either the subsonic region $\Omega^{-}_{sub}$ or the transonic region $\Gamma^{-}_{trans}$. If $0>u_->u_+$ and the boundary value $(u_-,\theta_-)$ belongs to the curve $S^{P}_{3}(\rho_+,u_+,\theta_+)$, which is a curve projected by the $3$-shock curve to the $(u,\theta)$-plane, then there exists a unique $\rho_-$ such that $(\rho_-,u_-,\theta_-)$ lies on the 3-shock curve starting from $(\rho_+,u_+,\theta_+)$. Moreover, the Lax entropy condition $\lambda_3(\rho_+,u_+,\theta_+) < \sigma\left((\rho_-,u_-,\theta_-),(\rho_+,u_+,\theta_+)\right)< \lambda_3(\rho_-,u_-,\theta_-)$ implies that the shock speed $\sigma$ is positive. Consequently, under these conditions, the viscous 3-shock wave is compatible with the prescribed boundary values. Accordingly, the solution $(\rho,u,\theta)(t,x)$ to \eqref{eq:NSF}-\eqref{eq:initial}-\eqref{outflow} is expected to converge to the viscous shock profile $(\bar{\rho}(\xi),\bar{u}(\xi),\bar{\theta}(\xi))$ up to a shift. (See the figure below.)

\begin{figure}[h]
    \centering
    \includegraphics[width=0.55\textwidth]{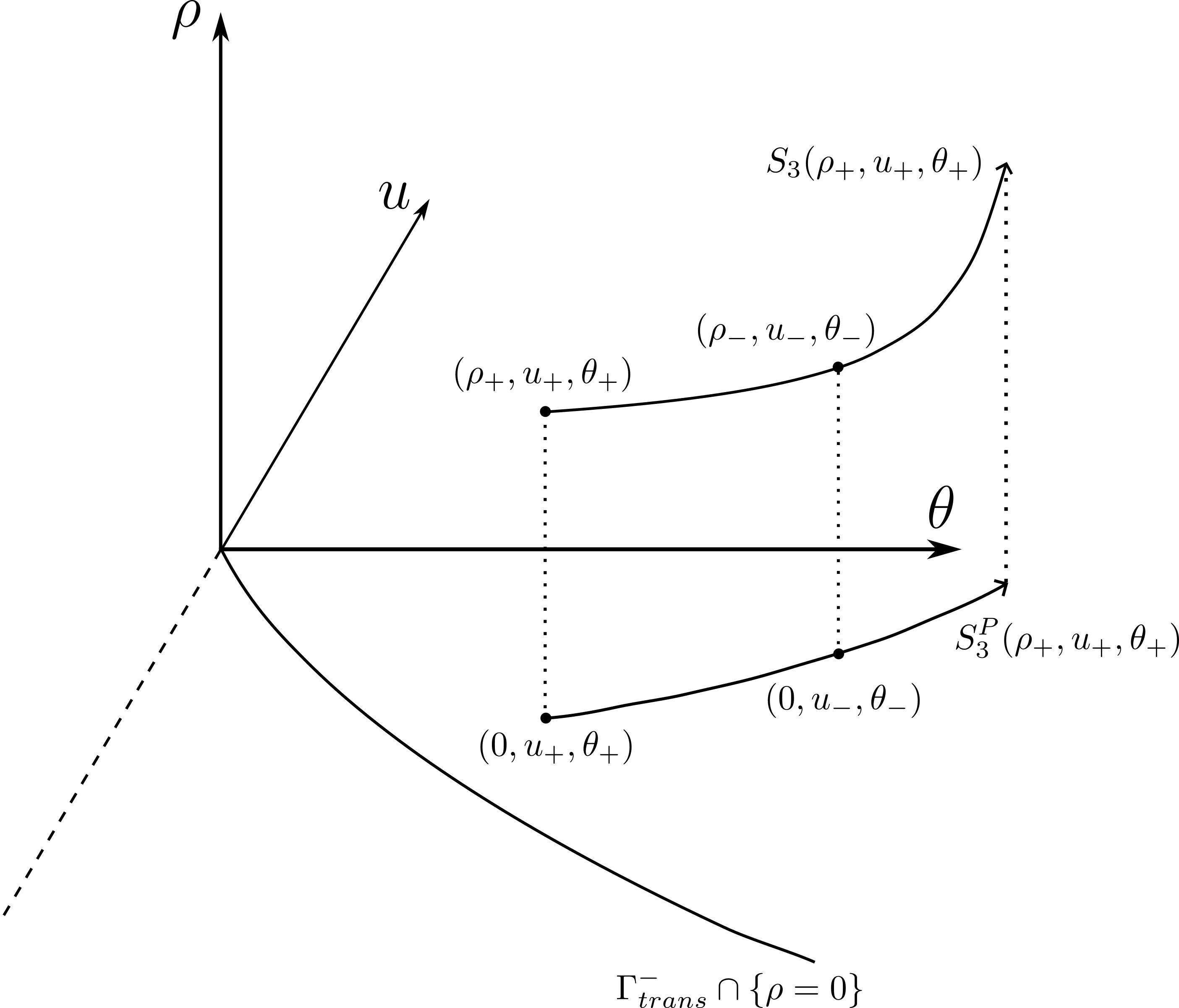}
\end{figure}

\begin{remark}
    When $(\r_+, u_+, \th_+) \in \Omega_{super}^-$, a single viscous 3-shock wave may appear as an asymptotic profile. This situation can occur when $(u_-,\th_-)$ lies in the subsonic region, in which case the shock strength is generally not small. Therefore, this case cannot be treated by our approach, which is applicable to the stability of weak viscous shock waves. We leave this case for future work.
\end{remark}

Now, we consider the impermeable wall problem \eqref{eq:NSF}-\eqref{eq:initial}-\eqref{imperable}. If $u_+<0$ and $(u_-,\theta_-) \in S_3^P(\rho_+, u_+, \theta_+)$, then we can uniquely determine the value $\rho_-$ satisfying $(\rho_-, u_-, \theta_-) \in S_3(\rho_+, u_+, \theta_+)$. In this case, we can expect that the solution $(\rho, u, \theta)$ to the impermeable wall problem \eqref{eq:NSF}-\eqref{eq:initial}-\eqref{imperable} converges to the viscous 3-shock wave. Here, from the Rankine-Hugoniot condition \eqref{eq:RH}$_1$, it follows that the $3$-shock wave is outgoing.

\begin{remark}
    In \cite{MA99}, all possible asymptotic patterns were classified for the outflow and impermeable wall problems of the isentropic Navier-Stokes (NS) system. However, for the NSF system, to the best of our knowledge, the classification of asymptotic profiles remains open. In particular, for the outflow problem, large boundary layer solutions have not yet been extensively studied. For the impermeable wall problem, apart from the case considered in this paper, even the possible asymptotic patterns for more general configurations have not yet been classified.
\end{remark}

\subsection{Literature review}
There has been plenty of literature on the IBVP for (viscous) conservation laws. First, we refer to \cite{AMA97, AC97, AMS24, BRN79,DL88, FS01, KSX97} for studies on the inviscid conservation laws. Now, we focus on the well-posedness of viscous conservation laws. For the Cauchy problem, it is well known that solutions to viscous conservation laws converge to elementary waves, such as viscous shock waves, rarefaction waves, viscous contact waves, or their compositions (see, for example, \cite{MatRP}). Unlike the whole-space problem, the IBVP for viscous conservation laws may exhibit not only basic wave patterns but also a stationary solution, called a boundary layer solution, which arises due to boundary effects.

We review further results on the stability of elementary waves and boundary layer solutions for the IBVP of viscous conservation laws, with particular emphasis on the NS and NSF systems. First, using the Evans function analysis,  Costanzino-Humpherys-Nguyen-Zumbrun \cite{CHNZ09} proved the spectral stability of boundary layer solutions in both the outflow and inflow cases of the NS system. Nguyen-Zumbrun \cite{NZ09} further established the nonlinear stability of boundary layer solutions for general hyperbolic-parabolic systems, including the NS system, by deriving pointwise Green function bounds. We also refer to the work of Serre and Zumbrun \cite{SZ01}, which presents stability and instability results for boundary layer solutions in the inflow problem for the NSF system.

In contrast to the spectral approach, the energy method has also been employed to study the NS and NSF systems in the presence of boundaries. First, for the impermeable wall problem of the NS system, the asymptotic profile of the solutions can be classified into two cases: (outgoing) viscous shocks and rarefaction waves. These cases were studied in \cite{MM99} and \cite{MN00}, respectively. In the outflow setting, the stability of boundary layer solutions and the superposition of boundary layer solutions and rarefaction waves were investigated in \cite{HQ09, KNZ03, KZ08}. As for the inflow problem, we refer to \cite{MN01} for the study of the stability of boundary layer solutions and rarefaction waves, and their superpositions.

We next turn to the NSF system. For the impermeable wall problem, Huang-Li-Shi \cite{HLS10} studied the large-time behavior of the rarefaction waves. In the outflow setting, Kawashima-Nakamura-Nishibata-Zhu \cite{KNNZ10} investigated the well-posedness of boundary layer solutions. We refer to \cite{Q11, WWZ16, WWzou16} for further studies on the stability of boundary layer solutions in this context. For the inflow problem, Qin-Wang \cite{QW09, QW11} proved the stability of the superposition of a subsonic or transonic boundary layer solution, rarefaction, and viscous contact wave.

We now turn to the stability of shock waves. One classical and powerful approach for proving shock stability is the anti-derivative method. For the NS system, Matsumura-Mei \cite{MM99} established the stability of a single shock for the impermeable wall problem by introducing a phase shift that enables the use of the anti-derivative method. Later, Huang-Matsumura-Shi \cite{HMS03} proved the stability of the superpositions of a boundary layer solution and a viscous shock through a suitable change of variables, which allows one to control boundary effects associated with the anti-derivative of the fluid velocity $u$. However, to the best of our knowledge, there have been no results on the stability of viscous shocks for the IBVP of the NSF system.

In this work, we utilize the method of $a$-contraction with shifts, invented by Kang-Vasseur \cite{KV21, KV-Inven}, which provides a contraction estimate for shock waves at the $L^2$ level. Based on this method, the $L^2$-perturbation localized by the shock, or composite waves involving a shock, can be controlled (see, for example, \cite{GKV24, EoEunKang24, HKK23, HL25,  MY25, KKS25, KO25, KVW23, KVW-NSF}). With this method, Huang-Kang-Kim-Lee \cite{HKKL24} proved the stability of a single shock for the inflow and impermeable wall problems of the NS system, and Kang-Oh-Wang \cite{KOW25} established the corresponding result for the outflow problem. We also mention the work of Han-Kang-Kim-Kim-Oh \cite{HKKKO25}, which studied the stability of the superposition of a (degenerate) boundary layer solution, a rarefaction wave, and a shock wave for the inflow problem of the NS system.

In this paper, we use the method of $a$-contraction with shifts to establish the stability of a single (outgoing) shock for the boundary value problems \eqref{outflow} and \eqref{imperable} for the NSF system \eqref{eq:NSF}.

\subsection{Main results}
We now state the main result on global existence and large-time behavior for solutions to the outflow and impermeable wall problems. Without loss of generality, we may assume that the viscous shock $(\rhotil, \util, \thtil)(x-\sigma t)$ connecting $(\rho_-, u_-, \th_-)$ to $(\rho_+, u_+, \th_+)$ satisfies $\rhotil(0) = \frac{\rho_- + \rho_+}{2}$.\\

For simplicity, we rewrite the system \eqref{eq:NSF} in the following form,
which is not in conservation form but is equivalent to the original system:
\begin{equation} \label{eq:NSFS}
	\begin{cases}
		& \rho_t + (\rho u)_x = 0,\quad x\ge0,\quad t\ge0,\\
		&u_t + uu_x + \frac{p_x}{\rho}= \mu\frac{u_{xx}}{\rho}, \\
		&\frac{R}{\gamma-1}\theta_t+\frac{R}{\gamma-1}u\theta_x+\frac{p}{\rho}u_x=\frac{\kappa}{\rho}\theta_{xx}+\frac{\mu}{\rho}u_x^2.
	\end{cases}
\end{equation}
 \begin{theorem}[Outflow problem] \label{thm:outflow}
	For a given constant state $(\rho_+,u_+, \theta_+) \in\Omega^{-}_{sub}\cup \Gamma_{trans}^{-}$, there exist positive constants $\delta_0,\e_0 >0$ small enough such that the following holds. \\
	For any $(u_-,\theta_-)$ satisfying $0>u_->u_+$ and $(u_-,\theta_-) \in S_3^{P}(\rho_+,u_+,\theta_+)$ with
	\[
	|u_+-u_-| + |\theta_+ - \theta_-|<\delta_0,
	\]
	let $\r_->0$ be the (unique) constant state such that $(\rho_-,u_-,\theta_-) \in S_3(\rho_+,u_+,\theta_+)$. Denote $(\bar{\rho},\bar{u},\bar{\theta})(x-\sigma t)$ the viscous 3-shock wave of \eqref{eq:Vshock} with end states $(\r_-, u_-, \th_-)$ and $(\r_+, u_+, \th_+)$, where the shock speed $\sigma>0$ is given in \eqref{eq:RH}. Then there exists $\beta>0$ large enough (depending only on the shock strength $\delta_0$) such that the following holds.  
	Let $(\rho_0,u_0, \theta_0)$ be any initial data satisfying 
	   \begin{equation*} 
		\begin{aligned}
		&\norm{(\r_0, u_0,\th_0) - (\r_+, u_+,\th_+)}_{L^2(\b, \infty)} + \norm{(\r_0, u_0,\th_0) - (\r_-, u_-,\th_-)}_{L^2(0, \b)}\\
		&\phantom{\norm{(\r_0, u_0,\th_0) - (\r_+, u_+,\th_+)}_{L^2(\b, \infty)} + (\r_0, u_0,\th_0)}  + \norm{(\rd_x \r_0, \rd_x u_0, \rd_x \th_0)}_{L^2(\Rp)} < \e_0.
		\end{aligned}
	   \end{equation*}
	   Then the outflow problem \eqref{eq:NSF}--\eqref{eq:initial}--\eqref{outflow} admits a unique global-in-time solution $(\rho,u, \theta)(t,x)$ as follows: there exists a Lipschitz shift  $t\mapsto \bX(t)$ such that
	   \begin{align*}
		   &(\rho,u, \theta)(t,x)-(\bar{\rho}, \bar{u}, \bar{\theta})(x-\sigma t -\bX(t)-\beta)\in C(0,\infty;H^1(\R_+)),\\
		   &(u_{xx},\theta_{xx})(t,x)-(\bar{u}_{xx}, \bar{\theta}_{xx})(x-\sigma t-\bX(t)-\beta)\in L^2(0,\infty;L^2(\R_+)).
	   \end{align*}
	   Moreover, the solution asymptotically converges to the (shifted) viscous shock:
	   \begin{equation}\label{long_time_outflow}
		   \lim_{t\to\infty}\sup_{x\in\R_+}\left|(\rho,u, \theta)(t,x)-(\bar{\rho},\bar{u}, \bar{\theta})(x-\sigma t-\bX(t)-\beta)\right|=0
	   \end{equation}
	   as $t \to \infty$, and 
	   \[
	   \lim\limits_{t\to \infty}|\dot{\bX}(t)|=0.
	   \]
   \end{theorem}
Next, we state the analogous result for the impermeable wall problem.
   \begin{theorem}[Impermeable wall problem] \label{thm:impermeable}
	For given $(\rho_+, \theta_+)$, there exist positive constants $\delta_0,\e_0 >0$ small enough such that the following holds. \\
	For any constant states $u_+<0$ satisfying $$- u_+ = |u_+| < \delta_0,$$ let $\r_->0$ and $\theta_->0$ be the (unique) constant states such that $(\rho_-,u_-,\theta_-):= (\rho_-,0,\theta_-)\in S_3(\rho_+,u_+,\theta_+)$. Denote  $(\bar{\rho} (x-\sigma t),\bar{u}(x-\sigma t), \bar{\theta}(x-\sigma t))$ the viscous 3-shock wave of \eqref{eq:Vshock} with end states $(\r_-, u_-, \th_-)$ and $(\r_+, u_+, \th_+)$, where the shock speed $\sigma>0$ is given in \eqref{eq:RH}. Then there exists $\beta>0$ large enough (depending only on the shock strength $\delta_0$) such that the following holds. 
	Let $(\rho_0,u_0, \theta_0)$ be any initial data satisfying 
	   \begin{equation*} 
		\begin{aligned}
		&\norm{(\r_0, u_0,\th_0) - (\r_+, u_+,\th_+)}_{L^2(\b, \infty)} + \norm{(\r_0, u_0,\th_0) - (\r_-, u_-,\th_-)}_{L^2(0, \b)}\\
		&\phantom{\norm{(\r_0, u_0,\th_0) - (\r_+, u_+,\th_+)}_{L^2(\b, \infty)} + (\r_0, u_0,\th_0)}  + \norm{(\rd_x \r_0, \rd_x u_0, \rd_x \th_0)}_{L^2(\Rp)} < \e_0.
		\end{aligned}
	   \end{equation*}
	   Then the impermeable wall problem \eqref{eq:NSF}--\eqref{eq:initial}--\eqref{imperable} admits an unique global-in-time solution $(\rho,u, \theta)(t,x)$ as follows: there exists a Lipschitz shift  $t\mapsto \bX(t)$ such that
	   \begin{align*}
		   &(\rho,u, \theta)(t,x)-(\bar{\rho}, \bar{u}, \bar{\theta})(x-\sigma t -\bX(t)-\beta)\in C(0,\infty;H^1(\R_+)),\\
		   &(u_{xx},\theta_{xx})(t,x)-(\bar{u}_{xx}, \bar{\theta}_{xx})(x-\sigma t-\bX(t)-\beta)\in L^2(0,\infty;L^2(\R_+)).
	   \end{align*}
	   Moreover, the solution asymptotically converges to the (shifted) viscous shock:
	   \begin{equation}\label{long_time_impermeable}
		   \lim_{t\to\infty}\sup_{x\in\R_+}\left|(\rho,u, \theta)(t,x)-(\bar{\rho},\bar{u}, \bar{\theta})(x-\sigma t-\bX(t)-\beta)\right|=0
	   \end{equation}
	   as $t \to \infty$, and 
	   \[
	   \lim\limits_{t\to \infty}|\dot{\bX}(t)|=0.
	   \]
   \end{theorem}

\begin{remark} 
	\begin{enumerate}
		\item In Theorems \ref{thm:outflow} and \ref{thm:impermeable}, we prove the time-asymptotic stability of the viscous shock under small initial perturbations in the $H^1$ norm. Here, the two parameters $\delta_0$ and $\e_0$ are independent, representing the smallness of shock strength and initial perturbation respectively.
		\item We fix the position of the shock, as $\rhotil(\beta) = (\r_- + \r_+)/2$. Here, the large constant $\beta$ depends only on the shock strength $\delta_0$ (but is independent of $\e_0$). This enables us to control the boundary effect by ensuring that the discrepancy between the solution to \eqref{eq:NSF} and viscous shock at the boundary remains small.
		\item Our result implies that the solution to \eqref{eq:NSF} time-asymptotically converges to the viscous shock $(\bar{\rho},\bar{u},\bar{\theta})(x-\sigma t-\beta)$ as $t \to \infty$, up to a dynamical shift $\mathbf{X}(t)$, where the asymptotic profile is initially located far away from the boundary. Indeed, in the above results, the decay estimate $\lim\limits_{t\to\infty}|\dot{\bX}(t)|=0$ implies that the shift function $\bX(t) = o(t)$ as $t\to\infty$. Thus, the shifted wave $(\bar{\rho},\bar{u},\bar{\theta})(x-\sigma t-\bX(t)-\beta)$ tends to the original wave $(\bar{\rho},\bar{u},\bar{\theta})(x-\sigma t)$ as $t \to \infty$.
	\end{enumerate}
\end{remark}

\subsection{Main ideas for the proof}
First, we consider our problem in Eulerian coordinates, as the outflow problem \eqref{outflow} in Lagrangian mass coordinates leads to a free boundary value problem. To handle two boundary conditions in a unified way, the Eulerian framework is more suitable than the Lagrangian one in our setting, even though it is technically more complicated. 

Second, as mentioned above, when proving the results in Theorems \ref{thm:outflow} and \ref{thm:impermeable}, we use the method of $a$-contraction with shifts, which provides a way to control the $L^2$-perturbation of viscous shock waves. In this paper, we aim to employ the method of $a$-contraction with shifts to the IBVP for the NSF system.

However, applying this method in our setting involves several major difficulties. First, when introducing the change of variable $x \mapsto y=y(x)$ in \eqref{ydef} and applying the Poincar\'e-type inequality, it is necessary to estimate the Jacobian $\frac{dy}{dx}$, which arises from the localization of diffusion terms by shock. In contrast to the approach in \cite{KVW-NSF}, where the Lagrangian mass coordinate was used, our analysis is based on the Eulerian coordinate, which makes the control of the Jacobian $\frac{dy}{dx}$ technically more delicate. To address this difficulty and simplify the computations, we estimate the main term $(p-p_\pm)/(u-u_\pm)$ appearing in the expression for $\frac{dy}{dx}$ (see \eqref{firstexj}) in a different manner. Instead of relying on the Taylor expansion argument used in \cite[Appendix B]{KVW-NSF}, we directly make use of the structural properties of viscous shock waves \eqref{shock_prop}$_3$ and \eqref{estderi} in Lemma~\ref{lem:viscous_shock}, together with the viscous shock equations \eqref{eq:Vshock}. These yield explicit relations between $p-p_\pm$ and $u-u_\pm$, and consequently we derive Lemma~\ref{lem:Jac}.

Next, for the boundary effects that arise from integration by parts, we need to control the discrepancy between the boundary values given in \eqref{outflow} (or \eqref{imperable}) and the values of viscous shocks at $x=0$. For this purpose, we choose the parameter $\beta>0$ large enough so that the perturbation at the boundary remains small. This smallness is also crucial for applying the Poincar\'e-type inequality in Lemma \ref{Poincare} (see Section \ref{sec:lot} for details).

Moreover, an additional difficulty arises from controlling the boundary effect associated with the density $\rho$, since there is no specified boundary condition for the density $\rho$. However, the outflow boundary condition \eqref{outflow} yields a good boundary term related to the density, which enables us to control the boundary effect (see Lemmas \ref{lem:main} and \ref{lem:rhox}). For the impermeable wall problem, the boundary effects involving $\rho$ and $\rho_x$ vanish due to the condition $u_-=0$.

\subsection{Organization of paper}
The paper is organized as follows. In Section 2, we present the properties of weak viscous shocks and the Poincaré-type inequality on any compact interval. Section 3 provides an \textit{a priori} estimate. In Section 4, we establish the $L^2$-estimate for the two boundary value problems, and in Section 5, we provide higher-order estimates.

	\section{Preliminaries}
\setcounter{equation}{0}

We first present several useful properties of the weak viscous shock waves. Subsequently, we provide the Poincar\'e-type inequality on any compact interval.
\subsection{Viscous shock waves}
We turn to the viscous 3-shock wave connecting $(\rho_-, u_-, \theta_-)$ and $(\rho_+,u_+,\theta_+)$ such that $(\rho_-, u_-, \theta_-) \in S_3(\rho_+,u_+,\theta_+)$. Recall from \eqref{eq:RH} that the shock speed $\sigma$ is explicitly given by
\begin{align*} 
	\sigma= u_++\sqrt{\frac{\rho_-}{\rho_+}\left(\frac{p_- - p_+}{\rho_--\rho_+}\right)},
\end{align*}
where $p_\pm := R\r_\pm \th_\pm$. 
The following lemma summarizes the main properties of the viscous shock waves, which will be used in the subsequent analysis.
\begin{lemma}\label{lem:viscous_shock}\cite{KVW-NSF}
	For a given right-end state $(\rho_+,u_+, \theta_+)$, there exists a constant $C>0$ such that the following holds. For any left-end state $(\rho_-,u_-,\theta_-)$ connected with $(\rho_+,u_+, \theta_+)$ via 3-shock curve, there exists a unique solution $(\bar{\rho}, \bar{u}, \bar{\theta})(x-\sigma t) = (\bar{\rho}, \bar{u}, \bar{\theta})(\xi)$ to \eqref{eq:Vshock} with $\rhotil(0) = (\rho_- + \rho_+)/2$. Let $\delta$ be the strength of the shock defined as $\delta:=|u_+-u_-|\sim|\rho_+-\rho_-|\sim |\theta_+ - \theta_-|$. Then we have
	\[\bar{\rho}_{\xi} <0,\quad \bar{u}_{\xi }<0,\quad \bar{\theta}_{\xi} <0, \quad \forall \xi \in \mathbb{R},\]
	and
	\begin{align}
		\begin{aligned}\label{shock_prop}
			&|\left(\bar{\rho}(\xi)-\rho_{\pm}, \bar{u}(\xi) - u_{\pm}, \bar{\theta}(\xi) - \theta_{\pm}\right)|\le C\delta e^{-C\delta|\xi|},\quad \pm \xi>0,\\
			&|(\bar{\rho}_\xi,\bar{u}_\xi, \bar{\theta}_\xi)|\le C\delta^2 e^{-C\delta|\xi|},\quad \xi \in\R,\\
			&|(\bar{\rho}_{\xi\xi}, \bar{u}_{\xi\xi}, \bar{\theta}_{\xi\xi})|\le C\delta|(\bar{\rho}_\xi,\bar{u}_\xi, \bar{\theta}_\xi)|,\quad \xi\in\R.
		\end{aligned}
	\end{align}
	In particular, $|\bar{\rho}_\xi| \sim |\bar{u}_\xi| \sim |\bar{\theta}_\xi|$ for all $\xi \in \R$, more explicitly,
		\begin{equation} \label{estderi}
			\left|\bar{\rho}_\xi-\frac{\rho_-}{\sqrt{\gamma R \theta_-}}\bar{u}_\xi \right|\le C\delta |\bar{u}_\xi| \quad \text{and} \quad 
			\left|\bar{\theta}_\xi -\frac{(\gamma-1)\theta_-}{\sqrt{\gamma R \theta_-}}\bar{u}_\xi \right|\le C\delta |\bar{u}_\xi|, \quad \forall \xi \in \R.
		\end{equation}
		Moreover, define $\sigma_- := u_- +\sqrt{\gamma R\theta_-}$. Then we have 
		\begin{align}\label{est:sigma}
			|\sigma- \sigma_-| \leq C\delta.
		\end{align}
\end{lemma} 
\begin{remark}\label{rmk:VS}
By Lemma \ref{lem:viscous_shock} and the Lax entropy condition $\lambda_3(\rho_-,u_-,\th_-) > \sigma > \lambda_3(\rho_+,u_+,\th_+)$, it follows that 
\[
\left|\sigma - \util - \sqrt{\gamma R \thtil}\right| \leq C\delta.
\]
By choosing $\delta>0$ sufficiently small, we obtain $\sigma - \util > \frac{\sqrt{\gamma R \theta_+}}{2} > 0$. 
\end{remark}
\subsection{Poincar\'e-type inequality}
One of the main tools for proving shock stability is the Poincar\'e-type inequality. Note that the constant $\frac{1}{2}$ in the inequality is optimal and independent of the domain size. The proof of the lemma below can be found in \cite{HKKL24} (see also \cite{KV21}).

\begin{lemma}\label{Poincare}\cite{HKKL24}
	For any $c<d$ and function $f : \left[c, d\right]  \longrightarrow  \mathbb{R}$ satisfying $\int_{c}^d (y-c)(d-y)|f'(y)|^2dy<\infty$,   
	\[\int_{c}^d \left|f(y)-\frac{1}{d-c}\int_{c}^d f(y)dy\right|^2dy\le \frac{1}{2}\int_{c}^d(y-c)(d-y)|f'(y)|^2dy.\]
\end{lemma}
\section{A priori estimates and proof of the main theorem}
\setcounter{equation}{0}
In this section, we derive \textit{a priori estimates} for the $H^1$-perturbation between the solution and the viscous shock wave. Based on these estimates, we establish the large-time behavior of solutions to both IBVPs toward the viscous shock waves.

\subsection{Local existence of solutions}
We first state the local existence of a solution to the outflow and impermeable wall problem.

\begin{proposition}\label{prop:local}
	For any constant $\beta>0$, let $\underline{\rho}, \underline{u}$, and $\underline{\theta}$ be smooth monotone functions such that
	\[
    (\underline{\rho}(x), \underline{u}(x),\underline{\theta}(x))=(\rho_+,u_+,\theta_+), \, \, \text{for } x \ge \beta,\,\,  \underline{\rho}(0)>0, \, \, \underline{u}(0) = u_-, \, \, \text{and} \, \, \underline{\theta}(0)=\theta_- .
    \]
	For any constants $M_0$, $M_1$, $\underline{\kappa}_0$, $\overline{\kappa}_0$, $\underline{\kappa}_1$, and $\overline{\kappa}_1$ with $0<M_0<M_1$ and $0<\underline{\kappa}_1<\underline{\kappa}_0<\overline{\kappa}_0<\overline{\kappa}_1$, there exists a constant $T_0>0$ such that if
	\begin{align*}
		&\|(\rho_0-\underline{\rho},u_0-\underline{u}, \theta_0-\underline{\theta})\|_{H^1(\R_+)}\le M_0,\\
		&0<\underline{\kappa}_0\le \rho_0(x), \theta_0(x)\le \overline{\kappa}_0,\quad x\in\R_+,
	\end{align*}
	the outflow problem \eqref{eq:NSF}-\eqref{eq:initial}-\eqref{outflow} (or the impermeable wall problem \eqref{eq:NSF}-\eqref{eq:initial}-\eqref{imperable}) admits a unique solution $(\rho,u, \theta)$ on $[0,T_0]$ such that
	\begin{align*}
		\rho -\underline{\rho}\in C([0,T_0];H^1(\R_+)),\quad(u-\underline{u}, \theta-\underline{\theta}) \in C([0,T_0];H^1(\R_+))\cap L^2(0,T_0;H^2(\R_+)),
	\end{align*}
	and
	\[\|(\rho-\underline{\rho},u-\underline{u}, \theta-\underline{\theta})\|_{L^\infty(0,T_0;H^1(\R_+))}\le M_1.\]
	Moreover, for the outflow problem (or the impermeable wall problem), we have
	\[
	\underline{\kappa}_1\le \rho(t,x) \le \overline{\kappa}_1, \quad \underline{\kappa}_1\le \theta(t,x)\le \overline{\kappa}_1, \quad \forall t>0, \, \, \forall x\in\R_+,
	\]
	and
	\[
	u(t,0)=u_- <0 \, \, (\text{or } u_-=0 \, \, \text{for the impermeable case}),\quad \theta(t,0) = \theta_->0, \quad \forall t>0.
	\]
\end{proposition}
\begin{proof}
Since the local-in-time existence can be proved by a standard iteration argument, we only present the corresponding iteration scheme.

We define a sequence $\{(\rho_n(t,x),u_n(t,x),\theta_n(t,x))\}$ as follows. First, we set 
\[
(\rho_0(x,t), u_0(x,t), \th_0(x,t)) := (\rho_0(x), u_0(x), \th_0(x)), \quad \text{and} \quad \rho_1(t,x) := \rho_0(x).
\]
Then define $u_1(t,x)$ and $\th_1(t,x)$ by
\begin{align*}
\left\{
\begin{aligned}
     &(u_1)_t - \frac{\mu}{\rho_1}(u_1)_{xx} = -u_0 (u_0)_x - \frac{1}{\rho_0}p(\rho_0,\th_0)_x,\\
    &\frac{R}{\gamma-1}(\theta_1)_t - \frac{\kappa}{\rho_1}(\theta_1)_{xx}=-\frac{R}{\gamma-1}u_0(\theta_0)_x-\frac{p(\rho_0,\th_0)}{\rho_0}(u_0)_x+\frac{\mu}{\rho_0}|(u_0)_x|^2,\\
    &(u_1,\th_1)(0,x) = (u_0,\th_0), \quad (u_1, \th_1)(t,0) = (u_- , \th_-).
\end{aligned}
\right.
\end{align*}
Suppose $n\geq 2$. For given $(\rho_{n-1},u_{n-1},\th_{n-1})$, we define $\rho_n$ and then define $u_n$ and $\th_n$ by
\begin{align*}
    \left\{
    \begin{aligned}
    &(\rho_n)_t+ u_{n-1}(\rho_n)_x = -\rho_{n-1}(u_{n-1})_x\\
    &(u_n)_t - \frac{\mu}{\rho_n}(u_n)_{xx} = -u_{n-1} (u_{n-1})_x - \frac{1}{\rho_{n-1}}p(\rho_{n-1},\th_{n-1})_x,\\
    &\frac{R}{\gamma-1}(\theta_n)_t - \frac{\kappa}{\rho_n}(\theta_n)_{xx}=-\frac{R}{\gamma-1}u_{n-1}(\theta_{n-1})_x-\frac{p(\rho_{n-1},\th_{n-1})}{\rho_{n-1}}(u_{n-1})_x+\frac{\mu}{\rho_{n-1}}|(u_{n-1})_x|^2,\\
    &(\rho_n,u_n,\th_n)(0,x) = (u_0,\th_0), \quad (u_n, \th_n)(t,0) = (u_- , \th_-).
    \end{aligned}
    \right.
\end{align*}

We omit the details of the proof.
\end{proof}

\subsection{Construction of weight function}
To prove the main results, we employ the method of $a$-contraction with shifts, namely, a weighted relative entropy method. To this end, we construct a suitable weight function $a(t,x)$ to ensure a certain contraction property of the viscous shock wave.

We now define the weight function $a(t,x)=a(\xi)$ as follows:
\begin{equation}\label{def_a}
	a(\xi):=1+\frac{u_--\bar{u}(\xi)}{\sqrt{\delta}},\quad \xi = x-\sigma t,
\end{equation}
where $\delta = |u_+-u_-|$ denotes the shock strength. We note that the weight function $a$ satisfies $1\le a\le 1+\sqrt{\delta}<\frac{3}{2}$ for small enough $\delta$, and 
\begin{equation*} 
	a'(\xi)=-\frac{ \bar{u}'(\xi)}{\sqrt{\delta}}>0,\quad \mbox{and}\quad |a'(\xi)|\sim \frac{\bar{u}'(\xi)}{\sqrt{\delta}}.
\end{equation*} 
	\subsection{Construction of shift}
To obtain the stability estimate, the viscous shock wave needs to be shifted appropriately. Here, we explicitly construct the shift function.

We define the shift function $\bX:\bbr_+\to\bbr$ as a solution to the following ODE:
\begin{equation}\label{def_shift}
	\begin{aligned}
		\dot{\mathbf{X}}(t)&=-\frac{M}{\delta}\int_{\mathbb{R}_+}a^{\mathbf{X},\beta}\left[ R\frac{\bar{\theta}^{\mathbf{X},\beta}}{\bar{\rho}^{\mathbf{X},\beta}}(\rho-\bar{\rho}^{\mathbf{X},\beta})\bar{\rho}^{\mathbf{X},\beta}_x+\rho(u-\bar{u}^{\mathbf{X},\beta})\bar{u}^{\mathbf{X},\beta}_x+\frac{R}{\gamma-1}\frac{\rho}{\bar{\theta}^{\mathbf{X},\beta}}(\theta-\bar{\theta}^{\mathbf{X},\beta})\bar{\theta}_x^{\mathbf{X},\beta}\right]\,dx,\\
		\bX(0)&=0,
	\end{aligned}
\end{equation}
	where $M, \beta>0$ are positive constants which will be chosen later. Here, for any function $f:\bbr\to\bbr$, we use the abbreviated notation
\[
f^{\bX,\beta}(\cdot):=f(\cdot-\bX(t)-\beta).
\]
The existence and Lipschitz continuity of the shift $\mathbf{X}(t)$ can be proved by applying the Cauchy-Lipschitz theorem (see, for example, \cite[Lemma 3.2]{KVW-NSF}). Moreover, the definition of the shift in \eqref{def_shift} originates from the linear part of the $\mathbf{Y}$ in \eqref{term:Y}. Defining the shift in this way ensures that the shift contributes a good term when applying the Poincaré inequality in Lemma \ref{Poincare} (see \eqref{est:goodx}).

\subsection{A priori estimates}
We now state the main propositions on the \textit{a priori} estimates for both boundary value problems, \eqref{outflow} and \eqref{imperable}.

\begin{proposition}[Outflow problem]\label{prop:main}
	For a given constant state $(\rho_+,u_+, \theta_+) \in\Omega^{-}_{sub}\cup \Gamma_{trans}^{-}$, there exist positive constants $C_0, \delta_0,\e >0$ such that the following holds. \\
	For any $(u_-,\theta_-)$ satisfying $0>u_->u_+$ and $(u_-,\theta_-) \in S_3^{P}(\rho_+,u_+,\theta_+)$ with
	\[
	|u_+-u_-| + |\theta_+ - \theta_-|<\delta_0,
	\]
	let $\r_->0$ be the (unique) constant state such that $(\rho_-,u_-,\theta_-) \in S_3(\rho_+,u_+,\theta_+)$. Denote $(\bar{\rho},\bar{u},\bar{\theta})$ the viscous 3-shock wave with end states $(\r_-, u_-, \th_-)$ and $(\r_+, u_+, \th_+)$. Then there exists $\beta>0$ large enough (depending only on the shock strength $\delta_0$) such that the following holds.
	Suppose that $(\rho,u,\theta)$ is the solution to the outflow problem \eqref{eq:NSF}-\eqref{eq:initial}-\eqref{outflow} on $[0,T]$ for some $T>0$, and the shift $\bX(t)$ is defined in \eqref{def_shift}. Suppose that
	\begin{equation}\label{apriori1}
			\begin{aligned}
		&\rho-\bar{\rho}^{\bX,\beta}\in C([0,T];H^1(\R_+)),\\
		&(u-\bar{u}^{\bX,\beta}, \theta-\bar{\theta}^{\bX,\beta})\in C([0,T];H^1(\R_+))\cap L^2(0,T;H^2(\R_+)),
	\end{aligned}
	\end{equation}
	and
	\begin{equation}\label{apriori_small}
		\|\rho-\bar{\rho}^{\bX,\beta}\|_{L^\infty(0,T;H^1(\R_+))}+\|u-\bar{u}^{\bX,\beta}\|_{L^\infty(0,T;H^1(\R_+))} + 	\|\theta-\bar{\theta}^{\bX,\beta}\|_{L^\infty(0,T;H^1(\R_+))}\le \e.
	\end{equation}
	Then we have
	\begin{equation}\label{est:apriori}
		\begin{aligned} 
			&\sup\limits_{t\in [0,T]}\Big[\|\rho-\bar{\rho}^{\bX,\beta}\|_{H^1(\R_+)}+\|u-\bar{u}^{\bX,\beta}\|_{H^1(\R_+)} +\|\theta-\bar{\theta}^{\bX,\beta}\|_{H^1(\R_+)}\Big]\\
			&\quad +\sqrt{\int_0^T (\delta|\sX|^2+\mathbf{G}_1 + \mathbf{G}_2+ \mathbf{G}^S+ D_{\r_1} +D_{u_1}+D_{\theta_1}+D_{u_2} + D_{\theta_2})\,d\tau}\\
			&\quad + |u_-|\int_{0}^{T}|\left(\rho-\bar{\rho}, u- \bar{u}, \theta -\bar{\theta}\right)(t,0)|^2  \,dt+ |u_-| \int_{0}^{T} \left|(\rho -\bar{\rho})_x(t,0)\right|^2 \,dt     \\
			&\le C_0\left(\|\rho_0-\bar{\rho}^{\bX,\beta}(0,\cdot)\|_{H^1(\R_+)}+\|u_0-\bar{u}^{\bX,\beta}(0,\cdot)\|_{H^1(\R_+)}+ \|\theta_0-\bar{\theta}^{\bX,\beta}(0,\cdot)\|_{H^1(\R_+)} \right) +C_0e^{-C\delta\beta},
		\end{aligned}
	\end{equation}
	where $C_0$ is independent of $T$ and
	\begin{align}\label{good_terms}
		\begin{split}
			&\begin{aligned}
				& \mathbf{G}_1 := \frac{R\theta_-}{2\rho_-}\sqrt{\gamma R \theta_-} \int_{\mathbb{R}_+} a_x\left[(\rho- \bar{\rho}^{\bX,\beta})- \frac{\rho_-}{\gamma R \theta_-}(u-\bar{u}^{\bX,\beta})\right]^2 \,dx, \\
				& \mathbf{G}_2 := \frac{R\rho_-}{2(\gamma-1)\theta_-}\sqrt{\gamma R\theta_-} \int_{\mathbb{R}_+} a_x\left[(\theta-\bar{\theta}^{\bX,\beta}) - \frac{(\gamma-1)\theta_-}{\sqrt{\gamma R \theta_-}}(u-\bar{u}^{\bX,\beta})\right]^2 \,dx, \\
				& \mathbf{G}^S := \int_{\R_+} |\bar{u}^{\bX,\beta}_x| \left|\left(\rho -\bar{\rho}^{\bX,\beta},  u-\bar{u}^{\bX,\beta}, \theta -\bar{\theta}^{\bX,\beta}\right)\right|^2 \,dx,
			\end{aligned}\\
			&\begin{aligned}
				& D_{\r} := \int_{\R_+} |(\r - \rhotilX)_x|^2 \,dx, &&D_{u_1}:= \int_{\R_+}|(u-\bar{u}^{\bX,\beta})_{x}|^2\,dx,  \quad D_{\theta_1} := \int_{\R_+}|(\theta-\bar{\theta}^{\bX,\beta})_{x}|^2\,dx,  \\
				& D_{u_2}:= \int_{\R_+}|(u-\bar{u}^{\bX,\beta})_{xx}|^2\,dx, && D_{\theta_2} := \int_{\R_+}|(\theta-\bar{\theta}^{\bX,\beta})_{xx}|^2\,dx,
			\end{aligned}
		\end{split}
	\end{align}
	where the weight function $a$ is defined in \eqref{def_a}. In particular, for all $0\le t \le T$,
	\begin{equation} \label{bddx12}
		|\dot{\bX}(t)|\le C_0\lVert (\rho-\bar{\rho}^{\bX,\beta},u-\bar{u}^{\bX,\beta}, \theta-\bar{\theta}^{\bX,\beta} )(t,\cdot) \rVert_{L^\infty(\mathbb{R}_+)}.
	\end{equation}
\end{proposition}
The \textit{a priori} estimate for the impermeable case is analogous to the one in the outflow case, as follows:
\begin{proposition}[Impermeable wall problem]\label{prop:main2}
	For given $(\rho_+, \theta_+)$, there exist positive constants $C_0, \delta_0,\e >0$ such that the following holds. For any constant states $u_+<0$ satisfying $$- u_+ = |u_+| < \delta_0,$$ let $\r_-$ and $\theta_-$ be the (unique) constant states such that $(\rho_-,u_-,\theta_-):= (\rho_-,0,\theta_-)\in S_3(\rho_+,u_+,\theta_+)$. Denote  $(\bar{\rho},\bar{u}, \bar{\theta})$ the viscous 3-shock wave with end states $(\r_-, u_-, \th_-)$ and $(\r_+, u_+, \th_+)$. Then there exists $\beta>0$ large enough (depending only on the shock strength $\delta_0$) such that the following holds. Suppose that $(\rho,u,\theta)$ is the solution to the impermeable wall problem \eqref{eq:NSF}-\eqref{eq:initial}-\eqref{imperable} on $[0,T]$ for some $T>0$, and the shift $\bX(t)$ is defined in \eqref{def_shift}. Assume further that the solution $(\r,u,\th)$ satisfies \eqref{apriori1} and \eqref{apriori_small}. Then the estimates \eqref{est:apriori} and \eqref{bddx12} hold.
\end{proposition}

\subsection{Proof of Theorem \ref{thm:outflow}}
In the outflow setting, we use a continuation argument based on Proposition \ref{prop:local} and Proposition \ref{prop:main} (or Proposition \ref{prop:main2} for the impermeable case) to establish the global-in-time existence of solutions. We also use Proposition \ref{prop:main} (or Proposition \ref{prop:main2} for the impermeable case) to prove the large-time behavior \eqref{long_time_outflow} (or \eqref{long_time_impermeable}). Since these proofs are similar to those in the previous articles (e.g. \cite{HKKL24, KOW25}), we omit the details.\\

Hence, it only remains to prove Proposition \ref{prop:main} and Proposition \ref{prop:main2}.

\section{Zeroth order estimates}
\setcounter{equation}{0}
In this section, we provide the zeroth order estimates for two boundary value problems, by using the method of $a$-contraction with shifts. 
In what follows, for notational simplicity, we suppress the dependence on the shift $\bX$ and $\beta$. Specifically, we use the following concise notations without any confusion:
\begin{align*}
	\begin{split}
		\begin{aligned}
		&a(t,x)=a(x-\sigma t-\bX(t)-\beta),\\
		&\bar{\rho}(t,x) = \bar{\rho}^{\bX,\beta}(t,x)=\bar{\rho}(x-\sigma t-\bX(t)-\beta),&&
		\bar{u}(t,x) = \bar{u}^{\bX,\beta}(t,x)=\bar{u}(x-\sigma t-\bX(t)-\beta),\\
		&\bar{\theta}(t,x) = \bar{\theta}^{\bX,\beta}(t,x)=\bar{\theta}(x-\sigma t-\bX(t)-\beta),
		&&\bar{E}(t,x) = \bar{E}^{\bX,\beta}(t,x)=\bar{E}(x-\sigma t-\bX(t)-\beta).
		\end{aligned}
	\end{split}
\end{align*}
In what follows, we denote by $C$ a positive $O(1)$-constant that may vary from line to line, but is independent of the parameters $\delta(=|u_- - u_+|), \varepsilon, \beta$, and the time $T$.	

This section is dedicated to the proof of the following lemma.
	\begin{lemma}\label{lem:main}
	Under the hypothesis of Proposition \ref{prop:main} (or Proposition \ref{prop:main2}), there exists a positive constant $C$ such that
	\begin{align*}
		\begin{aligned} 
			&\sup_{t \in [0,T]}\|(\rho-\bar{\rho}, u-\bar{u}, \theta - \bar{\th})(t,\cdot)\|_{L^2(\R_+)}^2+\int_0^T (\delta|\sX(s)|^2 + \mathbf{G}_1 +\mathbf{G}_2 + \mathbf{G}^S +D_{u_1}  + D_{\theta_1}) \,d\tau  \\
			&\qquad + |u_-|\int_{0}^{T}|\left(\rho(\tau,0) -\bar{\rho}(\tau,0), u(\tau, 0)- \bar{u}(\tau,0), \theta(\tau,0) -\bar{\theta}(\tau,0)\right)|^2\,d\tau   \\
			&\quad\le C\|(\rho-\bar{\rho}, u-\bar{u}, \theta - \bar{\th})(0,\cdot)\|_{L^2(\R_+)}^2\\
			&\qquad +Ce^{-C\delta \beta}  +C\varepsilon^2\int_{0}^T\left(\|(u-\bar{u})_{xx}\|_{L^2(\mathbb{R}_+)}^{2} + \|(\theta-\bar{\theta})_{xx}\|_{L^2(\mathbb{R}_+)}^{2}\right)\,d\tau,
		\end{aligned}
	\end{align*}
	where $\mathbf{G}_1, \mathbf{G}_2, \mathbf{G}^S$, $D_{u_1}$, and $D_{\th_1}$ are the terms defined in \eqref{good_terms}.
\end{lemma}

\subsection{Relative entropy method}
	The proofs of Proposition \ref{prop:main} and \ref{prop:main2} are based on the relative entropy method, introduced by Dafermos and DiPerna \cite{D96, D79}. 
	
	First, let $U$ denote the conserved quantities:
	\[
	U := \begin{pmatrix}
		\rho \\ m \\ E
	\end{pmatrix} = \begin{pmatrix}
		\rho \\ \rho u \\ \rho\left(e + \frac{1}{2}u^2\right)
	\end{pmatrix}.
	\] 
	According to the Gibbs relation $\theta d s = d e + p d \left(\frac{1}{\r}\right)$, the entropy $s(U)$ is given by
	\[
	s(U) = -R \ln \r + \frac{R}{\gamma - 1} \ln \theta.
	\] 
	The mathematical entropy $\eta(U)$ is then defined as $\eta(U) := -(\rho s)(U)$. For a given entropy $\eta(U)$ and conserved quantities $U$ and $V$, the relative entropy $\eta(U|V)$ is defined by 
	\[
	\eta(U|V):=\eta(U) - \eta(V) - \nabla \eta(V) \cdot (U-V),
	\]
	which is locally quadratic and positive definite.

	If $U$ is a solution to \eqref{eq:NSF} and $\bar{U}$ is a (shifted) viscous shock wave, then the relative entropy weighted by $\bar{\theta}$, as provided in Appendix \ref{Appendix A}, is given by
\begin{align*}
	\bar{\theta}\eta(U|\Bar{U})=\rho \left[R\bar{\theta} \left( \frac{\bar{\rho}}{\rho }-\ln \frac{\bar{\rho}}{\rho}  -1 \right) + \frac{R}{\gamma -1} \bar{\theta} \left(\frac{\theta}{\bar{\theta}} -\ln \frac{\theta}{\bar{\theta}} -1\right) + \frac{1}{2}(u-\bar{u})^2\right].
\end{align*}
Then by virtue of the convex function $\Phi(z):=z-\ln z -1$, it can be concisely rewritten as
\begin{align*} 
		\bar{\theta}\eta(U|\Bar{U})= \rho\left[R\bar{\theta} \Phi\left(\frac{\bar{\rho}}{\rho}\right)+\frac{R}{\gamma-1}\bar{\theta} \Phi\left(\frac{\theta}{\bar{\theta}}\right)+\frac{1}{2}(u-\bar{u})^2\right].
\end{align*}
Next, we will compute the relative entropy of $U$ and $\bar{U}$ weighted by $a(t,x) \bar{\theta}(t,x)$:
\[\int_{\mathbb{R}_+} a(t,x) \bar{\theta} (t,x)\eta\big(U(t,x) | \Bar{U}(t,x)\big) \,dx. \]
	

\begin{lemma}\label{lma:BRE}
	Let $a(t,x)$ be the weight function defined by \eqref{def_a}. Let $U$ be the solution to \eqref{eq:NSF} and $\bar{U}$ be the shifted viscous shock wave. Then we have 
	\begin{align*}
	\frac{d}{dt}	\int_{\mathbb{R}_+} a (t,x)\bar{\theta} (t,x)\eta\left(U(t,x) | \Bar{U}(t,x)\right) \,dx  = \dot{\mathbf{X}} \mathbf{Y}(U) + \mathcal{J}^{bad}(U) - \mathcal{J}^{good}(U) + \mathcal{P}(U),
	\end{align*}
	where 
\begin{align}
&\begin{aligned}\label{term:Y}
	\mathbf{Y}(U) :=&- \int_{\mathbb{R}_+} a_x \bar{\theta} \eta(U|\bar{U})  \,dx - \int_{\mathbb{R}_+} a \rho  \bar{\theta}_x \left[R\bar{\theta} \Phi\left(\frac{\bar{\rho}}{\rho}\right)+\frac{R}{\gamma-1}\bar{\theta} \Phi\left(\frac{\theta}{\bar{\theta}}\right)\right] \,dx \\
	& + \int_{\mathbb{R}_+} a \left[\rho \bar{u}_x(u-\bar{u}) + R\bar{\theta} (\rho -\bar{\rho}) \frac{\bar{\rho}_x}{\bar{\rho}} + \frac{R}{\gamma -1} \rho \frac{\theta -\bar{\theta}}{\bar{\theta}} \bar{\theta}_x\right]\,dx,
\end{aligned}\\
&\begin{aligned} \notag 
	\mathcal{J}^{bad}(U) & = \int_{\mathbb{R}_+} a\bigg[ -\rho \bar{u}_x (u-\bar{u})^2 + R \bar{\theta}_x (\rho -\bar{\rho})(u-\bar{u}) + \rho (u-\sigma)\bar{\theta}_x \bigg(R\Phi\left(\frac{\bar{\rho}}{\rho}\right)+\frac{R}{\gamma-1}\Phi\left(\frac{\theta}{\bar{\theta}}\right)\bigg)    \\
	&\phantom{\quad +  \int_{\mathbb{R}_+} a\bigg[ }-\frac{R}{\gamma- 1} \rho \bar{\theta}_x (u-\bar{u}) \frac{\theta -\bar{\theta}}{\theta }\bigg] \,dx + \int_{\mathbb{R}_+} a_x (u-\bar{u})(p-\bar{p}) \,dx \\
	&\quad - \int_{\mathbb{R}_+}  a_x \bigg[\mu (u-\bar{u})(u-\bar{u})_x   + \kappa \frac{(\theta -\bar{\theta})}{\theta}(\theta -\bar{\theta})_x \bigg]\,dx \\
	&\quad + \int_{\mathbb{R}_+} a\bigg[-\frac{R}{\gamma-1} \rho \bar{\theta}_x (u-\bar{u}) \frac{(\theta- \bar{\theta})^2}{\theta \bar{\theta}} + \mu \rho \bar{u}_{xx} (u-\bar{u}) (\frac{1}{\rho} - \frac{1}{\bar{\rho}}) + \mu \frac{\theta -\bar{\theta}}{\theta} (u^2_x - \bar{u}^2_x )   \\
	&\phantom{:=-\int_{\mathbb{R}_+} a\bigg[} + \mu \rho \bar{u}^2_x \frac{\theta - \bar{\theta}}{\theta} (\frac{1}{\rho} - \frac{1}{\bar{\rho}}) - \kappa \frac{\rho}{\bar{\rho}} \frac{(\theta -\bar{\theta})^2}{\theta \bar{\theta}} \bar{\theta}_{xx}  - \mu  \frac{\rho}{\bar{\rho}}  \frac{(\theta -\bar{\theta})^2}{\theta \bar{\theta}} \bar{u}_x^2 + \kappa \rho \frac{\theta -\bar{\theta}}{\theta} (\frac{1}{\rho } -\frac{1}{\bar{\rho}}) \bar{\theta}_{xx}   \\
	&\phantom{:=-\int_{\mathbb{R}_+} a\bigg[} + \kappa \frac{\theta_x}{\theta^2} (\theta -\bar{\theta}) (\theta -\bar{\theta})_x\bigg] \,dx,
\end{aligned}
\end{align}
and
\begin{align}\label{term:Jgood}
	\begin{aligned}
			\mathcal{J}^{good} (U) := &  \int_{\mathbb{R}_+} a_x (\sigma -u)  \bar{\theta} \eta (U|\bar{U}) \,dx + \int_{\mathbb{R}_+} a \left(\mu |(u-\bar{u})_x|^2 + \frac{\kappa}{\theta}|(\theta -\bar{\theta})_x|^2\right)\,dx, \\
			\mathcal{P}:= & \bigg[a u \bar{\theta} \eta(U|\bar{U}) - \mu a (u-\bar{u})(u-\bar{u})_x - \kappa \frac{a}{\theta}(\theta -\bar{\theta})(\theta -\bar{\theta})_x + Ra \rho (u-\bar{u})(\theta -\bar{\theta})  \\
			& + Ra \bar{\theta}(\rho -\bar{\rho})(u-\bar{u})\bigg]\bigg|_{x=0} =: \mathcal{P}_1 + \mathcal{P}_2 + \mathcal{P}_3 +\mathcal{P}_4 +\mathcal{P}_5  .
		\end{aligned}
	\end{align}
\end{lemma}
\begin{remark}
	From Remark \ref{rmk:VS} and the smallness of $\e$, it follows that $\sigma - u > 0$. Thus, by the definition of the weight function $a$ in \eqref{def_a}, we have $a_x (\sigma - u) >0$. Consequently, the term $\mathcal{J}^{good}(U)$ consists of good terms.
\end{remark}
\begin{remark}
	The terms in $\mathcal{P}$ arise from the boundary values when performing integration by parts. For the outflow problem, the term $\mathcal{P}_1$ is a good term due to the condition $u_- <0$. On the other hand, $\mathcal{P}_1$ vanishes in the impermeable setting.
\end{remark}
\begin{proof}
	Before proving Lemma \ref{lma:BRE}, note that the shifted viscous shock $\bar{U} = (\rhotil, \bar{m}, \bar{E})^t$ satisfies the following equations:
	\begin{equation} \label{eq:NSFs}
	\begin{cases}
		& 	\bar{\rho}_t +(\bar{\rho}\bar{u}  )_x=-\dot{\mathbf{X}}(t)	\bar{\rho}_x , \\
		& 	\bar{u}_t+ 	\bar{u}\bar{u}_x  + \frac{	\bar{p}_x }{\bar{\rho}}=\mu \frac{	\bar{u}_{xx}  }{\bar{\rho} }- \dot{\mathbf{X}}(t)	\bar{u} _x,\\
		& 	\frac{R}{\gamma-1}\bar{\theta} _t + \frac{R}{\gamma-1} \bar{u}\bar{\theta}_x + \frac{\bar{p}}{\bar{\rho}}\bar{u}_x  = \frac{\kappa}{\bar{\rho}} \bar{\theta}_{xx}+ \frac{\mu}{\bar{\rho}}( \bar{u}_x)^2 -\frac{R}{\gamma-1}\dot{\mathbf{X}}(t)\bar{\theta}_x.\\
	\end{cases}
\end{equation}
Combining \eqref{eq:NSFS} and \eqref{eq:NSFs}, we derive the equations for the perturbations:
	\begin{equation} \label{eq:pertubation1}
		\begin{cases}
			& (\rho-\bar{\rho})_t+(\rho u-\bar{\rho}\bar{u})_x=\dot{\mathbf{X}}(t)\bar{\rho}_x, \\
			& (u-\bar{u})_t+(uu_x-\bar{u}\bar{u}_x)+\left(\frac{p_x}{\rho}-\frac{\bar{p}_x}{\bar{\rho}}\right)=\mu\left(\frac{u_{xx}}{\rho}-\frac{\bar{u}_{xx}}{\bar{\rho}}\right)+\dot{\mathbf{X}}(t)\bar{u}_x,\\
			& (\theta-\bar{\theta})_t+(u\theta_x-\bar{u}\bar{\theta}_x)=\frac{\gamma-1}{R}\left[-\left(\frac{p}{\rho}u_x-\frac{\bar{p}}{\bar{\rho}}\bar{u}_x\right)+\kappa\left(\frac{\theta_{xx}}{\rho}-\frac{
				\bar{\theta}_{xx}
			}{\bar{\rho}}\right)+\mu\left(\frac{u_x^2}{\rho}-\frac{\bar{u}_x^2}{\bar{\rho}}\right)\right]+\dot{\mathbf{X}}(t)\bar{\theta}_x.
		\end{cases}
	\end{equation}
	We now prove Lemma \ref{lma:BRE}. Observe that
\begin{align*}
	& \frac{d}{dt}\int_{\mathbb{R}_+} a\bar{\theta}\eta\left(U | \Bar{U}\right) \,dx   \\
	&\quad=-(\dot{\mathbf{X}}(t)+\sigma)\int_{\mathbb{R}_+}a'\bar{\theta}\eta\left(U | \Bar{U}\right) \,dx +\int_{\mathbb{R}_+} a \rho_t\left[R\bar{\theta}\Phi\left(\frac{\bar{\rho}}{\rho}\right)+\frac{R}{\gamma-1}\bar{\theta}\Phi\left(\frac{\theta}{\bar{\theta}}\right)+\frac{|u-\bar{u}|^2}{2}\right]\,dx  \\
	&\qquad +\int_{\mathbb{R}_+}a \rho \partial_t \left[R\bar{\theta}\Phi\left(\frac{\bar{\rho}}{\rho}\right)+\frac{R}{\gamma-1}\bar{\theta}\Phi\left(\frac{\theta}{\bar{\theta}}\right)+\frac{|u-\bar{u}|^2}{2}\right]\,dx\\
	&\quad=-(\dot{\mathbf{X}}(t)+\sigma)\int_{\mathbb{R}_+}a'\bar{\theta}\eta\left(U|\Bar{U}\right)\,dx -\int_{\mathbb{R}_+}a(\rho u)_x \left[R\bar{\theta}\Phi\left(\frac{\bar{\rho}}{\rho}\right)+\frac{R}{\gamma-1}\bar{\theta}\Phi\left(\frac{\theta}{\bar{\theta}}\right)+\frac{|u-\bar{u}|^2}{2}\right]\,dx\\
	&\qquad+\int_{\mathbb{R}_+}a \rho \partial_{t}\left[R\bar{\theta}\Phi\left(\frac{\bar{\rho}}{\rho}\right)+\frac{R}{\gamma-1}\bar{\theta}\Phi\left(\frac{\theta}{\bar{\theta}}\right)+\frac{|u-\bar{u}|^2}{2}\right]\, dx.
\end{align*}
Then integration by parts yields
\begin{align}\label{eq:REalculation}
	\begin{aligned}
	& \frac{d}{dt}\int_{\mathbb{R}_+} a\bar{\theta}\eta\left(U | \Bar{U}\right) \,dx   \\
	&\quad= - \int_{\mathbb{R}_+}a' (\sigma - u)\bar{\theta}\eta\left(U|\Bar{U}\right)\,dx -\dot{\mathbf{X}}(t) \int_{\mathbb{R}_+}a'\bar{\theta}\eta\left(U|\Bar{U}\right)\,dx  + a u \bar{\theta}\eta\left(U | \Bar{U}\right)\bigg|_{x=0} \\
	&\qquad+\int_{\mathbb{R}_+}a \rho (\partial_{t} + u\partial_x)\left[R\bar{\theta}\Phi\left(\frac{\bar{\rho}}{\rho}\right)+\frac{R}{\gamma-1}\bar{\theta}\Phi\left(\frac{\theta}{\bar{\theta}}\right)+\frac{|u-\bar{u}|^2}{2}\right]\,dx \\
	&\quad= - \int_{\mathbb{R}_+}a'(\sigma - u)\bar{\theta}\eta\left(U|\Bar{U}\right)\,dx -\dot{\mathbf{X}}(t) \int_{\mathbb{R}_+}a'\bar{\theta}\eta\left(U|\Bar{U}\right)\,dx  + a u \bar{\theta}\eta\left(U | \Bar{U}\right)\bigg|_{x=0} \\
	&\qquad+ \int_{\mathbb{R}_+}a \rho \Big((\partial_{t} + u\partial_x) \bar{\theta} \Big)\left[R\Phi\left(\frac{\bar{\rho}}{\rho}\right)+\frac{R}{\gamma-1} \Phi\left(\frac{\theta}{\bar{\theta}}\right)\right]\,dx   \\
	&\qquad + \int_{\mathbb{R}_+}\underbrace{a \rho \left(\bar{\theta}
	(\partial_{t} + u\partial_x)\left[R\Phi\left(\frac{\bar{\rho}}{\rho}\right)+\frac{R}{\gamma-1}\Phi\left(\frac{\theta}{\bar{\theta}}\right)\right] +(\rd_t + u\rd_x)\left[\frac{|u-\bar{u}|^2}{2}\right]\right)}_{=:\mathcal{J}} \,dx.
	\end{aligned}
\end{align}
Using the relation 
\begin{align*}
	(\partial_{t} + u\partial_x) \bar{\theta}= \bar{\theta}_x(-\sigma - \dot{\mathbf{X}}(t)) + u\bar{\theta}_x = (u-\sigma)\bar{\theta}_x - \dot{\mathbf{X}}(t)\bar{\theta}_x,
\end{align*}
we have
\begin{align*}
	&\int_{\mathbb{R}_+}a \rho \Big((\partial_{t} + u\partial_x) \bar{\theta}\Big) \left[R\Phi\left(\frac{\bar{\rho}}{\rho}\right)+\frac{R}{\gamma-1} \Phi\left(\frac{\theta}{\bar{\theta}}\right)\right]\,dx \\
	&\, = \int_{\mathbb{R}_+}a \rho \big(u-\sigma\big)\bar{\theta}_x\left[R\Phi\left(\frac{\bar{\rho}}{\rho}\right)+\frac{R}{\gamma-1} \Phi\left(\frac{\theta}{\bar{\theta}}\right)\right]\,dx -\dot{\mathbf{X}}(t)\int_{\mathbb{R}_+}a \rho\bar{\theta}_x \left[R\Phi\left(\frac{\bar{\rho}}{\rho}\right)+\frac{R}{\gamma-1} \Phi\left(\frac{\theta}{\bar{\theta}}\right)\right]\,dx.
\end{align*} 
Now, we split $\mathcal{J}$ into three terms as follows: 
\begin{align*}
	\mathcal{J} &=R a \rho \bar{\theta}(\partial_{t} + u\partial_x)\Phi\left(\frac{\bar{\rho}}{\rho}\right) + \frac{R}{\gamma-1}  a \rho \bar{\theta}(\partial_{t} + u\partial_x)\Phi\left(\frac{\theta}{\bar{\theta}}\right)+ a \rho (\partial_{t} + u\partial_x)\frac{(u-\bar{u})^2}{2} \\
	&=:\mathcal{J}_1 + \mathcal{J}_2 + \mathcal{J}_3.
\end{align*}
To estimate $\mathcal{J}_{1}$, using $\Phi'(z)=1-\frac{1}{z}$, \eqref{eq:NSF}$_1$, \eqref{eq:NSFs}$_1$, and the chain-rule, observe that
\begin{align*}
	(\rd_t + u\rd_x)\Phi\left(\frac{\bar{\rho}}{\rho}\right) &= \Phi'\left(\frac{\bar{\rho}}{\rho}\right) \left[(\rd_t + u\rd_x)\left(\frac{\bar{\rho}}{\rho}\right) \right] =\frac{\rhotil - \r}{\rhotil} \frac{\r (\rd_t + u\rd_x)\rhotil  - \rhotil (\rd_t + u\rd_x)\r}{\r^2}\\
	&=\frac{\rhotil-\rho}{\bar{\rho}} \frac{\rho\left(-(\bar{\rho}\bar{u})_x -\dot{\mathbf{X}}(t)\bar{\rho}_x + u\bar{\rho}_x \right) - \bar{\rho}(-\rho u_x)}{\rho^2}  \\
	&= \frac{\rhotil-\rho}{\bar{\rho}} \frac{\bar{\rho}(u-\bar{u})_x +  \bar{\rho}_x(u-\bar{u})- \dot{\mathbf{X}}(t)\bar{\rho}_x}{\rho}.
\end{align*}
This implies
\begin{equation}\label{eq:j1}
	\begin{aligned}
		\mathcal{J}_1&=- R a\bar{\theta} (\rho -\bar{\rho})(u-\bar{u})_x - Ra\bar{\theta} \frac{\bar{\rho}_x}{\bar{\rho}}(\rho -\bar{\rho }) (u-\bar{u}) + \dot{\mathbf{X}}(t) Ra\bar{\theta}  \bar{\rho}_x  \frac{\rho-\rhotil}{\bar{\rho}} \\
		&= \underbrace{-Ra\bar{\theta} \left((\rho- \bar{\rho})(u-\bar{u})\right)_x}_{=:I_1} +\underbrace{ Ra\bar{\theta}\rho_x (u-\bar{u}) - Ra\bar{\theta} \frac{\rho}{\bar{\rho}} \bar{\rho}_x(u- \bar{u})}_{=:I_2} +\dot{\mathbf{X}}(t) Ra\bar{\theta}  \bar{\rho}_x  \frac{\rho-\rhotil}{\bar{\rho}}.
	\end{aligned}
\end{equation}
Likewise, since
\begin{align*}
	(\partial_{t} + u\partial_x)\Phi\left(\frac{\theta}{\bar{\theta}}\right) &=  \frac{\theta- \bar{\theta}}{\theta} \frac{\bar{\theta}(\theta_t + u\theta_x)- \theta (\bar{\theta}_t + u \bar{\theta}_x)}{\bar{\theta}^2} \\
	&= \frac{\theta- \bar{\theta}}{\theta} \left[\frac{1}{\bar{\theta}}(\theta-\bar{\theta})_t - \frac{\theta- \bar{\theta}}{\bar{\theta}^2} \bar{\theta}_t + \frac{u \theta_x}{\bar{\theta}} - \frac{u \theta \bar{\theta}_x}{\bar{\theta}^2}\right],
\end{align*}
we obtain
\begin{align}\label{eq:j22}
	\begin{aligned}
		\mathcal{J}_{2} &= \frac{R}{\gamma-1} a \rho \frac{\theta-\bar{\theta}}{\theta } (\theta- \bar{\theta})_t -  \frac{R}{\gamma-1} a \rho \frac{(\theta- \bar{\theta})^2}{\theta \bar{\theta}} \bar{\theta}_t +  \frac{R}{\gamma-1} a \rho \frac{\theta- \bar\theta}{\theta} \left(u(\theta-\bar{\theta})_x-u \frac{\theta -\bar{\theta}}{\bar{\theta}} \bar{\theta}_x\right) \\
		&=: \mathcal{J}_{21} +  \mathcal{J}_{22} + \frac{R}{\gamma-1} a \rho u \frac{\theta- \bar\theta}{\theta} (\theta-\bar{\theta})_x - \frac{R}{\gamma-1} a \rho u \frac{(\theta - \thtil)^2}{\theta \bar{\theta}} \bar{\theta}_x.
	\end{aligned}
\end{align}
For $\mathcal{J}_{21}$ and $\mathcal{J}_{22}$, we use $\eqref{eq:pertubation1}_3$ and \eqref{eq:NSFs}$_3$ to have
\begin{align}\label{eq:j221}
	\begin{aligned}
			\mathcal{J}_{21} 
		&= a \rho\frac{\theta-\bar{\theta}}{\theta } \bigg[- \frac{R}{\gamma-1} (u\theta_x -\bar{u}\bar{\theta}_x) + \left(-\frac{p}{\rho}u_x + \frac{\bar{p}}{\bar{\rho}}\bar{u}_x + \frac{\kappa}{\rho}(\theta-\bar{\theta})_{xx} + \kappa (\frac{1}{\rho} -\frac{1}{\bar{\rho}})\bar{\theta}_{xx} + \frac{\mu}{\rho}u^2_x -\frac{\mu}{\bar{\rho}}\bar{u}^2_x\right)  \\
		&\phantom{.= a \rho\frac{\theta-\bar{\theta}}{\theta } \bigg[} + \frac{R}{\gamma-1}\dot{\mathbf{X}}(t) \bar{\theta}_x \bigg],
	\end{aligned}
\end{align}
and
\begin{align}\label{eq:j222}
	\begin{aligned}
			\mathcal{J}_{22}
		= a \rho \frac{(\theta -\bar{\theta})^2}{\theta \bar{\theta}} \left[\frac{R}{\gamma-1} \bar{u}\bar{\theta}_x + \frac{\bar{p}}{\bar{\rho}}\bar{u}_x - \frac{\kappa}{\bar{\rho}} \bar{\theta}_{xx} - \frac{\mu}{\bar{\rho}}(\bar{u}_x)^2 + \frac{R}{\gamma -1}\dot{\mathbf{X}}(t) \bar{\theta}_x\right].
	\end{aligned}
\end{align}
Substituting \eqref{eq:j221} and \eqref{eq:j222} into \eqref{eq:j22}, we have 
\begin{align}\label{eq:j2}
	\begin{aligned}
		\mathcal{J}_2
		&= -\frac{R}{\gamma -1} a \rho \frac{\theta -\bar{\theta}}{\theta} (u-\bar{u})\bar{\theta}_x \underbrace{ - Ra \rho (\theta -\bar{\theta})(u-\bar{u})_x}_{=:I_3}   \\
		&\quad + \left[\kappa \frac{a}{\theta}(\theta - \bar{\theta})(\theta - \bar{\theta})_x \right]_x - \kappa \frac{a}{\theta} |(\theta -\bar{\theta})_x|^2- \kappa \left(\frac{a}{\theta}\right)_x (\theta - \bar{\theta})(\theta - \bar{\theta})_x  + \kappa a \rho \frac{\theta -\bar{\theta}}{\theta}\left(\frac{1}{\rho} - \frac{1}{\bar{\rho}}  \right)\bar{\theta}_{xx}  \\
		&\quad - \frac{R}{\gamma-1} a \rho (u-\bar{u}) \frac{(\bar{\theta} - \theta)^2}{\theta \bar{\theta}} \bar{\theta}_x   + a\rho \frac{\theta- \bar{\theta}}{\theta}\bigg[ \frac{\mu}{\rho }(u_x^2 - \bar{u}_x^2)+ \mu \bar{u}^2_x \left(\frac{1}{\rho} -\frac{1}{\bar{\rho}}\right)\bigg]  \\
		&\quad-\kappa a\frac{\r}{\rhotil} \frac{(\th - \thtil)^2}{\th \thtil}\thtil_{xx} - \mu a\frac{\r}{\rhotil}\frac{(\th - \thtil)^2}{\th \thtil}\util_x^2 + \frac{R}{\gamma -1} a\rho \frac{\theta -\bar{\theta}}{\bar{\theta}} \bar{\theta}_x \dot{\mathbf{X}}(t).
	\end{aligned}
\end{align}
For $\mathcal{J}_3$, we use $\eqref{eq:pertubation1}_2$ to have
\begin{align*}
	(\rd_t + u\rd_x)\frac{(u-\util)^2}{2} &= (u-\util)\Big((\rd_t + u\rd_x)(u-\util)\Big)\\
	&= (u-\bar{u}) \bigg[-(u-\bar{u})\bar{u}_x +  \dot{\mathbf{X}}(t) \bar{u}_x + \frac{\mu}{\rho} (u-\bar{u})_{xx} + \mu (\frac{1}{\rho} - \frac{1}{\bar{\rho}})\bar{u}_{xx} \\
	&\phantom{= (u-\bar{u}) \bigg[} - R\big(\frac{\rho_x\theta}{\rho} - \frac{\bar{\rho}_x \bar{\theta}}{\bar{\rho}}\big) - R(\theta -\bar{\theta})_x\bigg].
\end{align*}
Thus, we have
\begin{align}\label{eq:j3}
	\begin{aligned}
		\mathcal{J}_3 &=  -a\rho \bar{u}_x (u-\bar{u})^2 \underbrace{- R a\rho (u-\bar{u})\left(\frac{\rho_x\theta}{\rho} - \frac{\bar{\rho}_x \bar{\theta}}{\bar{\rho}}\right)}_{=:I_4} \underbrace{- Ra\rho (u-\bar{u})(\theta- \bar{\theta})_x}_{=:I_5}  \\
		&\quad  +\left[ \mu a(u-\bar{u})(u-\bar{u})_{x}\right]_x - \mu a|(u-\bar{u})_x|^2 - \mu a' (u-\bar{u})(u-\bar{u})_x   \\
		&\quad + \mu a \rho (u-\bar{u})\left(\frac{1}{\rho} -\frac{1}{\bar{\rho}} \right)\bar{u}_{xx} + \dot{\mathbf{X}}(t) a\rho \bar{u}_x(u-\bar{u}) .
	\end{aligned}
\end{align}
Now, observe that 
\[
\int_{\RR_+} I_1\,dx = \int_{\RR_+} Ra_x \thtil (\r - \rhotil)(u-\util)\,dx + \int_{\RR_+} R a \thtil_x (\r - \rhotil)(u-\util)\,dx + R a \thtil (\r-\rhotil)(u-\util)\Big|_{x=0}.
\]
On the other hand, we have
\[
\int_{\RR_+} \big(I_2 + I_4\big) \,dx = -\int_{\RR_+} R a \r_x(u-\util) (\th - \thtil)\,dx.
\]
This implies 
\begin{align*}
	\int_{\RR_+} \big(I_2 + I_3 + I_4 + I_5\big)\,dx &= -\int_{\RR_+} R a \Big[\r_x(u-\util)(\th- \thtil) + \r(u-\util)_x(\th- \thtil)+ \r(u-\util)(\th- \thtil)_x\Big]\,dx\\
	&=\int_{\RR_+} R a_x \r(u-\util)(\th- \thtil) + R a \r(u-\util)(\th- \thtil)\Big|_{x=0}.
\end{align*}
Thus, we have
\begin{equation}\label{eq:I1-5}
\begin{aligned}
	\sum_{i=1}^5 \int_{\RR_+} I_i \,dx &= \int_{\RR_+} Ra_x (u-\util)\big(\thtil (\r - \rhotil) + \r (\th - \thtil)\big)\,dx + \int_{\RR_+} R a \thtil_x (\r - \rhotil)(u-\util)\,dx +\mathcal{P}_4 + \mathcal{P}_5\\
	&=  \int_{\RR_+} a_x (u-\util)(p-\bar{p})\,dx + \int_{\RR_+} R a \thtil_x (\r - \rhotil)(u-\util)\,dx +\mathcal{P}_4 + \mathcal{P}_5.
\end{aligned}
\end{equation}
Substituting \eqref{eq:j1}, \eqref{eq:j2}, \eqref{eq:j3}, and \eqref{eq:I1-5} into \eqref{eq:REalculation}, and using integration by parts, we obtain the desired result.
\end{proof}

\subsection{Decompositions}

Here, we extract the leading-order bad terms from the first five terms of $\mathcal{J}^{bad}$ as below:
\begin{align}\label{def:B1}
	\begin{aligned}
			\mathbf{B}_1 &:= \int_{\mathbb{R}_+} a |\bar{u}_x| \bigg[\frac{R(\gamma -1) \theta_-}{2\rho_-} |\rho - \bar{\rho}|^2 +\rho_- |u-\bar{u}|^2 + \frac{R\rho_-}{2\theta_-} |\theta - \bar{\theta}|^2  + \frac{R(\gamma-1) \theta_-}{\sqrt{R\gamma \theta_-}}|\rho- \bar{\rho}||u-\bar{u}|    \\
		&\phantom{:= \int_{\mathbb{R}_+} a |\bar{u}_x| \bigg[} + \frac{R \rho_-}{\sqrt{R\gamma \theta_-}}|u-\bar{u}||\theta- \bar{\theta}|\bigg] \,dx.
	\end{aligned}
\end{align}
Since $ |\rho -\rho_- | \leq |\rho -\bar{\rho}| + |\bar{\rho} - \rho_-| \leq C(\varepsilon +  \delta)$, we have 
\begin{align*}
	\begin{aligned}
		& 	-a \rho \bar{u}_x (u-\bar{u})^2 \leq a \rho_- |\bar{u}_x| |u-\bar{u}|^2 + C (\delta + \epsilon) |\bar{u}_x| |u-\bar{u}|^2 .
	\end{aligned}
\end{align*}
Using \eqref{estderi}, we obtain
\begin{align*}
	\begin{aligned}
	\big| R a\bar{\theta}_x (\rho -\bar{\rho})(u-\bar{u}) \big|
	\leq  \frac{R(\gamma -1) \theta_-}{\sqrt{\gamma R \theta_-}} a |\bar{u}_x| |\rho -\bar{\rho}| |u- \bar{u}| + C\delta |\bar{u}_x| \left(|\rho -\bar{\rho}|^2 + |u-\bar{u}|^2\right).
	\end{aligned}
\end{align*}
Similarly, using \eqref{estderi} and $\left|\frac{1}{\thtil} - \frac{1}{\th_-}\right| \leq C(\e + \d)$, we have 
\begin{align*}
	\begin{aligned}
		\left| \frac{R}{\gamma- 1} a\rho \bar{\theta}_x (u-\bar{u}) \frac{\theta -\bar{\theta}}{\theta } \right| \leq  \frac{R\rho_-}{\sqrt{R \gamma \theta_-}} a |\bar{u}_x| |u-\bar{u}| |\theta -\bar{\theta}| + C(\e + \delta)|\bar{u}_x|\left(|u-\bar{u}|^2 + |\theta -\bar{\theta}|^2\right).
	\end{aligned}
\end{align*}
Now, we use  
\begin{equation}\label{approx:u-simga}
	|(u-\sigma) - (u_-- \sigma_-)| \leq |u- u_-| + |\sigma -\sigma_-| \leq (\varepsilon + \delta),
\end{equation}
and Taylor expansion of $\Phi(z) = z -1 - \ln z$ at $z=1$:
\[
\Big|\Phi(z) - \frac{1}{2}(z-1)^2\Big| \leq C|z-1|^3 \quad \text{for} \quad |z-1| \leq \frac{1}{2},
\]
to have
\begin{align}\label{eq:phi}
	\left|\Phi \left(\frac{\bar{\rho}}{\rho}\right) - \frac{(\rho -\bar{\rho})^2}{2\rho^2}\right| \leq C|\rho -\bar{\rho}|^3, \quad \text{and} \quad \left|\Phi\left(\frac{\th}{\thtil}\right) - \frac{(\th - \thtil)^2}{2\thtil^2}\right| \leq C|\th - \thtil|^3.
\end{align}
Thus, we obtain
\begin{align*}
	\begin{aligned}
	\left|R a\rho \bar{\theta} (u-\sigma)\bar{\theta}_x  \Phi \left(\frac{\bar{\rho}}{\rho}\right)\right| \leq  \frac{R(\gamma -1) \theta_-}{2\rho_-} a|\bar{u}_x||\rho - \bar{\rho}|^2 + C ( \varepsilon + \delta) |\bar{u}_x| |\rho -\bar{\rho}|^2,
	\end{aligned}
\end{align*}
and 
\begin{align*}
	\begin{aligned}
		\left|\frac{R}{\gamma -1}a \rho (u-\sigma) \bar{\theta}_x \Phi\left(\frac{\th}{\thtil}\right) \right| 
		\leq  \frac{R\rho_-}{2\theta_-}a |\bar{u}_x| |\theta -\bar{\theta}|^2 + C(\varepsilon + \delta)  |\bar{u}_x| |\theta -\bar{\theta}|^2.
	\end{aligned}
\end{align*}
Combining the above estimates, we have
\begin{align*}
	\begin{aligned}
		& \int_{\mathbb{R}_+} a\bigg[ -\rho \bar{u}_x (u-\bar{u})^2 + R \bar{\theta}_x (\rho -\bar{\rho})(u-\bar{u}) \\
		 &\phantom{\int_{\mathbb{R}_+} a\bigg[} + \rho (u-\sigma)\bar{\theta}_x \bigg(R\bar{\theta} \Phi\left(\frac{\bar{\rho}}{\rho}\right) + \frac{R}{\gamma-1}\bar{\theta} \Phi\left(\frac{\theta}{\bar{\theta}}\right)\bigg) -\frac{R}{\gamma- 1} \rho \bar{\theta}_x (u-\bar{u}) \frac{\theta -\bar{\theta}}{\theta }\bigg] \,dx \\
		&\quad \leq \mathbf{B}_1 + C(\delta + \varepsilon) \int_{\mathbb{R}_+} |\bar{u}_x| \left|(\rho -\bar{\rho}, u-\bar{u}, \theta -\bar{\theta})\right|^2 \,dx.
	\end{aligned}
\end{align*}
Therefore, it follows from Lemma \ref{lma:BRE} that
\begin{align}\label{est:REPri}
	\begin{aligned}
	\frac{d}{dt} \int_{\mathbb{R}_+} a(t,x) \bar{\theta} (t,x)\eta(U(t,x)|\bar{U}(t,x)) dx  \leq \dot{\mathbf{X}}(t) \mathbf{Y}(U) + \sum_{i=1}^{5} \mathbf{B}_i  - \mathbf{G}(U) - \mathbf{D}(U) + \mathcal{P}(U),
	\end{aligned}
\end{align}
where $\mathbf{X}, \mathbf{Y}, \mathcal{P}$ are defined in \eqref{def_shift}, \eqref{term:Y}, and \eqref{term:Jgood}, respectively, $\mathbf{B}_1$ is as in \eqref{def:B1}, and 
\begin{align*}
	\begin{aligned}
		\mathbf{B}_2 &:= \int_{\mathbb{R}_+} a_x (u-\bar{u})(p-\bar{p}) \,dx, \quad \mathbf{B}_3 := -\int_{\mathbb{R}_+} a_x\left[\mu (u-\bar{u})(u-\bar{u})_x + \kappa \frac{\theta -\bar{\theta}}{\theta} (\theta -\bar{\theta})_x\right] \,dx, \\
		\mathbf{B}_4 &:= \int_{\mathbb{R}_+} a\bigg[-\frac{R}{\gamma-1} \rho \bar{\theta}_x (u-\bar{u}) \frac{(\theta- \bar{\theta})^2}{\theta \bar{\theta}} + \mu \rho \bar{u}_{xx} (u-\bar{u}) (\frac{1}{\rho} - \frac{1}{\bar{\rho}}) + \mu \frac{\theta -\bar{\theta}}{\theta} (u^2_x - \bar{u}^2_x )   \\
		&\phantom{:=-\int_{\mathbb{R}_+} a\bigg[} + \mu \rho \bar{u}^2_x \frac{\theta - \bar{\theta}}{\theta} (\frac{1}{\rho} - \frac{1}{\bar{\rho}}) - \kappa \frac{\rho}{\bar{\rho}} \frac{(\theta -\bar{\theta})^2}{\theta \bar{\theta}} \bar{\theta}_{xx}  - \mu  \frac{\rho}{\bar{\rho}}  \frac{(\theta -\bar{\theta})^2}{\theta \bar{\theta}} \bar{u}_x^2 + \kappa \rho \frac{\theta -\bar{\theta}}{\theta} (\frac{1}{\rho } -\frac{1}{\bar{\rho}}) \bar{\theta}_{xx}   \\
		&\phantom{:=-\int_{\mathbb{R}_+} a\bigg[} + \kappa \frac{\theta_x}{\theta^2} (\theta -\bar{\theta}) (\theta -\bar{\theta})_x\bigg] \,dx, \\
		\mathbf{B}_5 &:= C(\delta + \e) \int_{\mathbb{R}_+} |\bar{u}_x| \left|(\rho- \bar{\rho}, u-\bar{u}, \theta -\bar{\theta})\right|^2 \,dx,
	\end{aligned}
\end{align*}
and
\begin{equation}\label{termG}
	\begin{aligned}
	\mathbf{G}(U) &:= \int_{\mathbb{R}_+} (\sigma -u) a_x \bar{\theta} \eta (U|\bar{U})  \,dx,\\
	\mathbf{D}(U) &:= \int_{\mathbb{R}_+} a \left(\mu |(u-\bar{u})_x|^2 + \frac{\kappa}{\theta}|(\theta -\bar{\theta})_x|^2 \right) \,dx = \mathbf{D}_{u_1}(U) + \mathbf{D}_{\theta_1}(U).
\end{aligned} 
\end{equation}
Now, we decompose the functional $\mathbf{Y}$ as:
\[\mathbf{Y}(U):=\sum\limits_{i=1}^6\mathbf{Y}_i(U),\]
where
\begin{align*}
\begin{split}
	\begin{aligned}
		&\mathbf{Y}_1(U):=\int_{\mathbb{R}_+}a\rho(u-\bar{u})\bar{u}_x\,dx,
		&&\mathbf{Y}_2(U):=R\int_{\mathbb{R}_+}a\frac{\bar{\theta}}{\bar{\rho}}(\rho-\bar{\rho})\bar{\rho}_x\,dx,\\
		&\mathbf{Y}_3(U):=\frac{R}{\gamma-1}\int_{\mathbb{R}_+}a\frac{\rho}{\bar{\theta}}(\theta-\bar{\theta})\bar{\theta}_x\,dx,
		&&\mathbf{Y}_4(U):=-R\int_{\mathbb{R}_+}a\rho\bar{\theta}\Phi\left(\frac{\bar{\rho}}{\rho}\right)\bar{\theta}_x\,dx,\\
		&\mathbf{Y}_5(U):=-\frac{R}{\gamma-1}\int_{\mathbb{R}_+}a\rho\bar{\theta} \Phi\left(\frac{\theta}{\bar{\theta}}\right)\bar{\theta}_x\,dx,
		&&\mathbf{Y}_6(U):=-\int_{\mathbb{R}_+}a_x\bar{\theta}\eta\left(U|\bar{U}\right)\,dx.
	\end{aligned}
\end{split}
\end{align*}
Note from \eqref{def_shift} that
\begin{equation} \label{XY123}
	\dot{\mathbf{X}}(t)=-\frac{M}{\delta}\left(\mathbf{Y}_1+\mathbf{Y}_2+\mathbf{Y}_3\right).
\end{equation}
This implies
\begin{equation}\label{eq:XY}
	\dot{\mathbf{X}}\mathbf{Y}=-\frac{\delta}{M}|\dot{\mathbf{X}}|^2+\dot{\mathbf{X}}\sum\limits_{i=4}^6\mathbf{Y}_i.    
\end{equation}

\subsection{Estimates of leading-order terms}\label{sec:lot}
\begin{lemma}\label{lma:leading}
	Under the assumption of Proposition \ref{prop:main} (or Proposition \ref{prop:main2}), there exists $C_*>0$ such that 
\begin{align*}
	-\frac{\delta}{2M}|\dot{\mathbf{X}}|^2+\mathbf{B}_1+\mathbf{B}_2-\mathbf{G}-\frac{7}{8}\mathbf{D}
	\le  -\frac{1}{4}(\mathbf{G}_1 + \mathbf{G}_2) - C_*\mathbf{G}^S,
\end{align*}
where \[\mathbf{G}^S := \int_{\mathbb{R}_+} |\bar{u}_x| \left|\left(\rho -\bar{\rho},  u-\bar{u}, \theta -\bar{\theta}\right)\right|^2 \,dx.\]
\end{lemma}
\begin{proof}
	For a fixed $t \in [0,T]$, we introduce a change of variable $x \mapsto y$ as 
	\begin{equation}\label{ydef}
		y:=\frac{u_--\bar{u}(x-\sigma t-\bX(t)-\beta)}{\delta}.
	\end{equation}
	Indeed, since $\util$ is decreasing, the map $x\mapsto y=y(x)$ is a one-to-one and increasing function satisfying
	\[\frac{dy}{dx} = -\frac{\bar{u}_x(x-\sigma t-\bX(t)-\beta)}{\delta}>0,\quad \lim_{x\to0}y=y_0(t),\quad \lim_{x\to+\infty}y=1,\]
	where
	\[y_0(t) := \frac{u_- -\bar{u}(-\sigma t-\bX(t)-\beta)}{\delta}>0.\]

	Recall that for the outflow problem, the assumption $(\rho_+, u_+, \theta_+) \in  \Omega_{sub}^- \cup \Gamma_{trans}^- $ together with the Lax entropy condition $\lambda_3(\rho_-,u_-,\th_-) > \sigma > \lambda_3(\rho_+,u_+,\th_+)$ implies $\sigma >0$. On the other hand, in the impermeable setting, the Rankine-Hugoniot condition \eqref{eq:RH}$_1$ immediately gives $\sigma >0$.

	Using $\sigma > 0$, \eqref{def_shift}, the \textit{a priori} assumption \eqref{apriori_small} with Sobolev inequality, and the smallness of $\varepsilon$, we have
	\[
	|\dot{\bX}(t)|\le \frac{C}{\delta} \|(\rho - \rhotil, u-\util, \theta - \thtil)\|_{L^\infty(\Rp)} \intRp |(\rhotil', \util', \thtil')|\,dx \le C\varepsilon \le \frac{\sigma}{2}, \quad t\le T,
	\]
	and so
	\[
	|{\bX}(t)| \le \frac{\sigma}{2} t, \quad t\le T,
	\]
	which yields
	\begin{equation} \label{bddx0}
		-\sigma t -\bX (t) -\beta \le - \frac{\sigma}{2} t -\beta <0, \quad t\le T.
	\end{equation}
	Using \eqref{shock_prop} and \eqref{bddx0}, we obtain
	\begin{align*} 
		\begin{aligned}
			& |\bar{u}(t,0)- u_-| \leq C\delta e^{-C\delta|-\sigma t-\mathbf{X}(t) -\beta|} \leq C\delta e^{-C\delta t} e^{-C\delta \beta} \leq C \delta e^{-C\delta \beta}.
		\end{aligned}
	\end{align*}
	Thus, we have
	\begin{equation}	\label{y0small}
		y_0(t) \le Ce^{-C\delta |\sigma t+\bX(t)+\beta|}\le Ce^{-C\delta\beta}.
	\end{equation}
	Thanks to \eqref{y0small}, by choosing $\beta$ sufficiently large, we can ensure that $y_0(t) < \frac{1}{8}$ for all $t \in [0,T]$.

	Moreover, for convenience, we introduce the notation
	\[w(y):=(u(t,\cdot)-\bar{u}(\cdot-\sigma t-\bX(t)-\beta))\circ y^{-1}.\]
	For later analysis, we observe that the weight function $a$ in \eqref{def_a} satisfies $a(t,x)=1+\sqrt{\delta} y$, and 
	\begin{align}\label{eq:ax}
		a_x=\sqrt{\delta} \frac{dy}{dx}= -\frac{1}{\sqrt{\delta}}\bar{u}_x.
	\end{align}

	\noindent $\bullet$ \textbf{Estimates on $\mathbf{B}_2 -\mathbf{G}$:} 
	First, notice that 
\begin{align*}
	\begin{aligned}
		& (u-\bar{u})(p-\bar{p}) -(\sigma- u) \bar{\theta} \eta(U|\bar{U})  \\
		&\quad = (u-\bar{u})(p-\bar{p}) - (\sigma- u) \rho\left[R\bar{\theta} \Phi\left(\frac{\bar{\rho}}{\rho}\right)+\frac{R}{\gamma-1}\bar{\theta} \Phi\left(\frac{\theta}{\bar{\theta}}\right)+\frac{1}{2}(u-\bar{u})^2\right].
	\end{aligned}
\end{align*}
Using $p-\bar{p} = R\rho (\theta-\bar{\theta}) + R\bar{\theta}(\rho -\bar{\rho})$, \eqref{approx:u-simga}, and \eqref{eq:phi}, we have 
\begin{align*}
	& (u-\bar{u})(p-\bar{p}) -(\sigma- u) \bar{\theta} \eta(U|\bar{U})   \\
	&\, \,  = (u-\bar{u}) \left[R\rho (\theta-\bar{\theta}) + R\bar{\theta}(\rho -\bar{\rho})\right]- (\sigma- u) \rho\left[R\bar{\theta} \Phi\left(\frac{\bar{\rho}}{\rho}\right)+\frac{R}{\gamma-1}\bar{\theta} \Phi\left(\frac{\theta}{\bar{\theta}}\right)+\frac{1}{2}(u-\bar{u})^2\right]  \\
	&\, \, \leq  (u-\bar{u}) \left[R\rho_- (\theta-\bar{\theta}) + R\theta_-(\rho -\bar{\rho})\right]- \frac{R\theta_-}{2\rho_-}\sqrt{\gamma R\theta_-}(\rho -\rhotil)^2 - \frac{R\rho_-}{2(\gamma-1)\theta_-}\sqrt{\gamma R\theta_-}(\theta-\bar{\theta})^2 \\
	&\quad -\frac{\rho_-}{2}\sqrt{\gamma R\theta_-}(u-\bar{u})^2 + C\left(|\rho-\rho_-| + |\bar{\theta} - \theta_-| + |u - u_-|\right)\left|(\rho- \bar{\rho}, u-\bar{u}, \theta- \bar{\theta})\right|^2 \\
	&\quad + C\Big(|\r -\rhotil|^3 + |\th - \thtil|^3\Big)\\
	&\, \, \leq -\frac{R\theta_-\sqrt{\gamma R \theta_-}}{2\rho_-}\left[(\rho- \bar{\rho})- \frac{\rho_-}{\sqrt{\gamma R \theta_-}}(u-\bar{u})\right]^2 - \frac{R\rho_-\sqrt{\gamma R\theta_-}}{2(\gamma-1)\theta_-}\left[(\theta-\bar{\theta}) - \frac{(\gamma-1)\theta_-}{\sqrt{\gamma R \theta_-}}(u-\bar{u})\right]^2  \\
	&\quad + C\left(|\rho-\rho_-| + |\bar{\theta} - \theta_-| + |u - u_-|\right)\left|(\rho- \bar{\rho}, u-\bar{u}, \theta- \bar{\theta})\right|^2 + C\Big(|\r -\rhotil|^3 + |\th - \thtil|^3\Big).
\end{align*}
Here, the last equality holds because
\begin{align*}
	\frac{R\theta_-}{2\rho_-}\sqrt{\gamma R\theta_-} \frac{\rho^2_-}{\gamma R \theta_-} + \frac{R\rho_-}{2(\gamma-1) \theta_-}\sqrt{\gamma R \theta_-} \frac{(\gamma-1)^2 \theta_-}{\gamma R} = \frac{\rho_-}{2}\sqrt{\gamma R \theta_-}.
\end{align*}
Therefore, using $|\r - \r_-|+|\thtil - \th_-| + |u - u_-| \leq |\r - \rhotil|+|\th - \thtil| + |u - \util| + C\delta$, we have 
\begin{align*}
	\mathbf{B}_2 - \mathbf{G} \leq -\mathbf{G}_1 - \mathbf{G}_2 + \mathbf{B}_{new},
\end{align*}
where
\begin{align*}
	\begin{aligned}
		& \mathbf{G}_1 = \frac{R\theta_-}{2\rho_-}\sqrt{\gamma R \theta_-} \int_{\mathbb{R}_+} a_x\left[(\rho- \bar{\rho})- \frac{\rho_-}{\sqrt{\gamma R \theta_-}}(u-\bar{u})\right]^2 \,dx, \\
		& \mathbf{G}_2 = \frac{R\rho_-}{2(\gamma-1)\theta_-}\sqrt{\gamma R\theta_-} \int_{\mathbb{R}_+} a_x\left[(\theta-\bar{\theta}) - \frac{(\gamma-1)\theta_-}{\sqrt{\gamma R \theta_-}}(u-\bar{u})\right]^2 \,dx, \\
		& \mathbf{B}_{new} = C \delta \int_{\mathbb{R}_+} a_x \left|(\rho -\bar{\rho}, u-\bar{u}, \theta -\bar{\theta})\right|^2 \,dx + C\int_{\mathbb{R}_+} a_x  \left|(\rho -\bar{\rho}, u-\bar{u}, \theta -\bar{\theta})\right|^3 \,dx.
	\end{aligned}
\end{align*}
The two good terms $\mathbf{G}_1$ and $\mathbf{G}_2$ will be used in the subsequent analysis. For $\mathbf{B}_{new}$, observe first that
\begin{align*}
	 \delta \int_{\mathbb{R}_+} a_x \left|(\rho -\bar{\rho}, u-\bar{u}, \theta -\bar{\theta})\right|^2 \,dx \leq  C\sqrt{\delta} \int_{\mathbb{R}_+} |\bar{u}_x| \left|(\rho -\bar{\rho}, u-\bar{u}, \theta -\bar{\theta})\right|^2 \,dx \leq C\sqrt{\delta} \mathbf{G}^S.
\end{align*}
For the cubic terms in $\mathbf{B}_{new}$, we use the interpolation inequality and Young's inequality to have
\begin{align*}
	&\int_{\mathbb{R}_+} a_x  \left|(\rho -\bar{\rho}, u-\bar{u}, \theta -\bar{\theta})\right|^3 \,dx  \\
	&\quad \leq C \int_{\mathbb{R}_+} a_x\left[\left|(\rho- \bar{\rho})- \frac{\rho_-}{\sqrt{\gamma R \theta_-}}(u-\bar{u})\right|^3 + \left|(\theta-\bar{\theta}) - \frac{(\gamma-1)\theta_-}{\sqrt{\gamma R \theta_-}}(u-\bar{u})\right|^3 +  \left|u-\bar{u}\right|^3\right] \,dx \\
	&\quad \leq C\varepsilon(\mathbf{G}_1 + \mathbf{G}_2 ) + C\frac{1}{\sqrt{\delta}}  \int_{\mathbb{R}_+} |\bar{u}_x| \|w\|^2_{L^{\infty}(\mathbb{R}_+)} |w| \,dx \\
	&\quad \leq   C\varepsilon(\mathbf{G}_1 + \mathbf{G}_2 ) + C\frac{1}{\sqrt{\delta}}  \|w\|_{L^2(\mathbb{R}_+)} \|w_x\|_{L^2(\mathbb{R}_+)} \left(\int_{\mathbb{R}_+} |\bar{u}_x| |w|^2 \,dx\right)^{\frac{1}{2}} \left(\int_{\mathbb{R}_+} |\bar{u}_x| \,dx\right)^{\frac{1}{2}}   \\
	&\quad \leq C \varepsilon (\mathbf{G}_1 + \mathbf{G}_2  + \mathbf{D}_{u_1} +\mathbf{G}^S ).
\end{align*}
Thus, we obtain
\begin{align*}
	\mathbf{B}_{new} \leq C\sqrt{\delta} \mathbf{G}^S + C\varepsilon (\mathbf{G}_1 + \mathbf{G}_2  + \mathbf{D}_{u_1} +\mathbf{G}^S ).
\end{align*}

$\bullet$ \textbf{Estimates on $\displaystyle{-\frac{\delta}{2M}|\dot{\mathbf{X}}|^2}$:} To estimate the term $-\frac{\delta}{2M}|\dot{\mathbf{X}}|^2$, we focus on estimating $\mathbf{Y}_1$, $\mathbf{Y}_2$, and $\mathbf{Y}_3$. First, using $|a-1|\le \sqrt{\delta}$, $|\rho-\bar{\rho}|\le C\varepsilon$, and $|\bar{\rho}-\rho_-|\le C\delta $, we obtain
\begin{equation} \label{esty2}
	\left|\mathbf{Y}_1+\delta\rho_-\int_{y_0}^{1}w\,dy\right|\le C\delta(\sqrt{\delta}+\varepsilon)\int_{y_0}^{1}|w|\,dy.    
\end{equation}
Next, we split $\mathbf{Y}_2$ into the following:
\begin{align*}
	\mathbf{Y}_2=&R\int_{\mathbb{R}_+}a\frac{\bar{\theta}}{\bar{\rho}}\left[(\rho-\bar{\rho})-\frac{\rho_-}{\sqrt{\gamma R \theta_-}}(u-\bar{u})\right]\bar{\rho}_x\,dx+\frac{R\rho_-}{\sqrt{\gamma R \theta_-}}\int_{\mathbb{R}_+}a\frac{\bar{\theta}}{\bar{\rho}}(u-\bar{u})\bar{\rho}_x\,dx\\   
	=&\frac{R\rho_-}{\sqrt{\gamma R \theta_-}}\int_{\mathbb{R}_+}a\frac{\bar{\theta}}{\bar{\rho}}(u-\bar{u})\left(\bar{\rho}_x-\frac{\rho_-}{\sqrt{\gamma R \theta_-}}\bar{u}_x\right)\,dx+\frac{\rho_-^2}{\gamma \theta_-}\int_{\mathbb{R}_+}a\frac{\bar{\theta}}{\bar{\rho}}(u-\bar{u})\bar{u}_x\,dx
	\\&+R\int_{\mathbb{R}_+}a\frac{\bar{\theta}}{\bar{\rho}}\left[(\rho-\bar{\rho})-\frac{\rho_-}{\sqrt{\gamma R \theta_-}}(u-\bar{u})\right]\bar{\rho}_x\,dx.\\
\end{align*}
Using \eqref{estderi}, we have
\begin{align}
	\begin{aligned}\label{esty1}
		\left|\mathbf{Y}_2+\delta\frac{\rho_-}{\gamma}\int_{y_0}^{1}w\,dy\right|	\le C\delta\sqrt{\delta}\int_{y_0}^{1}|w|\,dy+C\sqrt{\delta}\int_{\mathbb{R}_+}|a_x|\left|\rho-\bar{\rho}-\frac{\rho_-}{\sqrt{\gamma R \theta_-}}(u-\bar{u})\right|\,dx.   
	\end{aligned}
\end{align}

Finally, we decompose $\mathbf{Y}_3$ as follows:
\begin{align*}
	\mathbf{Y}_3=&\frac{R}{\gamma-1}\int_{\mathbb{R}_+}a\frac{\rho}{\bar{\theta}}\left[\theta-\bar{\theta}-\frac{(\gamma-1)\theta_-}{\sqrt{\gamma R \theta_-}}(u-\bar{u})\right]\bar{\theta}_x\,dx+\frac{R\theta_-}{\sqrt{\gamma R\theta_-}}\int_{\mathbb{R}_+}a\frac{\rho}{\bar{\theta}}(u-\bar{u})\bar{\theta}_x\,dx\\
	=&\frac{R\theta_-}{\sqrt{\gamma R\theta_-}}\int_{\mathbb{R}_+}a\frac{\rho}{\bar{\theta}}(u-\bar{u})\left(\bar{\theta}_x-\frac{(\gamma-1)\theta_-}{\sqrt{\gamma R \theta_-}}\bar{u}_x\right)\,dx+\frac{(\gamma-1)\theta_-}{\gamma}\int_{\mathbb{R}_+}a\frac{\rho}{\bar{\theta}}(u-\bar{u})\bar{u}_x\,dx\\
	&+\frac{R}{\gamma-1}\int_{\mathbb{R}_+}a\frac{\rho}{\bar{\theta}}\left[\theta-\bar{\theta}-\frac{(\gamma-1)\theta_-}{\sqrt{\gamma R \theta_-}}(u-\bar{u})\right]\bar{\theta}_x\,dx.
\end{align*}

From \eqref{estderi}, we get
\begin{align}
	\begin{aligned}\label{esty3}
		\left|\mathbf{Y}_3+\delta\frac{(\gamma-1)\rho_-}{\gamma}\int_{y_0}^{1}wdy\right| \le C\delta(\sqrt{\delta}+\varepsilon)\int_{y_0}^{1}|w|dy+C\sqrt{\delta}\int_{\mathbb{R}_+}|a_x|\left|\theta-\bar{\theta}-\frac{(\gamma-1)\theta_-}{\sqrt{\gamma R \theta_-}}(u-\bar{u})\right|dx.   
	\end{aligned}
\end{align}
Therefore, using \eqref{XY123}, \eqref{esty2}, \eqref{esty1}, and \eqref{esty3}, we have 
\begin{align*}
	\begin{aligned}
		\left| \dot{\mathbf{X}} - 2\rho_- M\int_{y_0}^{1}w\,dy\right| \leq & C (\sqrt{\d}+\varepsilon)\int_{y_0}^{1}|w|\,dy+C\frac{1}{\sqrt{\delta}}\int_{\mathbb{R}_+}|a_x|\left|\rho-\bar{\rho}-\frac{\rho_-}{\sqrt{\gamma R \theta_-}}(u-\bar{u})\right|\,dx\\
		&+C\frac{1}{\sqrt{\delta}}\int_{\mathbb{R}_+}|a_x|\left|\theta-\bar{\theta}-\frac{(\gamma-1)\theta_-}{\sqrt{\gamma R \theta_-}}(u-\bar{u})\right|\,dx,
	\end{aligned}
\end{align*}
which yields
\begin{align*}
	\left(\left|2\rho_- M\int_{y_0}^{1}w\,dy\right| - |\dot{\mathbf{X}}| \right)^2 \leq C (\sqrt{\d}+\varepsilon)^2\int_{y_0}^{1}|w|^2\,dy+C\frac{1}{\delta}(\mathbf{G}_1 + \mathbf{G}_2)\int_{\mathbb{R}_+}|a_x| \,dx.
\end{align*}
Thanks to the algebraic inequality $\frac{p^2}{2}-q^2 \le (p-q)^2$ for $p,q \in \mathbb{R}$, we obtain
\begin{align}\label{est:goodx}
	-\frac{\delta}{2M}|\dot{\mathbf{X}}|^2\le - M\rho_-^2\delta\left(\int_{y_0}^1w\,dy\right)^2 +C\delta(\sqrt{\delta}+\varepsilon)^2 \int_{y_0}^{1}|w|^2\,dy +C\sqrt{\delta} (\mathbf{G}_1+ \mathbf{G}_2). 
\end{align}
	\noindent $\bullet$ \textbf{Estimates on $\mathbf{B}_1$:}	
	We now decompose the leading-order bad term $\mathbf{B}_1$ as follows:
	\begin{align}\label{eq:B1}
		\begin{aligned}
			\mathbf{B}_1 &:= \int_{\mathbb{R}_+} a |\bar{u}_x| \bigg[\frac{R(\gamma -1) \theta_-}{2\rho_-} |\rho - \bar{\rho}|^2 +\rho_- |u-\bar{u}|^2 + \frac{R\rho_-}{2\theta_-} |\theta - \bar{\theta}|^2  + \frac{R(\gamma-1) \theta_-}{\sqrt{R\gamma \theta_-}}|\rho- \bar{\rho}||u-\bar{u}|    \\
			&\phantom{= \int_{\mathbb{R}_+} a |\bar{u}_x| \bigg[} + \frac{R \rho_-}{\sqrt{R\gamma \theta_-}}|u-\bar{u}||\theta- \bar{\theta}|\bigg] \,dx\\
			&=: \mathbf{B}_{11} + \mathbf{B}_{12} + \mathbf{B}_{13}  + \mathbf{B}_{14} + \mathbf{B}_{15}.\\
		\end{aligned}
	\end{align}
Using the inequality $(a+b)^2 \leq (1+\delta^{1/4})a^2 + \left(1+\frac{1}{2}\delta^{-1/4}\right)b^2$ for $a,b \in \mathbb{R}$ and \eqref{eq:ax}, we have
\begin{align}\label{est:B11}
	\begin{aligned}
		\mathbf{B}_{11} & = \frac{R(\gamma -1) \theta_-}{2\rho_-} \int_{\mathbb{R}_+} a |\bar{u}_x|  \left|\left((\rho - \bar{\rho})-\frac{\rho_-}{\sqrt{R\gamma \theta_-}}(u-\bar{u})\right)  + \frac{\rho_-}{\sqrt{R\gamma \theta_-}}(u-\bar{u})\right|^2   \,dx   \\
		& \leq \frac{R(\gamma -1) \theta_-}{2\rho_-} (1+C\delta^{\frac{1}{2}} + C\delta^{\frac{1}{4}}  ) \int_{\mathbb{R}_+} |\bar{u}_x|\left| \frac{\rho_-}{\sqrt{R\gamma \theta_-}}(u-\bar{u})\right|^2   \,dx    \\
		&\qquad + C\big(1+\delta^{-\frac{1}{4}}\big)  \int_{\mathbb{R}_+} |\bar{u}_x| \left|(\rho - \bar{\rho})-\frac{\rho_-}{\sqrt{R\gamma \theta_-}}(u-\bar{u})\right|^2 \,dx  \\
		& \leq \frac{(\gamma- 1)\rho_-}{2\gamma} (1+C\delta^{\frac{1}{4}}  ) \delta \int_{y_0}^{1} |w|^2 \,dy + C\delta^{\frac{1}{4}} \mathbf{G}_1.
	\end{aligned}
\end{align}
For $\mathbf{B}_{12}$, we have 
\begin{align}\label{est:B12}
	\mathbf{B}_{12} = \rho_-  \int_{\mathbb{R}_+}a |\bar{u}_x| |u-\bar{u}|^2 \,dx \leq \rho_-(1+ \delta^{\frac{1}{2}}) \delta  \int_{y_0}^{1} |w|^2 \,dy.
\end{align}
The terms $\mathbf{B}_{13}$ through $\mathbf{B}_{15}$ can be estimated in the same way as in \eqref{est:B11}, as follows:
\begin{align}\label{est:B1345}
	\begin{split}
		&\begin{aligned}
			\mathbf{B}_{13} & = \frac{R\rho_-}{2\theta_-} \int_{\mathbb{R}_+} a |\bar{u}_x|  \left|\left((\theta - \bar{\theta})-\frac{(\gamma-1) \theta_-}{\sqrt{R\gamma \theta_-}}(u-\bar{u})\right)  + \frac{(\gamma-1) \theta_-}{\sqrt{R\gamma \theta_-}}(u-\bar{u})\right|^2   \,dx \\
			& \leq \frac{(\gamma- 1)^2}{2\gamma} \rho_-(1+C\delta^{\frac{1}{4}}  ) \delta \int_{y_0}^{1} |w|^2 \,dy + C\delta^{\frac{1}{4}} \mathbf{G}_2,
		\end{aligned}\\
		&\begin{aligned}
			\mathbf{B}_{14} & =   \frac{R(\gamma-1) \theta_-}{\sqrt{R\gamma \theta_-}} \int_{\mathbb{R}_+} a |\bar{u}_x| \left|\left((\rho - \bar{\rho})-\frac{\rho_-}{\sqrt{R\gamma \theta_-}}(u-\bar{u})\right)  + \frac{\rho_-}{\sqrt{R\gamma \theta_-}}(u-\bar{u})\right| |u-\bar{u}| \,dx \\
			& \leq \frac{\gamma- 1}{\gamma} \rho_-(1+C\delta^{\frac{1}{4}}  ) \delta \int_{y_0}^{1} |w|^2 \,dy + C\delta^{\frac{1}{4}} \mathbf{G}_1,
		\end{aligned}\\
		&\begin{aligned}
			\mathbf{B}_{15}& =  \frac{R \rho_-}{\sqrt{R\gamma \theta_-}} \int_{\mathbb{R}_+} a |\bar{u}_x| |u-\bar{u}|  \left|\left((\theta - \bar{\theta})-\frac{(\gamma-1) \theta_-}{\sqrt{R\gamma \theta_-}}(u-\bar{u})\right)  + \frac{(\gamma-1) \theta_-}{\sqrt{R\gamma \theta_-}}(u-\bar{u})\right| \,dx \\
			& \leq \frac{\gamma- 1}{\gamma} \rho_-(1+C \delta^{\frac{1}{4}}  ) \delta \int_{y_0}^{1} |w|^2 \,dy + C\delta^{\frac{1}{4}} \mathbf{G}_2.
		\end{aligned}
	\end{split}
\end{align}
Substituting \eqref{est:B11}, \eqref{est:B12}, and \eqref{est:B1345} into \eqref{eq:B1} yields
\begin{align}\label{est:B1}
	\begin{aligned}
		\mathbf{B}_1 \leq \frac{\gamma^2 + 5\gamma -4}{2\gamma}  \rho_-(1+ C\delta^{\frac{1}{4}}  ) \delta \int_{y_0}^{1} |w|^2 \,dy + C\delta^{\frac{1}{4}} (\mathbf{G}_1 + \mathbf{G}_2).
	\end{aligned}
\end{align}

$\bullet$ \textbf{Estimates on $\mathbf{D}$:} First, recall from \eqref{termG} that
\[\mathbf{D}(U)=\mu\int_{\mathbb{R}_+}a|(u-\bar{u})_x|^2\,dx+\kappa\int_{\mathbb{R}_+}\frac{a}{\theta}|(\theta-\bar{\theta})_x|^2\,dx:=\mathbf{D}_{u_1}(U)+\mathbf{D}_{\theta_1}(U).\]
Using $a\ge1$, we find that
\begin{equation*} 
	\mathbf{D}_{u_1}\ge \mu \int_{\mathbb{R}_+}|(u-\bar{u})_x|^2\,dx=\mu\int_{y_0}^{1}|\partial_yw|^2\left(\frac{dy}{dx}\right)\,dy.    
\end{equation*}
From Appendix \ref{appjacobian}, we have
\begin{equation*} 
	\left|\mu\frac{1}{y(1-y)}\frac{dy}{dx}- \frac{\gamma+1}{2}\rho_- \frac{\mu R\gamma}{\mu R\gamma+\kappa (\gamma-1)^2} \delta\right| \leq C\delta^2.
\end{equation*}
This implies that
\begin{align}\label{est:Du1}
	\begin{aligned}
\mathbf{D}_{u_1}	\geq \mu\int_{y_0}^{1}|\partial_yw|^2\left(\frac{dy}{dx}\right)\,dy&\ge \frac{\gamma+1}{2}\rho_-\frac{\mu R\gamma}{\mu R\gamma+\kappa (\gamma-1)^2} (1-C\delta^2) \delta\int_{y_0}^{1}|\partial_yw|^2y(1-y)\,dy\\
	&\ge \frac{\gamma+1}{2}\rho_-\frac{\mu R\gamma}{\mu R\gamma+\kappa (\gamma-1)^2} (1-C\delta^2) \delta\int_{y_0}^{1}(y-y_0)(1-y)|\partial_yw|^2\,dy.
\end{aligned}
\end{align}
Likewise, we obtain
\begin{align*}
	\mathbf{D}_{\theta_1}&\ge\kappa \int_{y_0}^{1}\frac{1}{\theta}|(\theta-\bar{\theta})_y|^2\left(\frac{dy}{dx}\right)\,dy\\
	&\ge  \frac{\gamma+1}{2}\frac{\rho_-\kappa}{\mu\theta_-}\frac{\mu R\gamma}{\mu R\gamma+\kappa (\gamma-1)^2}\left(1-C(\delta +\varepsilon)\right) \delta \int_{y_0}^{1}(y-y_0)(1-y)|(\theta-\bar{\theta})_y|^2\,dy.
\end{align*}
At first glance, since the estimates of the leading-order bad term $\mathbf{B}_1$ in \eqref{est:B1} are related to the variable $u-\util$, one might hope to control it using only the diffusion term $\mathbf{D}_{u_1}$ in \eqref{est:Du1} and the Poincar\'e-type inequality in Lemma \ref{Poincare}. However, since
\begin{equation*} 
	1>\frac{\mu R\gamma}{\mu R\gamma+\kappa (\gamma-1)^2} \rightarrow 0  \qquad \text{as} \qquad \gamma \rightarrow \infty,
\end{equation*}
it follows that the diffusion term $\mathbf{D}_{u_1}$ alone is insufficient to control \eqref{est:B1}. Therefore, as in \cite{KVW-NSF}, we first apply the Poincar\'e-type inequality to $\mathbf{D}_{\theta_1}$, and subsequently obtain an additional good term involving $u-\util$.

First, using Lemma \ref{Poincare} and 
\[\int_{y_0}^{1} |w-\bar{w}|^2 \,dy = \int_{y_0}^{1} w^2 \,dy - (1-y_0) \bar{w}^2,\]
where $\bar{w} = \frac{1}{1-y_0} \int_{y_0}^{1} w \,dy$,
we have 
\begin{align*} 
	\begin{aligned}
	\mathbf{D}\ge & (\gamma + 1)\rho_-\frac{\mu R\gamma}{\mu R\gamma+\kappa (\gamma-1)^2}\left(1-C(\delta +\varepsilon)\right)\d \bigg[\int_{y_0}^1w^2\,dy-(1-y_0) |\bar{w}|^2   \\
	&\phantom{(\gamma + 1)\rho_-\frac{\mu R\gamma}{\mu R\gamma+\kappa (\gamma-1)^2}}+\frac{\kappa}{\mu \theta_-}\left(\int_{y_0}^{1}|\theta-\bar{\theta}|^2\,dy-\frac{1}{1-y_0}\left(\int_{y_0}^{1}(\theta-\bar{\theta})\,dy\right)^2\right)\bigg].  
	\end{aligned}
\end{align*}
Observe that Young's inequality and Cauchy-Schwartz inequality yield
\begin{align*}
	\int_{y_0}^{1}|\theta-\bar{\theta}|^2\,dy&=\int_{y_0}^{1}\left|(\theta-\bar{\theta})-\frac{(\gamma-1)\theta_-}{\sqrt{\gamma R \theta_-}}(u-\bar{u})+\frac{(\gamma-1)\theta_-}{\sqrt{\gamma R \theta_-}}(u-\bar{u})\right|^2\,dy \\
	&\ge\frac{(\gamma-1)^2\theta_-}{\gamma R}(1-\delta^{\frac{1}{4}})\int_{y_0}^{1}w^2\,dy-C\delta^{-\frac{1}{4}}\int_{y_0}^{1}\left|(\theta-\bar{\theta})-\frac{(\gamma-1)\theta_-}{\sqrt{\gamma R \theta_-}}(u-\bar{u})\right|^2\,dy,  
\end{align*}
and
\begin{align*}
	\left(\int_{y_0}^1(\theta-\bar{\theta})\,dy\right)^2&\le2\left(\int_{y_0}^{1}\frac{(\gamma-1)\theta_-}{\sqrt{\gamma R \theta_-}}(u-\bar{u})\,dy\right)^2+2 \left(\int_{y_0}^{1}\left(\theta-\bar{\theta}-\frac{(\gamma-1)\theta_-}{\sqrt{\gamma R \theta_-}}(u-\bar{u})\right)\,dy\right)^2\\
	&\le 2\frac{(\gamma-1)^2\theta_-}{\gamma R} (1-y_0)^2 \bar{w}^2+2(1-y_0)\int_{y_0}^{1}\left|(\theta-\bar{\theta})-\frac{(\gamma-1)\theta_-}{\sqrt{\gamma R \theta_-}}(u-\bar{u})\right|^2\,dy.
\end{align*}
Thus, from $\frac{\mu R\gamma}{\mu R\gamma+\kappa (\gamma-1)^2}<1$ and $0<y_0<\frac{1}{8}$, we obtain
\begin{align*}
	\begin{aligned} 
		\mathbf{D}\ge& \left(\gamma+1\right) \rho_-\left(1-\delta^{\frac{1}{4}}-C(\varepsilon+ \delta)\right)  \delta \int_{y_0}^{1}|w|^2\,dy\\
		&-\left(\gamma+1\right) \rho_- \delta\left[1+\frac{2\kappa(\gamma-1)^2}{\mu R \gamma}\right] \frac{8}{7} \left(\int_{y_0}^{1} w \,dy \right)^2-C\delta^{\frac{1}{4}}\mathbf{G}_2.
	\end{aligned}    
\end{align*}
Here, we define a quantity that will be used below:
\[\alpha_{\gamma}:=\frac{\gamma^2 + 5\gamma -4}{2\gamma}-\frac{7(\gamma+1)}{8}=-\frac{3\gamma^2-13\gamma+16}{8\gamma}<0,\]
for $\gamma> 1$. \vspace{0.3cm}

$\bullet$ \textbf{Conclusion:} Combining the above estimates and using the smallness of $\delta$ and $\varepsilon$, we have
\begin{align*}
	&-\frac{\delta}{2M}|\dot{\mathbf{X}}|^2+\mathbf{B}_1+\mathbf{B}_2-\mathbf{G}-\frac{7}{8}\mathbf{D}\\
	&\quad \le  -\frac{1}{2}(\mathbf{G}_1 + \mathbf{G}_2) + C(\varepsilon + \sqrt{\delta}) \mathbf{G}^S + \frac{1}{2}\alpha_{\gamma}\rho_-\delta\int_{y_0}^{1} |w|^2 \,dy \\
	&\qquad -\delta M\rho_-^2\left(\int_{y_0}^1w\,dy\right)^2+ \rho_-\left(\gamma+1\right) \delta\left[1+\frac{2\kappa(\gamma-1)^2}{\mu R \gamma}\right]\left(\int_{y_0}^{1}w\,dy\right)^2.
\end{align*}
Choosing $M= \frac{1}{\rho_-}\left(\gamma+1\right)\left[1+\frac{2\kappa(\gamma-1)^2}{\mu R\gamma}\right]$, we have 
\begin{align*}
	&-\frac{\delta}{2M}|\dot{\mathbf{X}}|^2+\mathbf{B}_1+\mathbf{B}_2-\mathbf{G}-\frac{7}{8}\mathbf{D}\\
	&\quad \le  -\frac{1}{2}(\mathbf{G}_1 + \mathbf{G}_2) + C(\varepsilon + \sqrt{\delta}) \mathbf{G}^S + \frac{1}{2} \alpha_{\gamma}\rho_-\delta\int_{y_0}^{1} |w|^2 \,dy.
\end{align*}
Then using 
\begin{align*}
   \delta \int_{y_0}^{1} |w|^2 \,dy = \int_{\mathbb{R}_+} |\bar{u}_x| |u-\bar{u}|^2 \,dx,
\end{align*}
together with
\begin{align*}
	\int_{\mathbb{R}_+} |\bar{u}_x| |\rho -\bar{\rho}|^2 \,dx & \leq 2 \int_{\mathbb{R}_+} |\bar{u}_x|\left|(\rho - \bar{\rho})-\frac{\rho_-}{\sqrt{R\gamma \theta_-}}(u-\bar{u})\right|^2\,dx + 2 \int_{\mathbb{R}_+}|\bar{u}_x|\left|\frac{\rho_-}{\sqrt{R\gamma \theta_-}}(u-\bar{u}) \right|^2\,dx \\
	& \leq C\sqrt{\delta} \mathbf{G}_1 + C\int_{\mathbb{R}_+} |\bar{u}_x| |u-\bar{u}|^2 \,dx,
\end{align*}
and
\begin{align*}
		\int_{\mathbb{R}_+} |\bar{u}_x| |\theta -\bar{\theta}|^2 \,dx & \leq 2	\int_{\mathbb{R}_+} |\bar{u}_x| \left|(\theta-\bar{\theta})-\frac{(\gamma-1)\theta_-}{\sqrt{\gamma R \theta_-}}(u-\bar{u}) \right|^2 \,dx +\int_{\mathbb{R}_+} |\bar{u}_x|\left|\frac{(\gamma-1)\theta_-}{\sqrt{\gamma R \theta_-}}(u-\bar{u})\right|^2\,dx \\
		&  \leq C\sqrt{\delta} \mathbf{G}_2 + C\int_{\mathbb{R}_+} |\bar{u}_x| |u-\bar{u}|^2 \,dx,
\end{align*}
we have 
\begin{align*}
		-\frac{\delta}{2M}|\dot{\mathbf{X}}|^2+\mathbf{B}_1+\mathbf{B}_2-\mathbf{G}-\frac{7}{8}\mathbf{D}
	\le  -\frac{1}{4}(\mathbf{G}_1 + \mathbf{G}_2) - C_*\mathbf{G}^S.
\end{align*}
for some constant $C_*>0$.
\end{proof}

\subsection{Estimates of remaining terms}
It follows from \eqref{est:REPri} and \eqref{eq:XY} that 

\begin{align*}
	\begin{aligned}
		\frac{d}{dt} \int_{\mathbb{R}_+} a\bar{\theta} \eta(U|\bar{U}) \,dx  
		\leq &  -\frac{\delta}{2M}|\dot{\mathbf{X}}|^2+\mathbf{B}_1+\mathbf{B}_2-\mathbf{G}-\frac{7}{8}\mathbf{D}  \\
		& -\frac{\delta}{2M}|\dot{\mathbf{X}}|^2 +   \dot{\mathbf{X}} \sum_{i=4}^{6}\mathbf{Y}_i+ \sum_{i=3}^{5} \mathbf{B}_i   - \frac{1}{8}\mathbf{D}+ \mathcal{P}.
	\end{aligned}
\end{align*}
Then by Lemma \ref{lma:leading} and Young's inequality
\begin{align}\label{est:twoparts}
	\begin{aligned}
		\frac{d}{dt} \int_{\mathbb{R}_+} a\bar{\theta} \eta(U|\bar{U}) \,dx  
		\leq &   -\frac{1}{4}(\mathbf{G}_1 + \mathbf{G}_2) - C_*\mathbf{G}^S  \\
		& -\frac{\delta}{4M}|\dot{\mathbf{X}}|^2 +  \frac{C}{\delta} \sum_{i=4}^{6}|\mathbf{Y}_i|^2 + \sum_{i=3}^{5} \mathbf{B}_i   - \frac{1}{8}\mathbf{D}+ \mathcal{P}.
	\end{aligned}
\end{align}
In what follows, we will control the remaining bad terms on the right hand side of \eqref{est:twoparts}. \vspace{0.3cm}

\noindent $\bullet$ \textbf{Estimate of $\mathbf{Y}_i(i=4,5,6)$:} By virtue of \eqref{eq:phi} and \eqref{estderi}, we obtain 
\begin{align*}
	\begin{aligned}
		|(\mathbf{Y}_4, \mathbf{Y}_5)| \leq C \int_{\mathbb{R}_+} |\bar{u}_x| \left|(\rho- \bar{\rho}, \theta - \bar{\theta})\right|^2 \,dx. 
	\end{aligned}
\end{align*}
On the other hand, from Lemma \ref{lem:viscous_shock} and \eqref{apriori_small}, we have 
\begin{align*}
		|(\mathbf{Y}_4, \mathbf{Y}_5)| \leq C  \delta^2 \int_{\mathbb{R}_+} \left|(\rho- \bar{\rho},  \theta - \bar{\theta})\right|^2 \,dx  \leq C\delta^2 \varepsilon^2.
\end{align*}
Thus, we have 
\begin{align*}
	\frac{C}{\delta} 	|(\mathbf{Y}_4, \mathbf{Y}_5)|^2 \leq C \delta \varepsilon^2 \mathbf{G}^S.
\end{align*}
Similarly, we obtain 
\begin{align*}
		\frac{C}{\delta} 	|\mathbf{Y}_6|^2 \leq C\varepsilon^2 \mathbf{G}^S.
\end{align*}
Therefore, we have 
\begin{align}\label{est:Y4-6}
	\frac{C}{\delta} \sum_{i=4}^{6}	|\mathbf{Y}_i|^2 \leq C \varepsilon \mathbf{G}^S.
\end{align}

\noindent $\bullet$ \textbf{Estimate of $\mathbf{B}_i(i=3,4,5)$:} Using Young's inequality, and noting from  \eqref{shock_prop} that
\[
\|a_x\|_{L^\infty} \leq C\delta^{-1/2}\|u_x\|_{L^\infty} \leq C\delta\sqrt{\delta},
\]
we have
\begin{align*}
	\begin{aligned}
		|\mathbf{B}_3|  & \leq C\int_{\mathbb{R}_+} |a_x|\left(|u-\bar{u}||(u-\bar{u})_x| + |\theta -\bar{\theta}||(\theta- \bar{\theta})_x|\right) \,dx   \\
		& \leq \frac{1}{80} (\mathbf{D}_{u_1} + \mathbf{D}_{\theta_1}) + C\int_{\mathbb{R}_+} |a_x|^2 |(u-\bar{u}, \theta- \bar{\theta})|^2 \,dx \\
		& \leq  \frac{1}{80} (\mathbf{D}_{u_1} + \mathbf{D}_{\theta_1}) + C \delta \mathbf{G}^S.
	\end{aligned}
\end{align*}
Similarly, using Young's inequality and \eqref{shock_prop}, we have 
\begin{align*}
	|\mathbf{B}_4| &\leq C\e \intRp \big(|\util_x||\th - \thtil|^2 + |(u-\util)_x|^2 + |(\th - \thtil)_x|^2\big)\,dx\\
	&\quad + C\intRp |\util_x||\th - \thtil|\big(|(u - \util)_x|, |(\th - \thtil)_x|\big)\,dx +C\delta \intRp |\util_x||(\rho -\rhotil, u-\util, \th-\thtil)|^2\,dx\\
	&\leq C(\varepsilon + \delta)\mathbf{G}^S + \left(C\varepsilon + \frac{1}{80}\right) (\mathbf{D}_{u_1}+\mathbf{D}_{\theta_1}).
\end{align*}
For $\mathbf{B}_5$, the definition of $\mathbf{G}^S$ implies that
\begin{align*}
	|\mathbf{B}_5| \leq C\delta \mathbf{G}^S.
\end{align*}
Thus, we conclude that
\begin{align}\label{est:B3-5}
	\sum_{i=3}^{5} \mathbf{B}_i \leq \frac{1}{20} (\mathbf{D}_{u_1} + \mathbf{D}_{\theta_1}) + C(\varepsilon + \delta) \mathbf{G}^S.
\end{align}

\subsection{Estimates on the boundary terms}
Compared to the whole-space problem (see, e.g., \cite{HKK23,KVW23,KVW-NSF}), the boundary terms in $\mathcal{P}$ arise from integration by parts. In the lemma below, we provide estimates for these boundary terms, which are controlled by the outflow and impermeable wall boundary conditions, the constant $\beta$, and the small portion of the second-order terms of $u$ and $\theta$.

\begin{lemma} \label{lem:boundary}
	Under the assumption of Proposition \ref{prop:main} (or Proposition \ref{prop:main2}), there exists $C>0$ independent of $\delta$ such that
	\begin{align*}
		\begin{aligned}
			\int_{0}^t\mathcal{P}\, d \tau  \le&  Cu_- 	\int_{0}^{t} \left(\bar{\theta} \eta(U|\bar{U})\right)|_{x=0} \,d\tau + Ce^{-C\delta \beta}   \\
			& +C\varepsilon^2\left(\int_{0}^t\|(u-\bar{u})_{xx}\|_{L^2(\mathbb{R}+)}^{2}+ \int_{0}^{t} \|(\theta-\bar{\theta})_{xx}\|^2_{L^{2}(\mathbb{R_+})}\right)\,d\tau
		\end{aligned}
	\end{align*}

	for all $t\in \left[0, T\right]$.
\end{lemma}
\begin{proof}
Since $u_- \leq 0$, $\mathcal{P}_1$  can be estimated as follows:
\begin{align*}
	\int_{0}^{t} \mathcal{P}_1 \,d\tau \leq Cu_- 	\int_{0}^{t} \left(\bar{\theta} \eta(U|\bar{U})\right)|_{x=0} \,d\tau .
\end{align*}
	Recall that $\sigma > 0$ for both boundary value problems and that, from \eqref{bddx12},
	\begin{equation*}
		-\sigma t -\bX (t) -\beta \le - \frac{\sigma}{2} t -\beta <0, \quad t\le T.
	\end{equation*}
	This, together with \eqref{shock_prop} implies
	\begin{align}\label{est_boundary}
		\begin{aligned}
			& |(\bar{u}(t,0)- u_-,\bar{\theta}(t,0)- \theta_-)| \leq C\delta e^{-C\delta|-\sigma t-\mathbf{X}(t) -\beta|} \leq C\delta e^{-C\delta t} e^{-C\delta \beta}.
		\end{aligned}
	\end{align}
	By Interpolation inequality, Young's inequality, \eqref{apriori_small} and \eqref{est_boundary}, we have
	\begin{equation*}
		\begin{aligned}
            \int_{0}^{t} \mathcal{P}_2 \,d\tau& \leq \left| \int_{0}^{t}  \left(\mu a (u-\bar{u})(u-\bar{u})_x\right)|_{x=0} \,d\tau \right|   \leq C\int_{0}^{t}   |\bar{u}(\tau,0)- u_-| \|(u-\bar{u})_x\|_{L^{\infty}(\mathbb{R_+})}  \,d\tau  \\
            & \leq C\int_{0}^{t} |\bar{u}(\tau,0)- u_-|^{\frac{4}{3}}  \,d\tau  + \int_{0}^{t} \|(u-\bar{u})_x\|^2_{L^{2}(\mathbb{R_+})}\|(u-\bar{u})_{xx}\|^2_{L^{2}(\mathbb{R_+})}  \,d\tau  \\
            & \leq C e^{-C\delta \beta} + C\varepsilon^2  \int_{0}^{t} \|(u-\bar{u})_{xx}\|^2_{L^{2}(\mathbb{R_+})}  \,d\tau.
		\end{aligned}
	\end{equation*}
	Similarly, 
	\begin{align*}
		& \int_{0}^{t} \mathcal{P}_3 \,d\tau \leq C e^{-C\delta \beta} + C\varepsilon^2  \int_{0}^{t} \|(\theta-\bar{\theta})_{xx}\|^2_{L^{2}(\mathbb{R_+})}  \,d\tau.
	\end{align*}
	Using \eqref{shock_prop} and \eqref{est_boundary}, we have 
	
	\begin{align*}
		\begin{aligned}
			 & \int_{0}^{t} \mathcal{P}_4 \,d\tau \leq C\|u-\bar{u}\|_{L^{\infty}(\mathbb{R_+})} \int_{0}^{t}  |\bar{\theta}(\tau,0)- \theta_-| \,d\tau \leq C e^{-C\delta \beta}, \\
			& \int_{0}^{t} \mathcal{P}_5 \,d\tau \leq C\|\rho-\bar{\rho}\|_{L^{\infty}(\mathbb{R_+})} \int_{0}^{t}  |\bar{u}(\tau,0)- u_-| \,d\tau \leq C e^{-C\delta \beta}.
		\end{aligned}
	\end{align*}
Combining the above estimates together, we complete the proof of Lemma \ref{lem:boundary}.
\end{proof}

\subsection{Proof of Lemma \ref{lem:main}}
Combining \eqref{est:twoparts}-\eqref{est:B3-5} and using the smallness of $\varepsilon, \delta$, we have 
\begin{align*}
	\begin{aligned}
			\frac{d}{dt} \int_{\mathbb{R}_+} a\bar{\theta} \eta(U|\bar{U}) \,dx  
		\leq &   -\frac{1}{4}(\mathbf{G}_1 + \mathbf{G}_2) - \frac{C_*}{2}\mathbf{G}^S  -\frac{\delta}{4M}|\dot{\mathbf{X}}|^2 - \frac{1}{10}\mathbf{D}+ \mathcal{P}(U).
	\end{aligned}
\end{align*}
Integrating the above inequality over $[0,t]$ for any $t\leq T$, we obtain
\begin{align*}
	\begin{aligned}
	& 	\int_{\mathbb{R}_+} \eta(U(t,x)|\bar{U}(t,x)) \,dx  + \delta \int_{0}^{t}|\dot{\mathbf{X}}(\tau)|^2 \,d\tau + \int_{0}^{t} \left(\mathbf{G}_1 + \mathbf{G}_2 + \mathbf{G}^S+ \mathbf{D} \right) \,d\tau  \\
		&\quad \leq C \int_{\mathbb{R}_+}\eta(U(0,x)|\bar{U}(0,x)) \,dx +C \int_{0}^{t} \mathcal{P} \,d\tau.
	\end{aligned}
\end{align*}
Then from Lemma \ref{lem:boundary} and the facts
\[	\lVert U-\bar{U} \rVert^2_{L^2(\mathbb{R}_+)} \sim \int_{\mathbb{R}_+} \eta(U|\bar{U}) \,dx, \quad D_{u_1} \sim \mathbf{D}_{u_1}, \quad D_{\theta_1} \sim \mathbf{D}_{\theta_1} , \quad \forall t \in [0,T],\]
we have 
\begin{align*}
	\begin{aligned}
		& 	\lVert U(t,\cdot)-\bar{U}(t,\cdot) \rVert^2_{L^2(\mathbb{R}_+)}  + \delta \int_{0}^{t}|\dot{\mathbf{X}}(\tau)|^2 \,d\tau + \int_{0}^{t} \left(\mathbf{G}_1 + \mathbf{G}_2 + \mathbf{G}^S+ D_{u_1} + D_{\theta_1} \right) \,d\tau   \\
		& \qquad \qquad + |u_-|\int_{0}^{t} |(\rho - \rhotil, u-\util, \th-\thtil)(\tau,0)|^2 \,d\tau \\
		& \quad \leq  C	\lVert U(0,\cdot)-\bar{U}(0,\cdot) \rVert^2_{L^2(\mathbb{R}_+)} +C\varepsilon^2\left(\int_{0}^t\|(u-\bar{u})_{xx}\|_{L^2(\mathbb{R}+)}^{2}+ \int_{0}^{t} \|(\theta-\bar{\theta})_{xx}\|^2_{L^{2}(\mathbb{R_+})}\right)\,d\tau  \\
		&\qquad + Ce^{-C\delta \beta}, 
	\end{aligned}
\end{align*}
which completes the proof of Lemma \ref{lem:main}.

\section{Higher order estimates}
\setcounter{equation}{0}
In this section, we provide $H^1$-estimates and then prove Propositions \ref{prop:main} and \ref{prop:main2}. For notational convenience, we denote $(\phi, \psi, \vartheta)(t,x)$ as
\[(\phi, \psi, \vartheta)(t,x): = (\rho -\rhotil, u-\util, \theta - \thtil)(t,x).\]
Using \eqref{eq:NSF} and \eqref{eq:NSFs}, we have the equations for $(\phi, \psi, \vartheta)$ given by 
\begin{equation}\label{eq:pertubation2}
	\begin{cases}
		&\phi_t + u\phi_x + \rho \psi_x = f + \bar{\rho}_x\dot{\mathbf{X}},\\
		&\rho \left(\psi_t + u\psi_x\right) + R\theta \phi_x -\mu\psi _{xx} = g+  \rho \bar{u}_x\dot{\mathbf{X}} ,\\
		&\frac{R}{\gamma-1} \rho \vartheta_t - \kappa \vartheta_{xx} = h + \frac{R}{\gamma-1} \rho  \bar{\theta}_x  \dot{\mathbf{X}}, 
	\end{cases}
\end{equation}
where 
\[
f := -\left(\bar{u}_x\phi + \bar{\rho}_x \psi \right), \quad g :=  -\rho \bar{u}_x\psi - R\phi \theta_x -R\bar{\rho}_x \vartheta - R\bar{\rho} \vartheta_x + R\frac{\bar{\rho}_x}{\bar{\rho}}\bar{\theta}\phi + R\bar{\theta}_x\phi - \mu \frac{\phi}{\bar{\rho}}\bar{u}_{xx} ,
\]
and 
\[
h:= -\frac{\kappa}{\rhotil} \bar{\theta}_{xx} \phi- \frac{\mu}{\bar{\rho}}(\bar{u}_x)^2 \phi  -\frac{R}{\gamma-1} \rho u\vartheta_x - \frac{R}{\gamma-1} \bar{\theta}_x   \rho \psi- p\psi_x - R \bar{u}_x \rho\vartheta  + \mu(u^2_x - \bar{u}^2_x).
\]
Observe that the functions $f, g$, and $h$ can be estimated as follows:
\begin{align}
	|f| &\leq C|\bar{u}_x||(\phi, \psi)|, \label{est:f} \\
	|g| &\leq C\left(|\bar{u}_x|\left|(\phi, \psi, \vartheta)\right| + |\vartheta_x|(|\phi|+|\rhotil|) + |\bar{u}_{xx} ||\phi|\right), \label{est:g}
\end{align}
and
\begin{equation}\label{est:h}
	|h| \leq C\left(|\bar{u}_x| |(\phi, \psi, \vartheta) | +  |(\psi_x, \vartheta_x)|+|\psi_x|^2\right).
\end{equation}

\subsection{$H^1$-estimates for $\rho - \rhotil$}
In this subsection, we derive higher-order estimates for the density perturbation $\rho - \rhotil$. When controlling the boundary terms, we treat the outflow and impermeable cases separately (see \eqref{est:H1bd}).
\begin{lemma}\label{lem:rhox}
		Under the hypothesis of Proposition \ref{prop:main} (or Proposition \ref{prop:main2}), there exists a positive constant $C$ such that
	\begin{align*}
	\begin{aligned}
		&\sup_{t \in [0,T]}\lVert \phi_x \rVert^2_{L^2(\mathbb{R}_+)} + \int_{0}^{T} \lVert \phi_x \rVert^2_{L^2(\mathbb{R}_+)}  \,dt + |u_-|\int_{0}^{T} \left|\phi_x(t,0)\right|^2 \,dt \\
		&\qquad \leq C\left(\lVert \phi_{0x} \rVert^2_{L^2(\mathbb{R}_+)} + \lVert \psi_0 \rVert^2_{L^2(\mathbb{R}_+)} +\lVert \psi \rVert^2_{L^2(\mathbb{R}_+)}\right) \\
		&\phantom{\qquad \leq \lVert \phi_{0x} \rVert^2_{L^2(\mathbb{R}_+)}} + C\int_0^T (D_{u_1} + D_{\th_1} + \mathbf{G}^S)\,dt + C\e \int_0^T D_{u_2}\,dt + Ce^{-C\delta \beta},
	\end{aligned}
\end{align*}
where $D_{u_1}, D_{u_2}, D_{\th_1}$  and $\mathbf{G}^S$ are the terms defined in \eqref{good_terms}.
\end{lemma}

\begin{proof}

	Differentiating $\eqref{eq:pertubation2}_1$ with $x$ and then dividing the resulting equation by $\rho$, we have 
	\begin{align}\label{eq:phix}
 \left(\frac{\phi_x}{\rho}\right)_t + u	\left(\frac{\phi_x}{\rho}\right)_x + \psi_{xx}  = \frac{f_x- \rho_x\psi_x}{\rho} + \frac{\bar{\rho}_{xx}}{\rho}\dot{\mathbf{X}} =: F +  \frac{\bar{\rho}_{xx}}{\rho}\dot{\mathbf{X}}.
	\end{align}
We first observe that
\[
	f_x = -\left(\bar{u}_x\phi + \bar{\rho}_x \psi \right)_x= -\left(\bar{u}_{xx}\phi +\bar{u}_{x}\phi _x + \bar{\rho}_{xx}\psi + \bar{\rho}_{x}\psi_x\right).
\]
This implies that 
\begin{align}\label{est:F}
	|F| \leq C \left(|(\phi , \psi )||\bar{u}_{xx}| + |(\phi_x, \psi_x)||\bar{u}_x| + |\phi_x||\psi_x|\right).
\end{align}

	Multiplying the equation \eqref{eq:phix} by $\frac{\phi_x}{\rho}$, integrating the resultings equality over $[0,t]\times \mathbb{R}_+$ and after integrating by parts, we obtain
	\begin{align}\label{eq:pertubationint}
		\begin{aligned}
			& 	\frac{1}{2} \int_{\mathbb{R}_+} \left(\frac{\phi_x}{\rho}\right)^2 (t,x)\,dx - \int_{0}^{t}\int_{\mathbb{R}_+} \frac{1}{2}u_x  \left(\frac{\phi_x}{\rho}\right)^2  \,dx\,d\tau + \frac{|u_-|}{2}\int_{0}^{t}  \left(\frac{\phi_x}{\rho}\right)^2 (\tau, 0)\,d\tau \\
			&  + \int_{0}^{t}\int_{\mathbb{R}_+} \frac{\phi_x\psi_{xx}}{\rho} \,dx\,d\tau = \frac{1}{2} \int_{\mathbb{R}_+} \left(\frac{\phi_x}{\rho}\right)^2 (0,x)\,dx + \int_{0}^{t}\int_{\mathbb{R}_+} F \frac{\phi_x}{\rho}  \,dx\,d\tau + \int_{0}^{t}\int_{\mathbb{R}_+} \frac{\bar{\rho}_{xx}}{\rho}\dot{\mathbf{X}}\frac{\phi_x}{\rho}dxd\tau.
		\end{aligned}
	\end{align}
	
	Next, we eliminate the last term on the left-hand side of \eqref{eq:pertubationint} by using $\eqref{eq:pertubation2}_2$. 
	Multiplying  $\eqref{eq:pertubation2}_2$ by $\frac{\phi_x}{\rho}$ yields that 
	
	\begin{align}\label{eq:psiphi}
		\left(\psi_t + u\psi_x\right) \phi_x + \frac{R\theta }{\rho}\phi^2_x - \frac{\mu\psi _{xx} \phi_x}{\rho}=  g \frac{\phi_x}{\rho}+  \phi_x\bar{u}_x\dot{\mathbf{X}},
	\end{align}
	and using the identity 
	\[
	\intRp \psi_t \phi_x \,dx  =  \intRp (\psi\phi_x)_t \,dx - \intRp  \psi \phi_{xt} \,dx =\intRp(\psi\phi_x)_t\,dx  - (\psi\phi_t )(t,0)  + \intRp \psi_x\phi_t \,dx,
	\]
	and integrating the resulting equality \eqref{eq:psiphi} over $[0,t]\times \mathbb{R}_+$, we have
	\begin{align}\label{eq:pertubpsiphi}
		\begin{aligned}
				& \int_{\mathbb{R}_+} \psi \phi_x  \,dx + \int_{0}^{t} \int_{\mathbb{R}_+} \left(\phi_t + u\phi_x\right) \psi_x \,dx\,d\tau +  \int_{0}^{t} \int_{\mathbb{R}_+} \frac{R\theta }{\rho}\phi^2_x \,dx\,d\tau -\mu  \int_{0}^{t} \int_{\mathbb{R}_+}\frac{\psi _{xx} \phi_x}{\rho} \,dx\,d\tau   \\
			& = \int_{\mathbb{R}_+} \psi_0 \phi_{0x}\,dx + \int_{0}^{t}\int_{\mathbb{R}_+} g \frac{\phi_x}{\rho} \,dx\,d\tau +  \int_{0}^{t}\int_{\mathbb{R}_+}\phi_x\bar{u}_x\dot{\mathbf{X}} \,dx\,d\tau + \int_{0}^{t} (\psi \phi_t)(\tau,0) \,d\tau.
		\end{aligned}
	\end{align}
	By substituting $\phi_t + u\phi_x = -\rho \psi_x + f + \bar{\rho}_x\dot{\mathbf{X}}$, which follows from $\eqref{eq:pertubation2}_1$, into \eqref{eq:pertubpsiphi}, multiplying \eqref{eq:pertubationint} by $\mu$, and summing the two resulting equations, we obtain
	\begin{align}\label{eq:sum1}
		\begin{aligned}
		& 	\frac{\mu}{2}\left\lVert \left(\frac{\phi_x}{\rho}\right)(t, \cdot) 
		\right\rVert^2_{L^2(\mathbb{R}_+)} +  \int_{\mathbb{R}_+} \psi \phi_x  \,dx   + \frac{\mu}{2}|u_-|\int_{0}^{t}  \left(\frac{\phi_x}{\rho}\right)^2 (\tau, 0)\,d\tau +  \int_{0}^{t} \int_{\mathbb{R}_+} \frac{R\theta }{\rho}\phi^2_x \,dx\,d\tau \\
		&\, \,  = \frac{\mu}{2}\left\lVert \frac{\phi_{0x}}{\rho_0} \right\rVert^2_{L^2(\mathbb{R}_+)}   +  \int_{\mathbb{R}_+} \psi_0 \phi_{0x}  \,dx +\int_{0}^{t}\int_{\mathbb{R}_+} \rho |\psi_x|^2 \,dx\,d\tau + \frac{\mu}{2}\int_{0}^{t}\int_{\mathbb{R}_+}u_x  \left(\frac{\phi_x}{\rho}\right)^2  \,dx\,d\tau \\
		&\quad  \, \, - \int_0^t \int_{\mathbb{R}_+} f \psi_x dx d\tau  + \int_{0}^{t}\int_{\mathbb{R}_+} \left(\mu F + g\right) \frac{\phi_x}{\rho}  \,dx\,d\tau   + \int_{0}^{t}\int_{\mathbb{R}_+}\left(\frac{\bar{\rho}_{xx}}{\rho}\frac{\phi_x}{\rho}+ \bar{u}_x\phi_x - \bar{\rho}_x \psi_x \right)  \dot{\mathbf{X}} (\tau) dx d\tau \\
		&\quad \, \, +  \int_{0}^{t} (\psi \phi_t)(\tau,0) \,d\tau =:\frac{\mu}{2}\left\lVert \frac{\phi_{0x}}{\rho_0} \right\rVert^2_{L^2(\mathbb{R}_+)}   +  \int_{\mathbb{R}_+} \psi_0 \phi_{0x}\,dx + \sum_{i=1}^6 \int_0^t I_i\,d\tau.
		\end{aligned}
	\end{align}
	First, by virtue of Young's inequality, observe that 
	\begin{align*}
		\left| \int_{\mathbb{R}_+} \psi \phi_x  \,dx \right| \leq \nu\lVert \phi_x \rVert^2_{L^2(\mathbb{R}_+)} + C_{\nu}\lVert \psi \rVert^2_{L^2(\mathbb{R}_+)},
	\end{align*}
where $\nu>0$ is a constant to be chosen sufficiently small later, and $C_\nu$ denotes a constant depending on $\nu$.

Next, Cauchy-Schwarz inequality implies
	\begin{align*}
			\left|\int_{\mathbb{R}_+} \psi_0 \phi_{0x}  \,dx\right|  \leq C\left(\lVert \psi_0 \rVert^2_{L^2(\mathbb{R}_+)}  +\lVert \phi_{0x} \rVert^2_{L^2(\mathbb{R}_+)} \right) .
	\end{align*}
Since $\rho$ has a strictly positive lower bound, we obtain
\begin{align*}
	\int_{0}^{t} \int_{\mathbb{R}_+} \frac{R\theta }{\rho}\phi^2_x \,dx\,d\tau  \geq c\int_{0}^{t} \lVert \phi_x \rVert^2_{L^2(\mathbb{R}_+)}  \,d\tau,
\end{align*}
for some constant $c>0$. Now, we estimate terms $I_1, \cdots, I_6$ in \eqref{eq:sum1}. First of all, we find that
	\begin{align*}
			\left|\int_0^t I_1\,d\tau\right| = \left|\int_{0}^{t}\int_{\mathbb{R}_+} \rho |\psi_x|^2 \,dx\,d\tau  \right|\leq C\int_{0}^{t} \lVert \psi_x \rVert^2_{L^2(\mathbb{R}_+)}  \,d\tau.
	\end{align*}
Using $u_x = \psi_x + \bar{u}_x$, Lemma \ref{lem:viscous_shock} and Sobolev embedding, we obtain
\begin{equation}\label{est:H1I2}
\begin{aligned}
	&\left|\int_0^t I_2\,d\tau \right| \leq C\left|\int_{0}^{t}\int_{\mathbb{R}_+}u_x  \left(\frac{\phi_x}{\rho}\right)^2  \,dx\,d\tau \right| \leq C \int_{0}^{t}\int_{\mathbb{R}_+} |\util_x||\phi_x|^2 \,dx\,d\tau + C\int_{0}^{t}\int_{\mathbb{R}_+} |\psi_x||\phi_x|^2\,dx\,d\tau \\
	&\quad \leq C \int_{0}^{t}\int_{\mathbb{R}_+} |\util_x||\phi_x|^2 \,dx\,d\tau + C\int_{0}^{t}\|\psi_x\|_{L^\infty(\Rp)}\|\phi_x\|_{L^2(\Rp)}^2\,d\tau\\
	&\quad  \leq C\delta\int_{0}^{t} \lVert \phi_x \rVert^2_{L^2(\mathbb{R}_+)} \,d\tau  + C\e\int_{0}^{t} \|\psi_x \|_{H^1(\mathbb{R}_+)} \|\phi_x\|_{L^2(\mathbb{R}_+)} \,d\tau  \\
	&\quad  \leq C\delta\int_{0}^{t} \lVert \phi_x \rVert^2_{L^2(\mathbb{R}_+)} \,d\tau  + C\varepsilon\left( \int_{0}^{t} \lVert \psi_x \rVert^2_{H^1(\mathbb{R}_+)} \,d\tau + \int_{0}^{t} \lVert \phi_x \rVert^2_{L^2(\mathbb{R}_+)} \,d\tau  \right).
\end{aligned}
\end{equation}
	Now, from \eqref{shock_prop}, \eqref{est:f} and Young's inequality, we have
	\begin{align*}
		\left|\int_0^t I_3\,d\tau \right| \leq\left|\int_0^t \int_{\mathbb{R}_+} f \psi_x dx d\tau \right| &\leq \int_0^t \int_{\mathbb{R}_+} |\ubar_x||(\phi,\psi)||\psi_x| dx d\tau \leq C\mathbf{G}^S + C\d \mathbf{D}_{u_2}.
	\end{align*}
	For $I_4$, by virtue of \eqref{est:F}, \eqref{est:g}, Lemma \ref{lem:viscous_shock} and Young's inequality, we obtain
	\begin{align*}
		\left|\int_0^t I_4 \,d\tau\right| &\leq \int_0^t \intRp |\phi_x|\Big[|(\util_x, \util_{xx})||(\phi,\psi,\vartheta)| + |\util_x||(\phi_x,\psi_x)| + |\phi_x||\psi_x| + |\vartheta_x|(|\phi| + |\rhotil|) \Big]\,dx\,d\tau\\
		&\leq C\delta \int_0^t \big(\mathbf{G}^S + \|\phi_x\|_{L^2(\Rp)}^2 + \|\psi_x\|_{L^2(\Rp)}^2 \big) \,d\tau +\int_0^t \intRp |\phi_x|^2|\psi_x|\,dx\,d\tau\\
		&\qquad  + \nu \int_0^t \|\phi_x\|_{L^2(\Rp)}^2\,d\tau + C_\nu \int_0^t \|\vartheta_x\|_{L^2(\Rp)}^2\,d\tau.
	\end{align*}
	Here, $\nu>0$ is a small constant (to be chosen sufficiently small later), and $C_\nu$ denotes a constant depending on $\nu$. From \eqref{est:H1I2}, we find that
	\[
	\int_0^t \intRp |\phi_x|^2|\psi_x|\,dx \,d\tau \leq C\e \left( \int_0^t \big(D_{u_1} + D_{u_2} + \|\phi_x\|_{L^2(\Rp)}^2\big)\,d\tau \right)
	\]
	Thus, we obtain 
	\[
	\left|\int_0^t I_4 \,d\tau \right| \leq C(\delta + \e + \nu)\int_0^t \|\phi_x\|_{L^2(\Rp)}^2\,d\tau + C\e \int_0^t \mathcal{D}_{u_2}\,d\tau + C\int_0^t \big(\delta \mathbf{G}^S + \mathcal{D}_{u_1} + \mathcal{D}_{\th_1})\,d\tau.
	\]
	To estimate $I_5$, we use Lemma \ref{lem:viscous_shock} and Young's inequality to have
	\begin{align*}
		\begin{aligned}
			&\left|\int_0^t I_5\,d\tau \right| \leq	\left|\int_{0}^{t}\int_{\mathbb{R}_+}\left(\frac{\bar{\rho}_{xx}}{\rho}\frac{\phi_x}{\rho}+ \bar{u}_x\phi_x - \bar{\rho}_x \psi_x \right)  \dot{\mathbf{X}} (\tau) \,d\tau \right|  \\
			& \leq C\delta \int_{0}^{t}\int_{\mathbb{R}_+}|\bar{u}_{x}| |\phi_x|  |\dot{\mathbf{X}} (\tau)|  \,dx\,d\tau +  \int_{0}^{t}\int_{\mathbb{R}_+} |\bar{u}_x| |(\phi_x, \psi_x)|  |\dot{\mathbf{X}} (\tau)| \,dx\,d\tau  \\
			& \leq C \int_{0}^{t}\int_{\mathbb{R}_+} |\bar{u}_x| \big(|\phi_x|^2 + |\psi_x|^2\big) \,dx\,d\tau + \int_{0}^{t}\int_{\mathbb{R}_+} |\bar{u}_x| |\dot{\mathbf{X}} (\tau)|^2 \,dx\,d\tau  \\
			& \leq C\delta \int_{0}^{t} \lVert \phi_x \rVert^2_{L^2(\mathbb{R}_+)} \,d\tau  + C\delta \int_{0}^{t} \lVert \psi_x \rVert^2_{L^2(\mathbb{R}_+)} \,d\tau + C\delta  \int_{0}^{t} |\dot{\mathbf{X}} (\tau)|^2 \,dx\,d\tau  .
		\end{aligned}
	\end{align*}
	Finally, we estimate the boundary term  $\int_0^t I_6\,d\tau = \int_{0}^{t} (\psi \phi_t)(\tau,0) \,d\tau$. To this end, we use the mass equation \eqref{eq:pertubation2}$_1$: 
	\[
	\phi_t + u\phi_x + \rho \psi_x = -\left(\bar{u}_x\phi + \bar{\rho}_x \psi \right) + \bar{\rho}_x\dot{\mathbf{X}}.
	\]
	Substituting this into the integrand gives
	\begin{equation}\label{est:H1bd}
		\begin{aligned}
			\left|\int_0^t I_6 \,d\tau\right| \leq \left|\int_0^t \underbrace{(u \phi_x \psi)(\tau,0)}_{=:\mathcal{I}}\,d\tau\right| + \left|\int_0^t (-\rho \psi_x \psi - \util_x \phi \psi - \rhotil_x \psi^2 + \rhotil_x \psi \dot{X})(\tau,0)\,d\tau\right|.
		\end{aligned}
	\end{equation}
	For the term $\mathcal{I}$, we consider the outflow and impermeable cases separately, as follows. 
	
	\noindent $\bullet$ \textit{Case I) Outflow problem}

	In the outflow setting, using Young's inequality and Lemma \ref{lem:viscous_shock}, we have
	\begin{align*}
		\begin{aligned}
			\left|	\int_{0}^{t} \mathcal{I}\,d\tau \right|  \leq \nu |u_-|\int_{0}^{t} \left|\phi_x|_{x=0}\right|^2 \,d\tau + C_{\nu} \int_{0}^{t} \left|\psi(\tau,0)\right|^2 \,d\tau   \leq \nu |u_-|\int_{0}^{t} \left|\phi_x(\tau,0)\right|^2 \,d\tau + Ce^{-C\delta \beta}.
		\end{aligned}
	\end{align*}
	Here, $\nu>0$ is a constant to be chosen sufficiently small later, and $C_\nu$ denotes a constant depending on $\nu$.

	\noindent $\bullet$ \textit{Case II) Impermeable wall problem}

	For the impermeable wall problem, thanks to the boundary condition $u_- = 0$, we find that 
	\[
	\left|\int_0^t \mathcal{I} \,d\tau\right| = \left| \int_0^t (u \phi_x \psi)(\tau,0) \,d\tau\right| = 0.
	\]
	Now, it remains to control the remaining part of $I_6$. By virtue of interpolation inequality, Young's inequality, and \eqref{apriori_small}, we get
	\begin{align*}
		\begin{aligned}
		\left|\int_{0}^{t} (\rho \psi_x \psi) (\tau,0) \,d\tau \right|  & \leq C\int_{0}^{t} \left|\psi (\tau,0)\right| \lVert \psi_x\rVert_{L^{\infty}({\mathbb{R}_{+})}}  \,d\tau   \\
		& \leq C\int_{0}^{t} \left|\psi (\tau,0)\right|^{\frac{4}{3}}  \,d\tau + C\int_{0}^{t} \lVert \psi_x \rVert^2_{L^2(\mathbb{R}_+)}\lVert \psi_{xx} \rVert^2_{L^2(\mathbb{R}_+)} \,d\tau   \\
		& \leq Ce^{-C\delta \beta} + C\varepsilon^2 \int_{0}^{t} D_{u_2} \,d\tau.
		\end{aligned}
	\end{align*}
	Similarly, we have
	\begin{align*}
	\left|	\int_{0}^{t}  (\bar{u}_x\phi \psi )(\tau,0) \,d\tau, \int_{0}^{t}  (\bar{\rho}_x\psi^2)(\tau,0) \,d\tau, \int_{0}^{t}  (\bar{\rho}_x \psi \dot{\mathbf{X}})(\tau,0) \,d\tau \right| \leq C \int_0^t |\psi(\tau,0)|d\tau \leq Ce^{-C\delta \beta}.
	\end{align*}
	Therefore, we obtain the followings:\\
	$\bullet$ \textit{Case I) Outflow problem}
	\begin{align*} 
		 \left|\int_0^t I_6\,d\tau\right| = \left|\int_{0}^{t} (\psi \phi_t)(\tau,0) \,d\tau \right| \leq \nu |u_-|\int_{0}^{t} \left|\phi_x(\tau,0)\right|^2 \,d\tau + C\varepsilon^2 \int_{0}^{t} D_{u_2} \,d\tau + Ce^{-C\delta \beta},
	\end{align*}
	$\bullet$ \textit{Case II) Impermeable wall problem}
	\begin{align*} 
		 \left|\int_0^t I_6\,d\tau\right| = \left|\int_{0}^{t} (\psi \phi_t)(\tau,0) \,d\tau \right| \leq C\varepsilon^2 \int_{0}^{t} D_{u_2}\,d\tau + Ce^{-C\delta \beta}.
	\end{align*}
	Substituting all the above estimates into \eqref{eq:sum1} and using the smallness of $\delta$ and $\varepsilon$, with $\nu>0$ chosen sufficiently small, we obtain
	\begin{align*}
		\begin{aligned}
		&\lVert \phi_x \rVert^2_{L^2(\mathbb{R}_+)} + \int_{0}^{t} \lVert \phi_x \rVert^2_{L^2(\mathbb{R}_+)}  \,d\tau + |u_-|\int_{0}^{t} \left|\phi_x(\tau,0)\right|^2 \,d\tau    \\
		&\qquad \leq C(\lVert \phi_{0x} \rVert^2_{L^2(\mathbb{R}_+)} + \lVert \psi_0 \rVert^2_{L^2(\mathbb{R}_+)} ) + C(\lVert \psi \rVert^2_{L^2(\mathbb{R}_+)}  + D_{u_1} + D_{\th_1} + \mathbf{G}^S) + C \varepsilon D_{u_2}+ Ce^{-C\delta \beta} .
		\end{aligned}
	\end{align*}
\end{proof}

	\subsection{$H^1$-estimates for $u - \util$}
	\begin{lemma}\label{lem:ux}
		Under the hypothesis of Proposition \ref{prop:main} (or Proposition \ref{prop:main2}), there exists a positive constant $C$ such that
		\begin{align*}
			\begin{aligned} 
				&\sup_{t \in [0,T]}\lVert \psi_x \rVert^2_{L^2(\mathbb{R}_+)}+\int_{0}^{T} \lVert \psi_{xx} \rVert^2_{L^2(\mathbb{R}_+)}\,dt \\
				&\qquad \leq C \lVert \psi_{0x} \rVert^2_{L^2(\mathbb{R}_+)} + C\int_{0}^{T} (D_{v_1}  + D_{u_1} + D_{\theta_1} + \mathbf{G}^S)\,dt + C\delta \int_{0}^{T}|\dot{\mathbf{X}}|^2 \,dt
				+ Ce^{-C\delta\beta},
			\end{aligned}
		\end{align*}
		where $D_{v_1}, D_{u_2}, D_{\theta_1}$ and $\mathbf{G}^S$ are the terms defined in \eqref{good_terms}.
	\end{lemma}
	\begin{proof}
	Multiplying $\eqref{eq:pertubation2}_2$ by $-\frac{\psi_{xx}}{\rho}$ and then integrating over $[0,t]\times \mathbb{R}_+$, we have 
	
\begin{align}\label{eq:pertubintu}
	\begin{aligned}
		\frac{1}{2} \int_{\mathbb{R}_+} \psi^2_x  \,dx + \int_{0}^{t}\int_{\mathbb{R}_+}\frac{\mu}{\rho} \psi^2_{xx} \,dx\,d\tau  &= 	\frac{1}{2} \int_{\mathbb{R}_+} \psi^2_{0x}  \,dx  +  \int_{0}^{t}\int_{\mathbb{R}_+} \left(\rho u \psi_x + R \theta \psi_x  -g\right)  \frac{\psi_{xx}}{\rho}dx\, d\tau   \\
	&\quad  -  \int_{0}^{t}\int_{\mathbb{R}_+}\bar{u}_x\psi_{xx}\dot{\mathbf{X}} \,dx\,d\tau - \int_{0}^{t} \psi_x \psi_t(\tau,0) \,d\tau.
	\end{aligned}
\end{align}
		We estimate the terms on the right-hand side of \eqref{eq:pertubintu}.
		Using Cauchy-Schwarz inequality and Young's inequality, we find that
		\begin{align*}
			\begin{aligned}
				&	\int_{0}^{t}\int_{\mathbb{R}_+} u\psi_x \psi_{xx} \,dx\,d\tau  \leq \nu \int_{0}^{t} \lVert \psi_{xx} \rVert_{L^2(\mathbb{R}_+)} \,d\tau+ C_{\nu} \int_{0}^{t} \lVert \psi_x \rVert_{L^2(\mathbb{R}_+)}  \,d\tau  , \\
				& 	\int_{0}^{t}\int_{\mathbb{R}_+}  R\theta \phi_x \frac{\psi_{xx}}{\rho} \,dx\,d\tau \leq \nu  \int_{0}^{t} \lVert \psi_{xx} \rVert_{L^2(\mathbb{R}_+)} + C_{\nu} \int_{0}^{t} \lVert \phi_x \rVert_{L^2(\mathbb{R}_+)} \,d\tau.
			\end{aligned}
		\end{align*}
		Here, $\nu>0$ is a constant to be chosen sufficiently small later, and $C_\nu$ denotes a constant depending on $\nu$.
		By virtue of \eqref{est:g}, Young's inequality and Cauchy-Schwarz inequality, we get 
		\begin{align*}
			\begin{aligned}
			& 	\left|	\int_{0}^{t}\int_{\mathbb{R}_+} g \frac{\psi_{xx}}{\rho} \,dx\,d\tau \right|  \\
			&  \leq 	C\int_{0}^{t}\int_{\mathbb{R}_+} |\bar{u}_x| |\psi_{xx}|^2 \,dx\,d\tau + C	\int_{0}^{t}\int_{\mathbb{R}_+} |\bar{u}_x| \left|(\phi, \psi, \vartheta)\right|^2\,dx\,d\tau   \\
			&\quad  + C\int_{0}^{t} \|\big(|\phi| + |\rhotil|\big)\|_{L^\infty(\Rp)}\lVert \vartheta_x \rVert_{L^2(\mathbb{R}_+)} \lVert \psi_{xx} \rVert_{L^2(\mathbb{R}_+)} d\tau + C\delta \bigg[ 	\int_{0}^{t}\int_{\mathbb{R}_+} |\bar{u}_x| \big(|\psi_{xx}|^2 + |\phi|^2\big) dx d\tau \bigg]  \\
			& \leq C(\delta + \varepsilon + \nu) \int_{0}^{t}  \lVert \psi_{xx} \rVert^2_{L^2(\mathbb{R}_+)} \,d\tau + C_\nu \int_{0}^{t}  \lVert \vartheta_{x} \rVert^2_{L^2(\mathbb{R}_+)} \,d\tau + C\mathbf{G}^S.
			\end{aligned}
		\end{align*}
		Again, $\nu>0$ denotes a small constant (to be chosen  later). Next, we use Young's inequality again to have
		\begin{align}\label{est:psix}
			\begin{aligned}
			\left| \int_{0}^{t}\int_{\mathbb{R}_+}\bar{u}_x\psi_{xx}\dot{\mathbf{X}} \,dx\,d\tau \right| & \leq C \int_{0}^{t}\int_{\mathbb{R}_+}  |\bar{u}_x| |\dot{\mathbf{X}}|^2 \,dx\,d\tau +  \int_{0}^{t}\int_{\mathbb{R}_+} |\bar{u}_x|   |\psi_{xx}|^2 \,dx\,d\tau   \\
			& \leq C\delta  \int_{0}^{t} |\dot{\mathbf{X}}|^2 \,d\tau + C\delta \int_{0}^{t}  \lVert \psi_{xx} \rVert^2_{L^2(\mathbb{R}_+)}  \,d\tau .
			\end{aligned}
		\end{align}
		Finally, interpolation inequality, Young's inequality, and Lemma \ref{lem:viscous_shock} yield that
		\begin{align}\label{est:boundaryp}
			\begin{aligned}
			 	\left|\int_{0}^{t} \psi_x \psi_t(\tau,0) \,d\tau\right| & \leq \int_{0}^{t} |\psi_t(\tau,0)| \lVert \psi_{x} \rVert_{L^{\infty}(\mathbb{R}_+)}  \,d\tau   \\
				& \leq C \int_{0}^{t} |\psi_t(\tau,0)| ^{\frac{4}{3}} \,d\tau + C\int_{0}^{t} \lVert \psi_{x} \rVert^2_{L^{2}(\mathbb{R}_+)} \lVert \psi_{xx} \rVert^2_{L^{2}(\mathbb{R}_+)}   \,d\tau  \\
				& \leq Ce^{-C\delta \beta} + C\varepsilon^{2}\int_{0}^{t} \lVert \psi_{xx} \rVert^2_{L^{2}(\mathbb{R}_+)}   \,d\tau .
			\end{aligned}
		\end{align}
		Here, for the last inequality in \eqref{est:boundaryp}, we use the fact that 
		\begin{align*}
			\int_0^t |\psi_t(\tau,0)|^{\frac{4}{3}}\,d\tau &= \int_0^t |\util_t(\tau,0)|^{\frac{4}{3}}\,d\tau \leq  C\int_0^t |\util'(-\sigma \tau - \mathbf{X}(\tau) - \beta)|^{\frac{4}{3}}|\sigma + \dot{\mathbf{X}}(\tau)|^{\frac{4}{3}}\,d\tau  \\
			&\leq C\int_0^t |\util'(-\sigma \tau - \mathbf{X}(\tau) - \beta)|^{\frac{4}{3}}\,d\tau \leq Ce^{-C\delta \beta}.
		\end{align*}
		
		Substituting the above estimates into \eqref{eq:pertubintu} and using the smallness of $\delta$ and $\varepsilon$, with $\nu>0$ chosen sufficiently small, we obtain the desired result.		
	\end{proof}

\subsection{$H^1$-estimates for $\th - \thtil$}
\begin{lemma}\label{lem:thetax}
		Under the hypothesis of Proposition \ref{prop:main} (or Proposition \ref{prop:main2}), there exists a positive constant $C$ such that
\begin{align*}
	\begin{aligned}
		& \sup_{t \in[0,T]}\lVert \vartheta_{x} \rVert^2_{L^{2}(\mathbb{R}_+)}  + \int_{0}^{T}  \lVert \vartheta_{xx} \rVert^2_{L^{2}(\mathbb{R}_+)}   \,dt  \\
		&\quad \leq C \lVert \vartheta_{0x} \rVert^2_{L^{2}(\mathbb{R}_+)}  +C\int_0^T \big(\mathbf{G}^S+ D_{u_1} + D_{\th_1} + \delta|\dot{\mathbf{X}}(\tau)|^2\big) \,dt + C\e \int_0^T D_{u_2}\, dt +  Ce^{-C\delta \beta},
	\end{aligned}
\end{align*}
where $D_{u_1}, D_{u_2}, D_{\th_1}$ and $\mathbf{G}^S$ are the terms defined in \eqref{good_terms}.
\end{lemma}

\begin{proof}
Multiplying \eqref{eq:pertubation2}$_3$ by $ -\frac{1}{\rho} \vartheta_{xx}$ and integrating over  $[0,t]\times \mathbb{R}_+$ yield
 \begin{align}\label{est:thetax}
 	\begin{aligned}
 		& \frac{R}{2(\gamma-1)}  	\lVert \vartheta_{x} \rVert^2_{L^{2}(\mathbb{R}_+)}  + \int_{0}^{t}\int_{\mathbb{R}_+} \frac{\kappa}{\rho} |\vartheta_{xx}|^2 \,dx\,d\tau = \frac{R}{2(\gamma-1)}  	\lVert \vartheta_{0x} \rVert^2_{L^{2}(\mathbb{R}_+)}   \\
 		& - \int_{0}^{t}\int_{\mathbb{R}_+} \frac{h}{\rho} \vartheta_{xx} \,dx\,d\tau + \frac{R}{\gamma-1}\int_{0}^{t}\int_{\mathbb{R}_+} \bar{\theta}_x \vartheta_{xx}  \dot{\mathbf{X}}(\tau) \,dx\,d\tau - \frac{R}{\gamma-1} \int_{0}^{t} (\vartheta_t \vartheta_x)(\tau, 0) \,d\tau.
 	\end{aligned}
 \end{align}
First, there exists $c>0$ such that
\begin{align*}
	\left|\int_{0}^{t}\int_{\mathbb{R}_+} \frac{\kappa}{\rho} |\vartheta_{xx}|^2 \,dx\,d\tau\right| \geq c \int_{0}^{t} 	\lVert \vartheta_{xx} \rVert^2_{L^{2}(\mathbb{R}_+)} \,d\tau.
\end{align*}

Next, by virtue of \eqref{est:h} and Young's inequality, we obtain
\begin{align*}
	\begin{aligned}
	&	\left| \int_{0}^{t}\int_{\mathbb{R}_+} \frac{h}{\rho} \vartheta_{xx} \,dx\,d\tau \right|  \\
		&  \leq C\int_{0}^{t}\int_{\mathbb{R}_+} |\bar{u}_x| |(\phi, \psi, \vartheta)|^2 \,dx\,d\tau  +  C\int_{0}^{t}\int_{\mathbb{R}_+}  |\bar{u}_x| |\vartheta_{xx}|^2 \,dx\,d\tau + \nu  \int_{0}^{t}\int_{\mathbb{R}_+} |\vartheta_{xx}|^2 \,dx\,d\tau  \\
		&\quad + C_{\nu} \int_{0}^{t}\int_{\mathbb{R}_+} (|\psi_{x}|^2 +|\psi_x|^4 + |\vartheta_{x}|^2 )\,dx\,d\tau \\
		& \leq C\mathbf{G}^S + C(\delta + \nu) \int_{0}^{t} 	\lVert \vartheta_{xx} \rVert^2_{L^{2}(\mathbb{R}_+)}  \,d\tau + C_{\nu} \int_{0}^{t} (	\lVert \psi_{x} \rVert^2_{L^{2}(\mathbb{R}_+)} + \|\psi_x\|_{L^4(\Rp)}^4 + \lVert \vartheta_{x} \rVert^2_{L^{2}(\mathbb{R}_+)}) \,d\tau .
	\end{aligned}
\end{align*}
Using the interpolation inequality and \eqref{apriori_small}, we have 
\begin{align*}
	\|\psi_x\|_{L^4(\Rp)}^4 &\leq \|\psi_x\|_{L^2(\Rp)}^3 \|\psi_{xx}\|_{L^2(\Rp)} \leq C\e^2\|\psi_x\|_{L^2(\Rp)} \|\psi_{xx}\|_{L^2(\Rp)}\\
	&\leq C\e^2\big(\|\psi_x\|_{L^2(\Rp)}^2+ \|\psi_{xx}\|_{L^2(\Rp)}^2\big).
\end{align*}
Thus, we have 
\begin{align*}
	\left| \int_{0}^{t}\int_{\mathbb{R}_+} \frac{h}{\rho} \vartheta_{xx} \,dx\,d\tau \right| &\leq C\mathbf{G}^S +C(\delta + \nu)\int_{0}^{t} \lVert \vartheta_{xx} \rVert^2_{L^{2}(\mathbb{R}_+)}  \,d\tau + C\e^2 \int_0^t \|\psi_{xx}\|_{L^2(\Rp)}^2\,d\tau \\
	&\quad + C \int_{0}^{t} (	\lVert \psi_{x} \rVert^2_{L^{2}(\mathbb{R}_+)} + \lVert \vartheta_{x} \rVert^2_{L^{2}(\mathbb{R}_+)}) \,d\tau.
\end{align*}
Similar to \eqref{est:psix}, we have 
\begin{align*}
	\left| \int_{0}^{t}\int_{\mathbb{R}_+} \bar{\theta}_x \vartheta_{xx}  \dot{\mathbf{X}}(\tau) \,dx\,d\tau \right| \leq C\delta  \int_{0}^{t} |\dot{\mathbf{X}}|^2 \,d\tau + C\delta \int_{0}^{t}  \lVert \vartheta_{xx} \rVert^2_{L^2(\mathbb{R}_+)}  \,d\tau .
\end{align*}
For the boundary term, using a similar method to \eqref{est:boundaryp}, we have 
\begin{align*} 
	\begin{aligned}
		& 	\left|\int_{0}^{t} \vartheta_x \vartheta_t(\tau,0) \,d\tau\right|   \leq Ce^{-C\delta \beta} + C\varepsilon^{2}\int_{0}^{t} \lVert \vartheta_{xx} \rVert^2_{L^{2}(\mathbb{R}_+)}   \,d\tau .
	\end{aligned}
\end{align*}
Plugging the above estimates into \eqref{est:thetax} and using the smallness of $\delta, \nu$ and $\varepsilon$, we obain the desired result.
\end{proof}
\vspace{0.3cm}
\noindent $\bullet$ Proof of Proposition \ref{prop:main} and Proposition \ref{prop:main2}

Combining Lemma \ref{lem:main}, Lemma \ref{lem:rhox}, Lemma \ref{lem:ux}, and Lemma \ref{lem:thetax}, we obtain the estimate \eqref{est:apriori} for both IBVPs \eqref{outflow} and \eqref{imperable}.

In addition, \eqref{bddx12} follows directly from the definition of the shift in \eqref{def_shift}:
\[
|\dot{\mathbf{X}}| \leq \frac{C}{\delta}\|(\rho -\rhotil, u-\util, \th - \thtil)\|_{L^\infty{(\Rp)}}\intRp |(\rhotil_x,\util_x,\thtil_x)|\,dx \leq C\e.
\]
Therefore, Proposition \ref{prop:main} and Proposition \ref{prop:main2} hold.

\qed
\begin{appendix}
	\section{Relative entropy}\label{Appendix A}
	\setcounter{equation}{0}
Let $U = (\r, m, E)$ with $m = \r u$, $E = \r\left(e + \frac{u^2}{2}\right)$ and $e = \frac{R}{\gamma - 1}\theta + const$. In this section, we compute the relative entropy weighted by $\thtil$. First, recalling from the Gibbs relation $\theta d s = d e + p d \left(\frac{1}{\r}\right)$, the entropy $s(U)$ takes the form:
\[
s(U) = -R \ln \r + \frac{R}{\gamma - 1} \ln \theta.
\] 
Using the Gibbs relation and $E = \r \left( e + \frac{u^2}{2} \right)$, we have
\begin{align*}
	\th \,ds &= \left(-\frac{p}{\r^2} - \frac{E}{\r^2} + \frac{m^2}{\r^3} \right)\,d\rho  - \frac{m}{\r^2}\,dm + \frac{1}{\r}\,dE\\
	&= \left( -R\frac{\th}{\r}- \frac{R}{\gamma - 1}\frac{\th}{\r} + \frac{u^2}{2\r} \right)d\rho - \frac{u}{\r}dm + \frac{1}{\r}\,dE.
\end{align*}
This implies that
\[
\nabla_U s(U) = \left(-\frac{R}{\r} - \frac{R}{\gamma - 1}\frac{1}{\r} + \frac{u^2}{2\r \th}, -\frac{u}{\r \th}, \frac{1}{\r \th} \right).
\]
Then for any $\Util = (\rhotil, \bar{m}, \bar{E})$ with $\bar{m} = \rhotil \util$, $\bar{E} = \r\left(\bar{e} + \frac{\util^2}{2}\right)$ and $\bar{e} = \frac{R}{\gamma - 1}\thtil + const$, we have
\begin{equation*}
\begin{aligned}
	\bar{\th}\,d(-\r s)(\bar{U})
	&= -\thtil s(\bar{U})\,d\r + \thtil \rhotil \,d(-s)(\bar{U})\\
	&=\Big( -\thtil s(\bar{U}) + R\thtil + \frac{R}{\gamma - 1}\thtil -\frac{\util^2}{2}\Big)\,d\r +\util \,dm - \,dE.
\end{aligned}
\end{equation*}
Now, we introduce the mathematical entropy $\eta = -\r s$. Then we obtain
\begin{equation}\label{grad-ent}
	\thtil \nabla_U \eta(\bar{U}) = \Big( -\thtil s(\bar{U}) + R\thtil + \frac{R}{\gamma - 1}\thtil -\frac{\util^2}{2},\util, -1\Big).
\end{equation}
Using \eqref{grad-ent}, we have
\begin{align*}
	\thtil \eta(U|\bar{U}) &= \thtil \Big(\eta(U) - \eta(\bar{U}) -  \nabla_U \eta(\bar{U})\cdot(U-\bar{U}) \Big)\\
	&=\thtil \left(R\r \ln \r - \frac{R}{\gamma -1}\r \ln \theta - R \rhotil \ln \rhotil + \frac{R}{\gamma - 1}\rhotil \ln \thtil \right)\\
	&\quad + \left(\thtil s(\bar{U}) - R\thtil - \frac{R}{\gamma  - 1}\thtil + \frac{\util^2}{2}\right)(\r - \rhotil) - \util (m - \bar{m}) + \big(E -\bar{E})\\
	&=\rho \left[R\bar{\theta} \left( \frac{\bar{\rho}}{\rho }-\ln \frac{\bar{\rho}}{\rho}  -1 \right) + \frac{R}{\gamma -1} \bar{\theta} \left(\frac{\theta}{\bar{\theta}} -\ln \frac{\theta}{\bar{\theta}} -1\right) + \frac{1}{2}(u-\bar{u})^2\right].
\end{align*}
\qed
	\section{Sharp estimate for the diffusion}\label{appjacobian}
	We here present the approximation of Jacobian $\frac{dy}{dx}$ to estimate the diffusion terms $D_{u_1}$ and $D_{\th_1}$ in Lemma \ref{lma:leading}. To this end, we first integrate the system \eqref{eq:Vshock} over $(\pm \infty,\xi]$, which leads to the following system:
\begin{equation}\label{eq:vshock2}
	\begin{aligned}
		&-\sigma (\bar{\rho} -\rho_\pm) + (\bar{\rho} \bar{u}- \rho_\pm u_\pm) =0, \\
		&-\sigma (\bar{\rho}\bar{u} - \rho_\pm u_\pm) + (\bar{\rho}\bar{u}^2 + \bar{p} - \rho_\pm u_\pm^2-p_\pm) =\mu \bar{u}', \\
		&-\sigma (\bar{E}- E_\pm) + (\bar{E}\bar{u} +\bar{p}\bar{u} - E_\pm u_\pm - p_\pm u_\pm)=\kappa \bar{\theta}' + \mu \bar{u}\bar{u}' .
	\end{aligned}
\end{equation}
From this system, we can derive the following autonomous system of ODEs:
\begin{equation}\label{eq:vshock1}
	\begin{aligned}
		& \mu \bar{u}' = \rho_\pm (u_\pm - \sigma)(\bar{u}- u_\pm) + \bar{p} -p_\pm,\\ 
		& \kappa \bar{\theta}' = \rho_\pm(u_\pm-\sigma)\Big[\bar{e}- e_\pm -\frac{1}{2}(\bar{u} -u_\pm)^2\Big] + p_\pm(\bar{u}-u_\pm).
	\end{aligned}
\end{equation}
	\setcounter{equation}{0}
	\begin{lemma}\label{lem:Jac}
		Let $y$ defined as in \eqref{ydef} and $\delta$ be the shock strength defined in Lemma \ref{lem:viscous_shock}. Then it holds that 
		\begin{equation*}
		\left|\mu\frac{1}{y(1-y)}\frac{dy}{dx}-\frac{\gamma+1}{2}\rho_-\frac{\mu R \gamma}{\mu R \gamma+\kappa(\gamma-1)^2} \delta\right|	\leq C\delta^2.
		\end{equation*}
	\end{lemma}
	\begin{proof}
It follows from the definition of $\delta$ and $y$ that
\begin{align}
	\begin{aligned} \label{mubaru'}
		\mu \frac{1}{y(1-y)}\frac{dy}{dx}&= -\frac{\mu}{\delta} \left(\frac{\bar{u}'}{y} + \frac{\bar{u}'}{1-y}\right) =\mu \left(\frac{\bar{u}'}{\bar{u}-u_-}-\frac{\bar{u}'}{\bar{u}-u_+}\right).
	\end{aligned}
\end{align}
	Then substituting \eqref{eq:vshock1}$_1$ into \eqref{mubaru'}, we have 
	\begin{align}
		\begin{aligned} \label{firstexj}
			\mu \left(\frac{\bar{u}'}{\bar{u}-u_-}-\frac{\bar{u}'}{\bar{u}-u_+}\right)&=(u_--\sigma)\rho_-+\frac{\bar{p}-p_-}{\bar{u}-u_-}-(u_+-\sigma)\rho_+-\frac{\bar{p}-p_+}{\bar{u}-u_+}\\
			&=\frac{\bar{p}-p_-}{\bar{u}-u_-}-\frac{\bar{p}-p_+}{\bar{u}-u_+}.
		\end{aligned}    
	\end{align}
	Using 
	\begin{align*}
		\bar{p}-p_\pm=&R\bar{\rho}\bar{\theta}-R\rho_\pm\theta_\pm    
		=R\bar{\rho}(\bar{\theta}-\theta_\pm)+R\theta_\pm (\bar{\rho}-\rho_\pm ),
	\end{align*}
	we have
	\begin{align}
		\begin{aligned} \label{partIII}
			\frac{\bar{p}-p_-}{\bar{u}-u_-}-\frac{\bar{p}-p_+}{\bar{u}-u_+}&=\left(R\bar{\rho}\frac{\bar{\theta}-\theta_-}{\bar{u}-u_-}+R\theta_-\frac{\bar{\rho}-\rho_-}{\bar{u}-u_-}\right)-\left(R\bar{\rho}\frac{\bar{\theta}-\theta_+}{\bar{u}-u_+}+R\theta_+\frac{\bar{\rho}-\rho_+}{\bar{u}-u_+}\right)\\
			&=R\bar{\rho}\left(\frac{\bar{\theta}-\theta_-}{\bar{u}-u_-}-\frac{\bar{\theta}-\theta_+}{\bar{u}-u_+}\right)+\left(R\theta_-\frac{\bar{\rho}-\rho_-}{\bar{u}-u_+}-R\theta_+\frac{\bar{\rho}-\rho_+}{\bar{u}-u_+}\right)\\
			&=:I+ II.
		\end{aligned}
	\end{align}
First, we estimate the second part $II$. By virtue of $\eqref{eq:vshock2}_1$, we have
	\begin{align*}
		 \bar{\rho}-\rho_-=\frac{\rho_-}{\sigma-\bar{u}}(\bar{u}-u_-).
	\end{align*}
	Hence, we obtain
	\begin{align*}
		II=&R\theta_-\frac{\rho_-}{\sigma-\bar{u}}-R\theta_+\frac{\rho_+}{\sigma-\bar{u}}  =\frac{p_--p_+}{\sigma-\bar{u}} \\
		= & \rho_-(u_--u_+)\frac{\sigma-u_-}{\sigma-\bar{u}} = \rho_-\delta \left(1+ O(\delta)\right)= \rho_-\delta + O(\delta^2).
	\end{align*}
Here, we used the fact:
	\begin{equation}\label{eq:p-p}
		\r_- (u_- - \sigma) = \r_+(u_+ - \sigma), \quad \text{and} \quad p_--p_+=(\sigma-u_-)\rho_-(u_--u_+),
	\end{equation}
which can be derived from  \eqref{eq:RH}.
	
For the first part $I$, it holds from $\eqref{eq:vshock1}_2$ that
	\begin{align*}
		\bar{\theta}-\theta_\pm=&\frac{\gamma-1}{R}\left[\frac{1}{2}(\bar{u}-u_\pm)^2+\frac{\kappa\bar{\theta}'}{\rho_\pm(u_\pm-\sigma)}-\frac{p_\pm(\bar{u}-u_\pm)}{\rho_\pm(u_\pm-\sigma)}\right].
	\end{align*}
	Thus, we have
	\begin{align*}
	&\frac{\bar{\theta}-\theta_-}{\bar{u}-u_-}-\frac{\bar{\theta}-\theta_+}{\bar{u}-u_+}\\
	&\quad =\frac{\gamma-1}{R}\left[\frac{1}{2}(\bar{u}-u_-)-\frac{p_-}{\rho_-(u_--\sigma)}+\frac{\kappa \bar{\theta}'}{\rho_-(u_--\sigma)(\bar{u}-u_-)}\right]\\
	&\qquad -\frac{\gamma-1}{R}\left[\frac{1}{2}(\bar{u}-u_+)-\frac{p_+}{\rho_-(u_--\sigma)}+\frac{\kappa \bar{\theta}'}{\rho_-(u_--\sigma)(\bar{u}-u_+)}\right]\\
	&\quad =-\frac{\gamma-1}{R}\left[\frac{1}{2}(u_--u_+)+\frac{p_--p_+}{\rho_-(u_--\sigma)}\right] +\underbrace{\frac{\gamma-1}{R}\frac{\kappa\bar{\theta}'}{\rho_-(u_--\sigma)}\left[\frac{1}{\bar{u}-u_-}-\frac{1}{\bar{u}-u_+}\right]}_{=:\mathcal{A}}.
	\end{align*}
	To estimate $\mathcal{A}$, using \eqref{estderi} and \eqref{est:sigma}, we have
	\begin{equation*}
		\bar{\theta}'=\left(\frac{\gamma-1}{\sigma-u_-}\theta_-+O(\delta)\right)\bar{u}'.    
	\end{equation*}
Substituting this into $\mathcal{A}$, and by \eqref{eq:p-p} and \eqref{mubaru'}, we have
\begin{align*}
	&\frac{\bar{\theta}-\theta_-}{\bar{u}-u_-}-\frac{\bar{\theta}-\theta_+}{\bar{u}-u_+} \\
	&=\frac{\gamma-1}{R}\left(-\frac{1}{2}\delta+\delta\right)+\frac{\gamma-1}{R}\frac{\kappa}{\rho_-(u_--\sigma)}\left(\frac{\gamma-1}{\sigma-u_-}\theta_-+O(\delta)\right)\bar{u}'\left[\frac{1}{\bar{u}-u_-}-\frac{1}{\bar{u}-u_+}\right]\\
	&=\frac{\gamma-1}{2R}\delta-\left(\frac{\kappa(\gamma-1)^2}{R\rho_-(\sigma-u_-)^2}\theta_-+O(\delta)\right)\bar{u}'\left[\frac{1}{\bar{u}-u_-}-\frac{1}{\bar{u}-u_+}\right]\\
	&=\frac{\gamma-1}{2R}\delta-\left(\frac{\kappa(\gamma-1)^2}{R\rho_-(\sigma-u_-)^2}\theta_-+O(\delta)\right)\frac{1}{y(1-y)}\frac{dy}{dx}.
\end{align*}
Therefore, combining \eqref{mubaru'}, \eqref{firstexj}, and \eqref{partIII}, we conclude
\begin{align*}
	\mu\frac{1}{y(1-y)}\frac{dy}{dx} &= \rho_-\delta+O(\delta^2) +R(\rho_-+O(\delta))\left[\frac{\gamma-1}{2R}\delta-\left(\frac{\kappa(\gamma-1)^2}{R\rho_-(\sigma-u_-)^2}+O(\delta)\right)\frac{1}{y(1-y)}\frac{dy}{dx}\right]\\
	&=\rho_-\delta+O(\delta^2)+\frac{(\gamma-1)\rho_-}{2}\delta-\left(\frac{\kappa(\gamma-1)^2}{(\sigma-u_-)^2}+O(\delta)\right)\frac{1}{y(1-y)}\frac{dy}{dx}\\
	&=\frac{\gamma+1}{2}\rho_-\delta-\left(\frac{\kappa(\gamma-1)^2}{(\sigma-u_-)^2}\theta_-+O(\delta)\right)\frac{1}{y(1-y)}\frac{dy}{dx}+O(\delta^2).
\end{align*}
Hence, using \eqref{est:sigma}, we obtain the desired result:
\begin{equation*}
	\mu\frac{1}{y(1-y)}\frac{dy}{dx}=\frac{\gamma+1}{2}\rho_-\frac{\mu R \gamma}{\mu R \gamma+\kappa(\gamma-1)^2}\delta+O(\delta^2).
\end{equation*}
\end{proof}
\end{appendix}

 \bibliographystyle{acm}
	\bibliography{ref}

\end{document}